\documentclass[11pt,reqno]{amsart}
\usepackage{amsmath,amsfonts,amssymb}
\usepackage{mathrsfs}

\def\wh{\widehat}
\def\wt{\widetilde}
\def\R{\mathbb R}
\def\C{\mathbb C}

\def\O{\mathcal O}

\def\x{\mathbf x}
\def\u{\mathbf u}

\def\p{\mathbf p}

\def\w{\mathbf w}
\def\z{\mathbf z}

\def\VV{\mathbf V}

\def\D{\mathbf D}
\def\FF{\mathbf F}
\def\1{\mathbf 1}

\def\O{\mathcal O}

\def\bxi{\boldsymbol \xi}

\def\eeta{\boldsymbol \eta}

\def\bPhi{\boldsymbol \Phi}

\def\1{\bold 1}

\def\eps{\varepsilon}

\def\leq{\leqslant}
\def\le{\leqslant}

\def\ge{\geqslant}

\usepackage{amsthm}
\theoremstyle{theorem}
\newtheorem{theorem}{Theorem}[section]
\newtheorem{proposition}[theorem]{Proposition}
\newtheorem{lemma}[theorem]{Lemma}
\newtheorem{condition}[theorem]{Condition}
\newtheorem{remark}[theorem]{Remark}
\newtheorem{corollary}[theorem]{Corollary}

\numberwithin{equation}{section}

\theoremstyle{plain}
\newtoks\thehProclaim
\newtheorem*{Proclaim}{\the\thehProclaim}

\usepackage[paperwidth=20.5cm,paperheight=29.7cm,width=16cm,height=24cm,top=24mm,left=19mm,headsep=11mm]{geometry}

\begin{document}

\title[Homogenization of the Neumann problem]{Homogenization of the Neumann problem for~higher-order elliptic equations \\with periodic coefficients}

\author{T.~A.~Suslina}

\thanks{Supported by Russian Science Foundation (grant no.~17-11-01069).}


 \address{St.~Petersburg State University, 7/9 Universitetskaya nab., St.~Petersburg, 199034, Russia}

\email{t.suslina@spbu.ru}


\begin{abstract}
Let $\mathcal{O}\subset\mathbb{R}^d$ be a bounded domain of class $C^{2p}$.
In $L_2(\mathcal{O};\mathbb{C}^n)$, we study a selfadjoint strongly elliptic operator $A_{N,\varepsilon}$ of order $2p$ given
by the expression $b(\D)^* g(\x/\eps) b(\D)$, $\eps >0$, with the Neumann boundary conditions.
Here $g(\x)$ is a bounded and positive definite $(m\times m)$-matrix-valued function in $\R^d$, periodic with respect to some lattice;
$b(\D)=\sum_{|\alpha|=p} b_\alpha \D^\alpha$ is a differential operator of order $p$ with constant coefficients; $b_\alpha$ are constant
$(m\times n)$-matrices. It is assumed that $m\geqslant n$ and that the symbol $b({\boldsymbol \xi})$ has maximal rank for any
$0 \ne {\boldsymbol \xi}\in \C^d$.
We find approximations for the resolvent $\left(A_{N,\varepsilon}-\zeta I \right)^{-1}$
in the  $L_2(\mathcal{O};\mathbb{C}^n)$-operator norm and in the norm of operators acting from $L_2(\mathcal{O};\mathbb{C}^n)$ to the Sobolev space $H^p(\mathcal{O};\mathbb{C}^n)$, with error estimates depending on $\varepsilon$ and $\zeta$.
\end{abstract}
\maketitle

\section*{Introduction}
\setcounter{section}{0}
\setcounter{equation}{0}

 A broad literature is devoted to homogenization problems for differential operators (DOs) with rapidly oscillating coefficients.
First of all, we mention the books \cite{BeLPa}, \cite{BaPan}, \cite{ZhKO}.

\subsection{Operator error estimates for homogenization problems in $\R^d$}
In a series of papers \cite{BSu0, BSu1, BSu2, BSu3} by Birman and Suslina an operator-theoretic approach
to homogenization problems was suggested and developed. This approach was applied to a wide class
of matrix selfadjoint strongly elliptic second order DOs acting in $L_2(\R^d;\C^n)$ and admitting a factorization of the form
\begin{equation}
\label{0.1}
{\mathcal A}_\eps = b(\D)^* g(\x/\eps) b(\D),\quad \eps >0.
\end{equation}
Here  an $(m\times m)$-matrix-valued function $g(\x)$ is bounded, uniformly positive definite, and periodic with respect to
some lattice $\Gamma \subset \R^d$. Next, $b(\D)$ is a first order DO of the form $b(\D)= \sum_{j=1}^d b_j D_j$, where $b_j$
are constant $(m\times n)$-matrices. It is assumed that $m\geqslant n$ and that the symbol $b(\bxi)$ has rank $n$ for any
$0 \ne \bxi \in \R^d$. The simplest example of the operator \eqref{0.1} is the acoustics operator $-{\rm div}\,g(\x/\eps) \nabla$;
the operator of elasticity theory also can be written in the required form. These and other examples are considered in \cite{BSu1}
in details.

In \cite{BSu0, BSu1}, it was shown that, as $\eps \to 0$, the resolvent $({\mathcal A}_\eps +I)^{-1}$ converges in the
$L_2(\R^d;\C^n)$-operator norm to the resolvent of the \textit{effective operator} ${\mathcal A}^0=b(\D)^* g^0 b(\D)$.
Here $g^0$ is a constant \textit{effective matrix}. It was proved that
\begin{equation}
\label{0.2}
\| ({\mathcal A}_\eps +I)^{-1} - ({\mathcal A}^0 +I)^{-1} \|_{L_2(\R^d) \to L_2(\R^d)} \leqslant C \eps.
\end{equation}
In \cite{BSu2}, a more accurate approximation for the resolvent  $({\mathcal A}_\eps +I)^{-1}$ in the $L_2(\R^d;\C^n)$-operator norm with error  $O(\eps^2)$ was found. In \cite{BSu3}, an approximation for the resolvent  $({\mathcal A}_\eps +I)^{-1}$ in the norm of operators acting from
$L_2(\R^d;\C^n)$ to the Sobolev space $H^1(\R^d;\C^n)$ was obtained. It was proved that
\begin{equation}
\label{0.3}
\| ({\mathcal A}_\eps +I)^{-1} - ({\mathcal A}^0 +I)^{-1} - \eps {\mathcal K}(\eps) \|_{L_2(\R^d) \to H^1(\R^d)} \leqslant C \eps.
\end{equation}
Here ${\mathcal K}(\eps)$ is the so-called \textit{corrector}. The operator ${\mathcal K}(\eps)$ contains rapidly oscillating
factors, and so depends on $\eps$; herewith,  $\|{\mathcal K}(\eps)\|_{L_2 \to H^1}= O(\eps^{-1})$.

Estimates of the form \eqref{0.2}, \eqref{0.3} are called \textit{operator error estimates}. They are order-sharp.
The method of \cite{BSu0, BSu1, BSu2, BSu3} is based on the scaling transformation, the Floquet-Bloch theory, and the analytic perturbation theory.

We also mention the recent paper \cite{Su}, where two-parametric analogs of estimates \eqref{0.2} and \eqref{0.3}
(depending on $\eps$ and $\zeta$) for the resolvent \hbox{$({\mathcal A}_\eps - \zeta I)^{-1}$} at arbitrary point $\zeta \in \C \setminus \R_+$
were obtained. In the presence of the lower order terms, similar results were obtained in~\cite{MSu1}.

A different approach to operator error estimates (\textit{the modified method of first-order approximation} or \textit{the shift method}) was
suggested by Zhikov.   In \cite{Zh} and \cite{ZhPas}, by this method, estimates of the form \eqref{0.2} and \eqref{0.3}
for the acoustics operator and the elasticity operator were obtained.
Concerning further results of Zhikov and Pastukhova, see the recent survey  \cite{ZhPas_survey} and references therein.

A homogenization problem for periodic elliptic DOs of  \textit{high even order} is of separate interest.
The operator-theoretic approach of Birman and Suslina was adapted for such operators in the paper
\cite{V} by Veniaminov and in the recent paper \cite{KuSu} by Kukushkin and Suslina.

In \cite{V}, the operator of the form ${\mathcal B}_\eps = (\D^p)^* g(\x/\eps) \D^p$ was studied.
Here $g(\x)$ is a symmetric positive definite and bounded tensor of order $2p$, periodic with respect to some lattice $\Gamma$.
Such an operator with $p=2$ arises in the theory of elastic plates (see \cite{ZhKO}).
The effective operator is given by  ${\mathcal B}^0 = (\D^p)^* g^0 \D^p$, where $g^0$ is the effective tensor.
In \cite{V}, the following analog of estimate \eqref{0.2} was proved:
\begin{equation*}
\| ({\mathcal B}_\eps +I)^{-1} - ({\mathcal B}^0 +I)^{-1} \|_{L_2(\R^d) \to L_2(\R^d)} \leqslant C \eps.
\end{equation*}

In \cite{KuSu}, a more general class of higher-order elliptic DOs, acting in $L_2(\R^d;\C^n)$ and admitting a factorization of the form
\begin{equation}
\label{0.5}
{A}_\eps = b(\D)^* g(\x/\eps) b(\D),
\end{equation}
was studied. Here $g(\x)$ is a bounded and uniformly positive definite $(m\times m)$-matrix-valued function, periodic with respect to the lattice $\Gamma$.
The operator $b(\D)$ of order $p \geqslant 2$ is of the form $b(\D)= \sum_{|\alpha|=p} b_\alpha \D^\alpha$, where $b_\alpha$
are constant $(m\times n)$-matrices. It is assumed that $m\geqslant n$ and that the symbol $b(\bxi)$ has rank $n$ for any
$0 \ne \bxi \in \R^d$. The main results of \cite{KuSu} are approximations of the resolvent
$({A}_\eps - \zeta I)^{-1}$, where $\zeta \in \C \setminus \R_+$, in various operator norms with two-parametric error estimates
(depending on $\eps$ and $\zeta$).
It is shown that the resolvent  $({A}_\eps - \zeta I)^{-1}$
converges in the $L_2(\R^d;\C^n)$-operator norm to the resolvent of the effective operator $A^0= b(\D)^* g^0 b(\D)$ (here $g^0$ is the
constant effective matrix), and
\begin{equation}
\label{0.6}
\| ({A}_\eps - \zeta I)^{-1} - ({A}^0 - \zeta I)^{-1} \|_{L_2(\R^d) \to L_2(\R^d)} \leqslant C_1(\zeta) \eps.
\end{equation}
Approximation for the resolvent in the  ,,energy'' norm (i.~e., in the norm of operators acting from $L_2(\R^d;\C^n)$ to $H^p(\R^d;\C^n)$),
with the corrector taken into account, is obtained:
\begin{equation}
\label{0.7}
\| ({A}_\eps - \zeta I)^{-1} - ({A}^0 -\zeta I)^{-1} - \eps^p {K}(\zeta;\eps)
\|_{L_2(\R^d) \to H^p(\R^d)} \leqslant C_2(\zeta) \eps.
\end{equation}
The corrector ${K}(\zeta;\eps)$ contains rapidly oscillating factors; herewith, $\|{K}(\zeta;\eps)\|_{L_2 \to H^p} = O(\eps^{-p})$.
The dependence of $C_1(\zeta)$ and $C_2(\zeta)$ on $\zeta$ is searched out.

Similar results on homogenization of higher-order elliptic operators in $\R^d$ have been independently
obtained in the recent papers \cite{Pas1,Pas2} by Pastukhova
with the help of the shift method (in these papers, estimates are one-parametric, it is assumed that $\zeta=-1$).

\subsection{Operator error estimates for homogenization problems in a bounded domain}
Operator error estimates have been also studied for second order elliptic operators
with rapidly oscillating coefficients in a bounded domain $\O \subset \R^d$ with sufficiently smooth boundary.
In \cite{Zh, ZhPas}, the acoustics operator and the operator of elasticity theory with the Dirichlet or Neumann boundary conditions
were studied; the analogs of estimates \eqref{0.2} and \eqref{0.3}, but with error estimates of order $O(\eps^{1/2})$, were obtained.
The error deteriorates because of the boundary influence. (In the case of the Dirichlet problem for the acoustics operator, the $(L_2 \to L_2)$-estimate
was improved in \cite{ZhPas}, but the order was not sharp.)

Similar results for the operator $-{\rm div}\,g(\x/\eps) \nabla$ in a bounded domain with the Dirichlet or Neumann conditions were obtained
in the papers \cite{Gr1,Gr2} by Griso with the help of the unfolding method. In \cite{Gr2}, the sharp order estimate of the form
\eqref{0.9} (see below) was obtained for the first time.

For the second order matrix operators ${\mathcal A}_{D,\eps}$ and ${\mathcal A}_{N,\eps}$ given by expression \eqref{0.1}
with the Dirichlet or Neumann conditions, respectively, the operator error estimates were obtained in the papers \cite{PSu, Su02,Su03}.
In \cite{PSu}, the Dirichlet problem was studied and the following estimate was proved:
\begin{equation}
\label{0.8}
\| {\mathcal A}_{D,\eps}^{-1} - ({\mathcal A}_D^0)^{-1} - \eps {\mathcal K}_D(\eps)
\|_{L_2(\O) \to H^1(\O)} \leqslant C \eps^{1/2}.
\end{equation}
Here ${\mathcal A}_D^0$ is the effective operator with the Dirichlet condition, and ${\mathcal K}_D(\eps)$ is the corresponding corrector.
In \cite{Su02}, a sharp-order estimate in the $L_2(\O;\C^n)$-operator norm was established:
\begin{equation}
\label{0.9}
\| {\mathcal A}_{D,\eps}^{-1} - ({\mathcal A}_D^0)^{-1}\|_{L_2(\O) \to L_2(\O)} \leqslant C \eps.
\end{equation}
Similar results for the Neumann problem were obtained in~\cite{Su03}. The method of \cite{PSu, Su02,Su03} was based on using the results
for the problem in $\R^d$, introduction of the boundary layer correction term, and estimation of this term in $H^1(\O;\C^n)$ and in $L_2(\O;\C^n)$. Some technical tricks were borrowed from \cite{ZhPas}.

In \cite{Su}, approximations for the resolvents $({\mathcal A}_{D,\eps} - \zeta I)^{-1}$
and $({\mathcal A}_{N,\eps} - \zeta I)^{-1}$ at an arbitrary point $\zeta \in \C \setminus \R_+$ (two-parametric analogs of estimates \eqref{0.8}
and \eqref{0.9}) were obtained. In the presence of the lower order terms, similar results were proved in~\cite{MSu2} in the case of the Dirichlet
condition.

Independently, estimate of the form \eqref{0.9} for uniformly elliptic second order systems  with the Dirichlet or Neumann conditions,
under some regularity assumptions on the coefficients, was obtained by a different method in the paper \cite{KeLiS} by Kenig, Lin, and Shen.

Finally, in the recent paper \cite{Su2017},  in a bounded domain $\O$ of class $C^{2p}$ the higher-order operator $A_{D,\eps}$
given in a factorized form \eqref{0.5} with the Dirichlet conditions was studied; approximations for the resolvent $(A_{D,\eps} - \zeta I)^{-1}$ at a regular point $\zeta$  with error estimates depending on $\eps$ and $\zeta$ were found.

\subsection{Main results}
In the present paper, we study the operator $A_{N,\eps}$ of order $2p$ in a bounded domain $\O$ of class $C^{2p}$.
This operator is given in a factorized form \eqref{0.5} under the Neumann conditions on the boundary $\partial \O$.
 \textit{Our goal} is to find approximations for the resolvent \hbox{$(A_{N,\eps} - \zeta I)^{-1}$} at a regular point $\zeta$
with error estimates depending on $\eps$ and $\zeta$.

Now we describe the main results. Let $\zeta = |\zeta| e^{i\varphi} \in \C \setminus \R_+$ and $|\zeta| \geqslant 1$.
It is proved that
\begin{equation}
\label{0.10}
\| ({A}_{N,\eps}-\zeta I)^{-1} - (A_N^0 - \zeta I)^{-1}\|_{L_2(\O) \to L_2(\O)} \leqslant {\mathcal C}_1(\varphi) \eps |\zeta|^{-1+1/2p},
\end{equation}
\begin{equation}
\begin{aligned}
\label{0.11}
\| ({A}_{N,\eps}-\zeta I)^{-1} - (A_N^0 - \zeta I)^{-1} - \eps^p K_N(\zeta;\eps)\|_{L_2(\O) \to H^p(\O)}
\leqslant {\mathcal C}_2(\varphi)\bigl( \eps^{1/2} |\zeta|^{-1/2+1/4p} + \eps^{p}\bigr),
\end{aligned}
\end{equation}
for $0< \eps \leqslant \eps_1$ (where $\eps_1$ is a sufficiently small number depending on the domain $\O$ and the lattice $\Gamma$).
Here $A_N^0$ is the effective operator given by the expression $b(\D)^* g^0 b(\D)$ with the Neumann conditions.
The corrector $K_N(\zeta;\eps)$ involves rapidly oscillating factors, herewith,
$\| K_N(\zeta;\eps)\|_{L_2 \to H^p} = O(\eps^{-p})$. The dependence of the constants
${\mathcal C}_1(\varphi)$ and ${\mathcal C}_2(\varphi)$ on $\varphi$ is traced; estimates \eqref{0.10} and \eqref{0.11}
are uniform with respect to $\varphi$ in any sector $\varphi \in [\varphi_0, 2\pi - \varphi_0]$ with arbitrarily small $\varphi_0 >0$.
For fixed $\zeta$, estimate  \eqref{0.10} is of sharp order $O(\eps)$ (the order is the same as in $\R^d$),
while estimate \eqref{0.11} is of order $O(\eps^{1/2})$
(the order deteriorates because of the boundary influence). Estimates \eqref{0.10} and \eqref{0.11} show that the error becomes smaller, as
$|\zeta|$ grows.

In the general case, the corrector $K_N(\zeta;\eps)$ involves an auxiliary smoothing operator.
We distinguish an additional condition under which the standard corrector (without smoothing) can be used.

Besides approximation for the resolvent, we find approximation for the operator $g(\x/\eps) b(\D)({A}_{N,\eps}-\zeta I)^{-1}$
(corresponding to the "flux") in the \hbox{$(L_2 \to L_2)$}-operator norm.

We also find approximations for the resolvent $({A}_{N,\eps}-\zeta I)^{-1}$ in a larger set of the parameter $\zeta$;
however, the character of dependence of the right-hand sides in estimates on $\zeta$ is different.
Let us describe these results. The point $\lambda=0$ is the minimal eigenvalue of both operators $A_{N,\eps}$ and $A_N^0$,
moreover, $\operatorname{Ker} A_{N,\eps}= \operatorname{Ker} A_{N}^0$. Let $\lambda_{\eps}$ (respectively, $\lambda^0$)
be the first non-zero eigenvalue of the operator $A_{N,\eps}$ (respectively, $A_N^0$). Let $c_\flat >0$ be their common lower bound,
i.~e., $c_\flat \le \min \{ \lambda_\eps,\lambda^0\}$. Suppose that
$\zeta \in \C \setminus [c_\flat,\infty)$, $\zeta \ne 0$. We put $\zeta - c_\flat = |\zeta - c_\flat | e^{i\psi}$.
For $0< \eps \leqslant \eps_1$ we have
\begin{align}
\label{0.12}
&\| ({A}_{N,\eps}-\zeta I)^{-1} - (A_N^0 - \zeta I)^{-1}\|_{L_2(\O) \to L_2(\O)} \leqslant {\mathfrak C}(\zeta) \eps,
\\
\label{0.13}
&\| ({A}_{N,\eps}-\zeta I)^{-1} - (A_N^0 - \zeta I)^{-1} - \eps^p \widehat{K}_N(\zeta;\eps)\|_{L_2(\O) \to H^p(\O)}
\leqslant  \bigl({\mathfrak C}_1(\zeta)  \eps \bigr)^{1/2} + {\mathfrak C}_2(\zeta) \eps.
\end{align}
Near the point $c_\flat$ the values ${\mathfrak C}(\zeta)$, ${\mathfrak C}_1(\zeta)$, and ${\mathfrak C}_2(\zeta)$  behave as
$C(\psi) |\zeta - c_\flat|^{-2}$.
The dependence of $C(\psi)$ on $\psi$ is traced.
Estimates \eqref{0.12} and \eqref{0.13} are uniform with respect to $\psi$ in any sector $\psi \in [\psi_0,2\pi - \psi_0]$
with arbitrarily small $\psi_0>0$.

\subsection{Method} We rely on the results for operator \eqref{0.5} of order $2p$ in $L_2(\R^d;\C^n)$ obtained in
\cite{KuSu} and \cite{Su2017} (estimates \eqref{0.6} and \eqref{0.7}).

The method of investigation of the operator $A_{N,\eps}$ is similar to the case of the second order operators
and the case of the higher-order operator with the Dirichlet conditions.
It is based on consideration of the associated problem in $\R^d$, introduction of the boundary layer correction term, and a careful analysis of
this term. An important technical role is played by using the Steklov smoothing (borrowed from \cite{ZhPas})
and estimates in the \hbox{$\eps$-neighborhood} of the boundary. First, estimate \eqref{0.11} is proved. Next, we prove estimate~\eqref{0.10}, using the already proved inequality \eqref{0.11} and the duality arguments.

Estimates \eqref{0.12} and \eqref{0.13} are deduced (in a relatively simple way) from the already proved estimates at
the point $\zeta=-1$ and suitable identities for the resolvents.

Two-parametric error estimates in approximation for the resolvent can be applied to homogenization of the parabolic initial boundary value
problems. This application is based on the following representation of the operator exponential:
 \begin{equation*}
e^{-A_{N,\varepsilon}t}=-\frac{1}{2\pi i}\int _\gamma e^{-\zeta t}(A_{N,\varepsilon}-\zeta I)^{-1}\,d\zeta,
\end{equation*}
where $\gamma\subset\mathbb{C}$ is a suitable contour.
For second order operators, parabolic problems have been studied by this method in~\cite{MSu16}.
The author plans to devote a separate paper to application of the results of \cite{Su2017} and the present paper to
parabolic problems (for higher-order operators).

\subsection{Plan of the paper} The paper consists of eight sections.
Section~1 is devoted to the problem in $\R^d$.
Here the class of operators $A_\eps$ in $L_2(\R^d;\C^n)$ is introduced, the effective operator $A^0$ is described,
the smoothing operator is introduced, and the results on homogenization problem in $\R^d$ (from the papers \cite{KuSu} and \cite{Su2017}) are given.
 In Section~2, the operator $A_{N,\eps}$ in a bounded domain with the Neumann conditions is defined and
the effective operator is described.  Section~3 contains auxiliary statements.
In Section~4, the main results for the Neumann problem, namely, estimates \eqref{0.10} and \eqref{0.11}, are formulated
(see Theorems \ref{th3.1} and \ref{th3.2}).
The first two steps of the proof are given: the associated problem in $\R^d$ is considered, the boundary layer correction term
$\w_\eps$ is introduced, and the problem is reduced to estimation of the correction term in $H^p(\O;\C^n)$ and in $L_2(\O;\C^n)$.
In Section~5, the required estimates for the correction term are found,  and the proof of Theorems~\ref{th3.1} and~\ref{th3.2} is completed.
In Section~6, we distinguish the case where the smoothing operator can be removed and the standard corrector can be used.
Some special cases are considered. In Section~7, approximation for the resolvent~$( A_{N,\eps} - \zeta I)^{-1}$ for
$\zeta \in \C \setminus \R_+$, $|\zeta| \le 1$, is obtained. In Section~8, the resolvent for $\zeta \in \C \setminus [c_\flat,\infty)$, $\zeta \ne 0$,
is considered; estimates \eqref{0.12} and \eqref{0.13} are obtained.

\subsection{Notation}
Let $\mathfrak{H}$ and $\mathfrak{G}$ be complex separable Hilbert spaces. The symbols $\left\Vert \cdot\right\Vert _{\mathfrak{H}}$
and $\left(\cdot,\cdot\right)_{\mathfrak{H}}$ stand for the norm and the inner product in $\mathfrak{H}$, respectively;
the symbol $\left\Vert \cdot\right\Vert _{\mathfrak{H} \to \mathfrak{G}}$ denotes the norm of a continuous linear operator acting from
$\mathfrak{H}$ to $\mathfrak{G}$.

The inner product and the norm in $\mathbb{C}^n$ are denoted by $\langle \cdot ,\cdot \rangle$ and $\vert \cdot \vert$, respectively.
Next, $\mathbf{1}=\mathbf{1}_n$ stands for the unit $(n\times n)$-matrix.
If $a$ is a matrix of size $m\times n$, then $\vert a\vert$ denotes the norm of the matrix $a$ viewed as an operator from
$\mathbb{C}^n$ to $\mathbb{C}^m$. The classes $L_{q}$ of $\mathbb{C}^{n}$-valued functions in a domain $\mathcal{O}\subset\mathbb{R}^{d}$
are denoted by $L_{q}(\mathcal{O};\mathbb{C}^{n})$, $1\leqslant q\leqslant\infty$. The Sobolev classes of $\mathbb{C}^{n}$-valued functions in a domain $\mathcal{O}\subseteq\mathbb{R}^{d}$ are denoted by $H^{s}(\mathcal{O};\mathbb{C}^{n})$, $s\in\mathbb{R}$.
If $n=1$, we write simply $L_{q}({\mathcal O})$ and $H^{s}(\mathcal{O})$, but sometimes we use
such simplified notation also for the spaces of vector-valued or matrix-valued functions.

The vectors are denoted by the bold font. We denote
$\x = (x_1,\dots,x_d)\in \R^d$, $iD_j = \partial_j = {\partial}/{\partial x_j}$,
$j=1,\dots,d$, $\mathbf{D}=-i{\nabla}= (D_1,\dots,D_d)$.
Next, if $\alpha=(\alpha_1,\dots,\alpha_d) \in {\mathbb Z}_+^d$ is a multiindex, then $|\alpha|= \sum_{j=1}^d \alpha_j$
and  $\D^\alpha = D_1^{\alpha_1} \cdots D_d^{\alpha_d}$.
For two multiindices $\alpha$ and $\beta$, we write $\beta \leqslant \alpha$ if
$\beta_j \leqslant \alpha_j$, $j=1,\dots,d$; the binomial coefficients are denoted by
$C_\alpha^\beta= C_{\alpha_1}^{\beta_1}\cdots C_{\alpha_d}^{\beta_d}$.

We use the notation $\R_+ = [0,\infty)$. By $C$, $c$, $\mathfrak c$, $\mathcal C$, $\mathfrak C$ (possibly, with indices and marks) we denote
various constants in estimates.

\section{Homogenization problem in $\R^d$}\label{sec1}

\subsection{Lattices in $\mathbb{R}^{d}$}
Let $\Gamma$ be a lattice in $\mathbb{R}^{d}$ generated by the basis $\mathbf{n}_{1},\dots,\mathbf{n}_{d}$:
$$
\Gamma=\left\{ \mathbf{n}\in\mathbb{R}^{d}:\,\mathbf{n}=\sum_{i=1}^{d}l_{i}\mathbf{n}_{i},\, l_{i}\in\mathbb{Z}\right\} ,
$$
and let $\Omega$ be the elementary cell of the lattice $\Gamma$:
$$
\Omega=\left\{ \mathbf{x}\in\mathbb{R}^{d}:\,\mathbf{x}=\sum_{i=1}^{d}t_{i}\mathbf{n}_{i},\, -\frac{1}{2}<t_{i}<\frac{1}{2}\right\} .
$$
The basis $\mathbf{s}_{1},\dots,\mathbf{s}_{d}$ in $\mathbb{R}^{d}$ dual to
the basis $\mathbf{n}_{1},\dots,\mathbf{n}_{d}$ is defined by the relations
$\left\langle \mathbf{s}_{i},\mathbf{n}_{j}\right\rangle _{\mathbb{R}^{d}}=2\pi\delta_{ij}$.
This basis generates a lattice \textit{$\widetilde{\Gamma}$ dual to the lattice $\Gamma$}.
We denote
$$
r_{0}=\frac{1}{2} \min_{0\neq\mathbf{s}\in\widetilde{\Gamma}}\left|\mathbf{s}\right|,
\quad
r_1 = \frac{1}{2} {\rm diam}\,\Omega.
$$

By $\widetilde{H}^{s}(\Omega;\mathbb{C}^{n})$ we denote the subspace of all functions in
$H^{s}(\Omega;\mathbb{C}^{n})$ whose $\Gamma$-periodic extension to $\mathbb{R}^{d}$ belongs to $H_{\mathrm{loc}}^{s}(\mathbb{R}^{d};\mathbb{C}^{n})$.
If $\varphi(\x)$ is a $\Gamma$-periodic function in $\R^d$, we denote
$$
\varphi^\eps(\x) := \varphi( \eps^{-1} \x),\quad \eps >0.
$$

\subsection{The class of operators}
In $L_{2}(\mathbb{R}^{d};\mathbb{C}^{n})$, consider a DO $A_\eps$ of order $2p$ formally given by the differential expression
\begin{equation}
\label{1.1}
A_\eps =b(\mathbf{D})^{*}g^\eps(\mathbf{x})b(\mathbf{D}), \quad \eps >0.
\end{equation}
Here $g(\mathbf{x})$ is a uniformly positive definite and bounded $(m\times m)$-matrix-valued function
(in general, $g(\mathbf{x})$ is Hermitian matrix with complex entries):
\begin{equation}
\label{1.2}
g,\, g^{-1}  \in L_{\infty}(\mathbb{R}^{d}); \quad
g(\mathbf{x}) >0.
\end{equation}
The operator $b(\mathbf{D})$ is given by
\begin{equation}
\label{1.3}
b(\mathbf{D})=\sum_{|\alpha|=p} b_{\alpha}\mathbf{D}^{\alpha},
\end{equation}
where $b_{\alpha}$ are constant $\left(m\times n\right)$-matrices (in general, with complex entries).
It is assumed that $m \geqslant n$ and that the symbol $b({\boldsymbol{\xi}})= \sum_{|\alpha|=p}b_{\alpha} {\boldsymbol \xi}^{\alpha}$
is such that
\begin{equation}
\label{1.3a}
\mathrm{rank}\, b(\boldsymbol{\xi}) =n,\quad 0\neq\boldsymbol{\xi}\in\mathbb{R}^{d}.
\end{equation}
This condition is equivalent to the inequalities
\begin{equation}
\label{1.4}
\begin{aligned}
  \alpha_{0}\mathbf{1}_{n}\leqslant b({\boldsymbol{\theta}})^{*}b(\boldsymbol{\theta})\leqslant\alpha_{1}\mathbf{1}_{n},\quad\boldsymbol{\theta}\in\mathbb{S}^{d-1}; \quad  0<\alpha_{0}\leqslant\alpha_{1}<\infty,
\end{aligned}
\end{equation}
for some positive constants $\alpha_0$ and $\alpha_1$.
Without loss of generality, we assume that
\begin{equation}
\label{1.5}
\left|b_{\alpha}\right|\leqslant\alpha_{1}^{1/2},\quad\left|\alpha\right|=p.
\end{equation}

The precise definition of the operator  $A_\eps$ is given in terms of the quadratic form
\begin{equation}
\label{1.6}
a_\eps[\u,\u] = \intop_{\R^d} \langle g^\eps(\x) b(\D) \u, b(\D)\u \rangle \,d\x,\quad \u \in H^p(\R^d;\C^n).
\end{equation}
Note that the following elementary inequalities hold:
\begin{equation}
\label{1.9}
 \sum_{|\alpha|=p} |\bxi^\alpha|^{2} \leqslant |\bxi|^{2p} \leqslant {\mathfrak c}_p \sum_{|\alpha|=p} |\bxi^\alpha|^{2},
\quad \bxi \in \R^d,
\end{equation}
where ${\mathfrak c}_p$ depends only on $d$ and $p$.
Using the Fourier transformation and relations \eqref{1.2}, \eqref{1.4}, and \eqref{1.9}, it is easy to check that
\begin{equation*}
 c_0 \intop_{\R^d} |\D^p \u|^2\,d\x \leqslant a_\eps \left[\mathbf{u},\mathbf{u}\right] \leqslant c_1 \intop_{\R^d} |\D^p \u|^2\,d\x,
 \quad  \u \in H^p(\R^d;\C^n),
\end{equation*}
where we denote $|\D^p \u|^2:= \sum_{|\alpha|=p} |\D^\alpha \u|^2$.
Here $c_0 :=   \alpha_0 \|g^{-1}\|^{-1}_{L_\infty}$ and $c_1 :=    {\mathfrak c}_p \alpha_1 \|g\|_{L_\infty}$.
Hence, the form \eqref{1.6} is closed and nonnegative.
The selfadjoint operator in $L_2(\R^d;\C^n)$ corresponding to this form is denoted by $A_\eps$.

\subsection{The effective operator}
Now we define the effective operator $A^0$.
Let an $(n \times m)$-matrix-valued function $\Lambda \in \wt{H}^p(\Omega)$ be the (weak) $\Gamma$-periodic solution of the problem
\begin{equation}
\label{1.10}
b(\mathbf{D})^{*}g(\mathbf{x})\left(b(\mathbf{D})\Lambda(\mathbf{x})+\mathbf{1}_{m}\right)=0,\qquad\intop_{\Omega}\Lambda(\mathbf{x})\,d\mathbf{x}=0.
\end{equation}
The \textit{effective matrix} $g^0$ of size $m \times m$ is defined as follows:
\begin{equation}
\label{1.11}
g^0 = |\Omega|^{-1} \intop_\Omega \wt{g}(\x) \,d\x,
\end{equation}
where
\begin{equation}
\label{1.12}
\wt{g}(\x):= g(\x)  \left(b(\mathbf{D})\Lambda(\mathbf{x})+\mathbf{1}_{m}\right).
\end{equation}
It turns out that the matrix $g^0$ is positive definite. The \textit{effective operator} $A^0$ for the operator~\eqref{1.1}
is given by the differential expression
\begin{equation}
\label{1.13}
A^0 = b(\D)^* g^0 b(\D)
\end{equation}
on the domain $H^{2p}(\R^d;\C^n)$.
The symbol $L(\bxi) = b(\bxi)^* g^0 b(\bxi)$ of the effective operator satisfies the estimate
\begin{equation}
\label{1.13a}
L(\bxi) \le C_* |\bxi|^{2p} \1_n,\quad \bxi \in \R^d,\quad C_* := \alpha_1 \|g\|_{L_\infty},
\end{equation}
which follows from \eqref{1.4} and from the estimate for the norm of the matrix $g^0$ (see~\eqref{1.17a} below).

\subsection{Properties of the effective matrix}
The following properties of the effective matrix were checked in \cite[Proposition~5.3]{KuSu}.

\begin{proposition}
\label{prop1.1}
Denote
$$
\overline{g}:  =\left|\Omega\right|^{-1} \int_{\Omega}g(\mathbf{x})d\mathbf{x},\quad
\underline{g}:  =\left(\left|\Omega\right|^{-1}\int_{\Omega}g(\mathbf{x})^{-1}d\mathbf{x}\right)^{-1}.
$$
The effective matrix $g^{0}$ satisfies the estimates
\begin{equation}
\label{1.17}
\underline{g}\leqslant g^{0}\leqslant\overline{g}.
\end{equation}
If $m=n$, then $g^0=\underline{g}$.
\end{proposition}

In homogenization theory for specific DOs, estimates~\eqref{1.17} are known as the Voight-Reuss bracketing.
From \eqref{1.17} it follows that
\begin{equation}
\label{1.17a}
|g^0|  \le \|g\|_{L_\infty}, \quad
|(g^0)^{-1}|  \le \|g^{-1}\|_{L_\infty}.
\end{equation}

Now we distinguish the cases where one of the inequalities in~\eqref{1.17} becomes an identity.
The following two statements were checked in~\cite[Propositions~5.4 and~5.5]{KuSu}.

\begin{proposition}
\label{prop1.2}
Let ${\mathbf g}_k(\x)$, $k=1,\dots,m,$ denote the columns of the matrix $g(\x)$.
The identity $g^0 = \overline{g}$ is equivalent to the relations
\begin{equation}
\label{1.18}
b(\D)^*  {\mathbf g}_k(\x)=0,\quad k=1,\dots,m.
\end{equation}
\end{proposition}

\begin{proposition}
\label{prop1.3}
Let ${\mathbf l}_k(\x)$, $k=1,\dots,m,$ denote the columns of the matrix $g(\x)^{-1}$.
The identity $g^0 = \underline{g}$ is equivalent to the representations
\begin{equation}
\label{1.19}
{\mathbf l}_k(\x) = {\mathbf l}_k^0 + b(\D) {\mathbf v}_k(\x),\quad {\mathbf l}_k^0 \in \C^m,\quad {\mathbf v}_k \in \wt{H}^p(\Omega;\C^n);\quad k=1,\dots,m.
\end{equation}
\end{proposition}

The following property was mentioned in \cite[Remark~5.6]{KuSu}.

\begin{remark}
\label{rem1.4}
If $g^0 = \underline{g}$, then the matrix-valued function \eqref{1.12} is constant{\rm :}
  $\wt{g}(\x) = g^0 = \underline{g}$.
\end{remark}

\subsection{The Steklov smoothing operator}
Let $S_\eps$ be the operator in $L_2(\R^d;\C^m)$ given by
\begin{equation}
\label{1.20}
(S_\eps \u)(\x) = |\Omega|^{-1} \int_{\Omega} {\u}(\x - \eps \z) \, d\z.
\end{equation}
It is called the \textit{Steklov smoothing operator}.
Note that $\| S_\eps \|_{L_2(\R^d) \to L_2(\R^d)} \le 1$. Obviously,
$\D^\alpha S_\eps \u = S_\eps \D^\alpha \u$ for $\u \in H^s(\R^d;\C^m)$ and $|\alpha| \leqslant s$.

We mention some properties of the operator~\eqref{1.20}; see \cite[Lemmas 1.1 and 1.2]{ZhPas} or \cite[Propositions~3.1 and~3.2]{PSu}.

\begin{proposition}
\label{prop1.4}
For any function ${\mathbf u}\in H^1(\R^d;\C^n)$ we have
$$
\| S_\eps \u - \u\|_{L_2(\R^d)}  \leqslant  \eps r_1 \| \D \u\|_{L_2(\R^d)}.
$$
\end{proposition}

\begin{proposition}
\label{prop1.5}
Let $f(\x)$ be a $\Gamma$-periodic function in $\R^d$ such that $f \in L_2(\Omega)$.
Let $[f^\eps]$ be the operator of multiplication by the function~$f(\eps^{-1}\x)$.
Then the operator $[f^\eps] S_\eps$ is continuous in $L_2(\R^d;\C^m)$, and
$$
\| [f^\eps]S_\eps \|_{L_2(\R^d) \to L_2(\R^d)}  \leqslant  |\Omega|^{-1/2} \| f \|_{L_2(\Omega)}, \quad \eps >0.
$$
\end{proposition}

\subsection{The results for homogenization problem in $\R^d$}

In this subsection, we formulate the results on homogenization of the operator $A_\eps$ in $L_2(\R^d;\C^n)$ obtained in~\cite{KuSu}
and~\cite{Su2017}.

A point $\zeta \in \C \setminus \R_+$ is regular for both operators $A_\eps$ and $A^0$.
We put $\zeta = |\zeta| e^{i\varphi}$, $\varphi \in (0,2\pi)$, and denote
\begin{equation}
\label{2.1}
c(\varphi) :=
\begin{cases}
|\sin \varphi|^{-1}, & \varphi \in (0,\pi/2) \cup (3\pi/2, 2\pi) \\
1, & \varphi \in [\pi/2, 3\pi/2]
\end{cases}.
\end{equation}

The following theorem was proved in \cite[Theorem~8.1]{KuSu}.

\begin{theorem}
\label{th2.1}
Suppose that $A_\eps$ is the operator \eqref{1.1} and $A^0$ is the effective operator~\eqref{1.13}.
Let $\zeta = |\zeta| e^{i\varphi}\in \C \setminus \R_+$, and let $c(\varphi)$ be given by~\eqref{2.1}.
Then for $\eps >0$ we have
\begin{equation*}
\| (A_\eps - \zeta I)^{-1} - (A^0 - \zeta I)^{-1} \|_{L_2(\R^d) \to L_2(\R^d)}  \leqslant  C_1 c(\varphi)^2 \eps |\zeta|^{-1+1/2p}.
\end{equation*}
The constant $C_1$ depends only on $d$, $p$, $\alpha_0$, $\alpha_1$, $\|g\|_{L_\infty}$, $\|g^{-1}\|_{L_\infty}$, and the parameters of the lattice~$\Gamma$.
\end{theorem}

To approximate the resolvent \hbox{$(A_\eps - \zeta I)^{-1}$} in the norm of operators acting from $L_2(\R^d;\C^n)$ to
the Sobolev space $H^p(\R^d;\C^n)$, we need to introduce a \textit{corrector}
\begin{equation}
\label{2.3}
K(\zeta;\eps):= [\Lambda^\eps] S_\eps b(\D) (A^0 - \zeta I)^{-1}.
\end{equation}
Recall that $\Lambda$ is the periodic solution of the problem \eqref{1.10} and $S_\eps$ is the smoothing operator~\eqref{1.20}.
By $[\Lambda^\eps]$ we denote the operator of multiplication by the matrix-valued function $\Lambda^\eps(\x)$.
The operator~\eqref{2.3} is a continuous mapping of $L_2(\R^d;\C^n)$ into $H^p(\R^d;\C^n)$.
This can be easily checked by using Proposition~\ref{prop1.5} and the relation $\Lambda \in \widetilde{H}^p(\Omega)$.
Herewith, $\|K(\zeta;\eps)\|_{L_2 \to H^p} = O(\eps^{-p})$.

The following result was obtained in \cite[Theorem~3.3]{Su2017}.

\begin{theorem}
\label{th2.2}
Suppose that the assumptions of Theorem~{\rm \ref{th2.1}} are satisfied.
Let $K(\zeta;\eps)$ be the operator~\eqref{2.3}, and let $\wt{g}(\x)$ be given by~\eqref{1.12}.
Then for $\zeta \in \C \setminus \R_+$ and $\eps>0$ we have
\begin{align*}
\begin{split}
&\| (A_\eps - \zeta I)^{-1} - (A^0 -\zeta I)^{-1} - \eps^{p} K(\zeta;\eps)\|_{L_2(\R^d)\to H^p(\R^d)}
\\
&\leqslant C_2 \left( c(\varphi)^2 \eps |\zeta|^{-1/2 + 1/2p} +  c(\varphi) \eps^p \right) (1+ |\zeta|^{-1/2}),
\end{split}
\\
\begin{split}
&\| g^\eps b(\D)(A_\eps - \zeta I)^{-1} - \wt{g}^\eps S_\eps b(\D)(A^0 -\zeta I)^{-1} \|_{L_2(\R^d)\to L_2(\R^d)}
\\
&\leqslant  C_3 \left(  c(\varphi)^2 \eps |\zeta|^{-1/2 + 1/2p} + c(\varphi) \eps^p \right).
\end{split}
\end{align*}
The constants $C_2$ and $C_3$ depend only on $m$, $d$, $p$, $\alpha_0$, $\alpha_1$, $\|g\|_{L_\infty}$, $\|g^{-1}\|_{L_\infty}$, and the
parameters of the lattice~$\Gamma$.
\end{theorem}

\section{The Neumann problem in a bounded domain}

\subsection{Coercivity}
Let $\O \subset \R^d$ be a bounded domain of class $C^{2p}$.
We impose an additional condition on the symbol of the operator~\eqref{1.3} for $\bxi \in \C^d$.

\begin{condition}
\label{cond2.1}
The matrix-valued function $b(\bxi)= \sum_{|\alpha|=p} b_\alpha \bxi^\alpha$,  $\bxi \in \C^d$, has maximal rank{\rm :}
\begin{equation}
\label{22.1}
\mathrm{rank}\, b(\boldsymbol{\xi}) =n,\quad 0\neq\boldsymbol{\xi}\in\mathbb{C}^{d}.
\end{equation}
\end{condition}

Note that condition \eqref{22.1} is more restrictive than~\eqref{1.3a}.
According to~\cite[Theorem~7.8 in Section~3.7]{Ne}, Condition~\ref{cond2.1} is necessary and sufficient for coercivity
of the form $\| b(\D) \u \|^2_{L_2(\O)}$ on $H^p(\O;\C^n)$.

\begin{proposition}\emph{(\cite{Ne})}
\label{prop2.2}
Condition \emph{\ref{cond2.1}} is necessary and sufficient for existence of constants $k_1, k_2 >0$ such that the
G\"arding-type inequality
\begin{equation}
\label{22.2}
\| \u \|^2_{H^p(\O)} \le k_1 \| b(\D) \u \|^2_{L_2(\O)} + k_2 \| \u \|^2_{L_2(\O)},
\quad \u \in H^p(\O;\C^n),
\end{equation}
holds.
\end{proposition}

\begin{remark}
\mbox{}

{\rm 1)} Estimate \eqref{22.2} is true for any bounded Lipschitz domain~$\O$.

{\rm 2)} The constants $k_1$ and $k_2$ depend on the symbol~$b(\bxi)$ and the domain~$\O$, but in the general case it is difficult to
control these constants explicitly. However, for some particular operators they are known. Therefore, in what follows we indicate
the dependence of other constants on $k_1$ and~$k_2$.
\end{remark}

\subsection{Statement of the problem}
In $L_2(\O;\C^n)$, consider the quadratic form
\begin{equation}
\label{22.3}
a_{N,\eps}[ \u, \u] = \intop_\O \langle g^\eps(\x) b(\D) \u, b(\D)\u \rangle\, d\x,
\quad \u \in H^p(\O;\C^n).
\end{equation}
By \eqref{1.3} and \eqref{1.5},
\begin{equation}
\label{22.4}
a_{N,\eps}[ \u, \u] \le \wt{\mathfrak c}_p \alpha_1 \| g\| _{L_\infty} \| \D^p \u \|^2_{L_2(\O)},
\quad \u \in H^p(\O;\C^n),
\end{equation}
where the constant $\wt{\mathfrak c}_p$ depends only on $d$ and $p$. From \eqref{22.2} it follows that
\begin{equation}
\label{22.5}
\begin{split}
a_{N,\eps}[ \u, \u] \ge \| g^{-1}\|^{-1} _{L_\infty} \| b(\D) \u \|^2_{L_2(\O)}
\ge \| g^{-1}\|^{-1} _{L_\infty} k_1^{-1} \left(\| \u \|^2_{H^p(\O)} - k_2 \| \u \|^2_{L_2(\O)}\right),
\quad \u \in H^p(\O;\C^n).
\end{split}
\end{equation}
Hence, the form \eqref{22.3} is closed and (obviously) nonnegative.
This form generates a selfadjoint operator $A_{N,\eps}$ in $L_2(\O;\C^n)$.
Formally, $A_{N,\eps}$ is given by the expression $b(\D)^* g^\eps(\x)b(\D)$ with the Neumann conditions (natural conditions) on the boundary.

A point $\zeta \in \C \setminus \R_+$ is regular for the operator $A_{N,\eps}$. \emph{Our goal} is to approximate
the resolvent $(A_{N,\eps} - \zeta I)^{-1}$ for small $\eps$ in various operator norms. In other words, we are interested in
the behavior of the solution $\u_\eps := (A_{N,\eps} - \zeta I)^{-1} \FF$ of the Neumann problem with $\FF \in L_2(\O;\C^n)$.
First, we assume in addition that $|\zeta| \ge 1$.

\begin{lemma}\label{lem1}
Let $\zeta = |\zeta| e^{i\varphi} \in \C \setminus \R_+$, $|\zeta| \ge 1$, and let $c(\varphi)$ be given by \eqref{2.1}.
Let $\u_\eps = (A_{N,\eps} - \zeta I)^{-1} \FF$, where $\FF \in L_2(\O;\C^n)$.
Then for $\eps>0$ we have
\begin{align}
\label{22.6}
\| \u_\eps \|_{L_2(\O)} &\le c(\varphi) |\zeta|^{-1} \| \FF \|_{L_2(\O)},
\\
\label{22.7}
\| \u_\eps \|_{H^p(\O)} &\le {\mathcal C}_0 c(\varphi) |\zeta|^{-1/2} \| \FF \|_{L_2(\O)}.
\end{align}
The constant ${\mathcal C}_0$ depends only on $\|g^{-1}\|_{L_\infty}$, $k_1$, and $k_2$.
In operator terms,
\begin{align}
\label{22.8}
\|  (A_{N,\eps} - \zeta I)^{-1} \|_{L_2(\O)\to L_2(\O)} &\le c(\varphi) |\zeta|^{-1},
\\
\label{22.9}
\|  (A_{N,\eps} - \zeta I)^{-1} \|_{L_2(\O)\to H^p(\O)} &\le {\mathcal C}_0 c(\varphi) |\zeta|^{-1/2}.
\end{align}
\end{lemma}

\begin{proof}
Since the form \eqref{22.3} is nonnegative, the spectrum of $A_{N,\eps}$ is contained in $\R_+$.
The norm of the resolvent  $(A_{N,\eps} - \zeta I)^{-1}$ does not exceed the inverse of the distance from the point $\zeta$ to $\R_+$.
This yields \eqref{22.8}.

To check \eqref{22.7}, we write down the integral identity for $\u_\eps$:
\begin{equation}
\label{22.9a}
(g^\eps b(\D) \u_\eps, b(\D) \eeta)_{L_2(\O)} - \zeta (\u_\eps, \eeta)_{L_2(\O)} = (\FF, \eeta)_{L_2(\O)},
\quad \eeta \in H^p(\O;\C^n).
\end{equation}
Substituting $\eeta = \u_\eps$ and using \eqref{22.6}, we arrive at
$$
(g^\eps b(\D) \u_\eps, b(\D) \u_\eps)_{L_2(\O)} \le 2 c(\varphi)^2 |\zeta|^{-1} \| \FF \|^2_{L_2(\O)}.
$$
Together with \eqref{22.5} and \eqref{22.6}, this implies
\begin{equation}
\label{2.10a}
\|\u_\eps \|^2_{H^p(\O)} \le c(\varphi)^2 \left( 2 k_1 \|g^{-1}\|_{L_\infty} |\zeta|^{-1} + k_2 |\zeta|^{-2}\right)
\| \FF \|^2_{L_2(\O)}.
\end{equation}
Taking into account that $|\zeta|\ge 1$, we deduce \eqref{22.7} with the constant ${\mathcal C}_0^2 := 2 k_1 \|g^{-1}\|_{L_\infty} + k_2$.
\end{proof}

 \subsection{The effective operator $A^0_N$}
In $L_2(\O;\C^n)$, consider the quadratic form
\begin{equation}
\label{22.10}
a_{N}^0[ \u, \u] = \intop_\O \langle g^0 b(\D) \u, b(\D)\u \rangle\, d\x,
\quad \u \in H^p(\O;\C^n).
\end{equation}
Here $g^0$ is the effective matrix given by \eqref{1.11}.
By \eqref{1.17a}, we see that the form \eqref{22.10} satisfies estimates of the form \eqref{22.4} and \eqref{22.5} with the same
constants. Thus, this form is closed and nonnegative.
The selfadjoint operator in $L_2(\O;\C^n)$ generated by the form \eqref{22.10} is denoted by $A_N^0$.
Since $\partial \O \in C^{2p}$, the domain of the operator $A_N^0$ is contained in $H^{2p}(\O;\C^n)$, and we have
\begin{equation}
\label{22.11}
\| (A_{N}^0 +I)^{-1} \|_{L_2(\O)\to H^{2p}(\O)} \le \widehat{c},
\end{equation}
where the constant $\widehat{c}$ depends only on $\|g\|_{L_\infty}$, $\|g^{-1}\|_{L_\infty}$, $\alpha_0$, $\alpha_1$,
$k_1$, $k_2$, and the domain $\O$.
To justify this fact, we refer to Theorems~2.2 and~2.3 of the paper~\cite{So}.

\begin{remark}
Instead of the condition $\partial\mathcal{O}\in C^{2p}$, one could impose the following implicit condition:
a bounded domain $\mathcal{O}\subset \mathbb{R}^d$ with Lipschitz boundary is such that
estimate \eqref{22.11} holds. The results of the paper remain valid for such domain.
In the case of the scalar elliptic operators, wide sufficient conditions on $\partial \mathcal{O}$ ensuring \eqref{22.11}
can be found in \textnormal{\cite{KoE}} and \textnormal{\cite[Chapter~7]{MaSh} (}in particular, it suffices
that $\partial\mathcal{O}\in C^{2p-1,\nu}$, $\nu > 1/2${\rm)}.
\end{remark}

The principal approximation of $\u_\eps$ is the function $\u_0 := (A_N^0 - \zeta I)^{-1}\FF$.

\begin{lemma}\label{lem_eff}
Let $\zeta = |\zeta| e^{i\varphi} \in \C \setminus \R_+$, $|\zeta| \ge 1$,
and let $c(\varphi)$ be given by \eqref{2.1}. Let $\u_0 = (A_{N}^0 - \zeta I)^{-1} \FF$, where $\FF \in L_2(\O;\C^n)$.
Then for $\eps>0$ we have
\begin{align}
\label{22.12}
\| \u_0 \|_{L_2(\O)} &\le c(\varphi) |\zeta|^{-1} \| \FF \|_{L_2(\O)},
\\
\label{22.13}
\| \u_0 \|_{H^p(\O)} &\le {\mathcal C}_0 c(\varphi) |\zeta|^{-1/2}  \| \FF \|_{L_2(\O)},
\\
\nonumber
\| \u_0 \|_{H^{2p}(\O)} &\le 2 \widehat{c} c(\varphi) \| \FF \|_{L_2(\O)}.
\end{align}
In operator terms,
\begin{align}
\label{22.15}
\|  (A_{N}^0 - \zeta I)^{-1} \|_{L_2(\O)\to L_2(\O)} &\le c(\varphi) |\zeta|^{-1},
\\
\nonumber
\|  (A_{N}^0 - \zeta I)^{-1} \|_{L_2(\O)\to H^p(\O)} &\le {\mathcal C}_0 c(\varphi) |\zeta|^{-1/2},
\\
\label{22.17}
\|  (A_{N}^0 - \zeta I)^{-1} \|_{L_2(\O)\to H^{2p}(\O)} &\le 2 \widehat{c} c(\varphi).
\end{align}
\end{lemma}

\begin{proof}
Estimates \eqref{22.12} and \eqref{22.13} are checked similarly to the proof of Lemma~\ref{lem1}.

Estimate \eqref{22.17} follows from \eqref{22.11} and the inequality
$$
\| (A_N^0 +I) (A_N^0 - \zeta I)^{-1}\|_{L_2(\O) \to L_2(\O)} \le \sup_{x \ge 0} \frac{x+1}{ |x - \zeta|}
\le 2 c(\varphi).
$$
\end{proof}

\section{Auxiliary statements}\label{auxiliary}

\subsection{A traditional lemma of homogenization theory}

\begin{lemma}\label{traditional}
Let $f_\alpha(\x)$, $|\alpha|=p$, be $\Gamma$-periodic $(n \times m)$-matrix-valued functions in $\R^d$ such that
$f_\alpha \in L_2(\Omega)$, $\overline{f_\alpha}=0$, and
\begin{equation}\label{trad1}
\sum_{|\alpha|=p} \partial^\alpha f_\alpha(\x) =0,
\end{equation}
where the last equation is understood in the sense of distributions. Then there exist
$\Gamma$-periodic $(n \times m)$-matrix-valued functions $M_{\alpha \beta}(\x)$, $|\alpha|=|\beta|=p$, such that
\begin{align}
\label{trad2}
&M_{\alpha \beta} \in \wt{H}^p(\Omega), \quad   \overline{M_{\alpha \beta}} =0,
\quad  M_{\alpha \beta}(\x) = - M_{ \beta \alpha}(\x),\quad |\alpha|=|\beta|=p,
\\
\label{trad3}
&f_\alpha(\x) = \sum_{|\beta|=p} \partial^\beta M_{\alpha \beta} (\x),\quad |\alpha|=p.
\end{align}
 We have
\begin{equation}\label{trad4}
\| M_{\alpha \beta} \|_{H^p(\Omega)} \le  \check{\mathfrak c}_p\left( \|f_\alpha\|_{L_2(\Omega)} + \|f_\beta\|_{L_2(\Omega)} \right),
\quad |\alpha|=|\beta|=p,
\end{equation}
where the constant $\check{\mathfrak c}_p$ depends only on $d$, $p$, and the parameters of the lattice $\Gamma$.
\end{lemma}

\begin{proof}
Let $|\alpha|=p$. Suppose that an $(n \times m)$-matrix-valued function $\Phi_\alpha(\x)$ is the $\Gamma$-periodic solution of the problem
\begin{equation}\label{trad5}
\wt{\Delta}_p \Phi_\alpha(\x) = f_\alpha(\x),\quad \int_\Omega \Phi_\alpha(\x) \, d\x =0,
\end{equation}
where $\wt{\Delta}_p:= \sum_{|\beta|=p} \partial^{2 \beta}$.
The solution exists, it is unique, we have $\Phi_\alpha \in \wt{H}^{2p}(\Omega)$, and
\begin{equation}\label{trad5a}
\| \Phi_\alpha\|_{H^{2p}(\Omega)}  \le \check{\mathfrak c}_p \|f_\alpha\|_{L_2(\Omega)},
\end{equation}
where the constant $\check{\mathfrak c}_p$ depends only on $d$, $p$, and the parameters of the lattice~$\Gamma$.
These properties can be easily checked by the Fourier series.

We put
\begin{equation}\label{trad6}
M_{\alpha \beta}(\x) := \partial^\beta \Phi_\alpha(\x) - \partial^\alpha \Phi_\beta(\x),
\quad |\alpha|=|\beta|=p.
\end{equation}
Obviously, \eqref{trad2} is valid.

Next, note that, by \eqref{trad1} and \eqref{trad5}, the matrix-valued function $\Psi(\x):= \sum_{|\beta|=p} \partial^\beta \Phi_\beta(\x)$
is a $\Gamma$-periodic solution of the problem
\hbox{$\wt{\Delta}_p \Psi(\x) =0$}, $\overline{\Psi}=0$, whence $\Psi(\x) =0$. Together with \eqref{trad5} and \eqref{trad6}, this yields
\begin{equation*}
\begin{split}
\sum_{|\beta|=p} \partial^\beta M_{\alpha \beta} (\x) =
\sum_{|\beta|=p} \partial^{2\beta} \Phi_{\alpha} (\x) - \sum_{|\beta|=p} \partial^\beta \partial^\alpha \Phi_{\beta} (\x)
= \wt{\Delta}_p \Phi_\alpha(\x) - \partial^\alpha \Psi(\x) = f_\alpha(\x),
\end{split}
\end{equation*}
which implies \eqref{trad3}. Estimates \eqref{trad4} follow from \eqref{trad5a} and~\eqref{trad6}.
\end{proof}

\subsection{Properties of $\Lambda(\x)$}

In what follows, we need the following estimates for the periodic solution $\Lambda(\x)$ of problem
\eqref{1.10} (see \cite[Corollary~5.8]{KuSu}):
\begin{align}
\label{Lambda2}
&\| b(\D)\Lambda\|_{L_2(\Omega)} \leqslant |\Omega|^{1/2} C_\Lambda^{(1)},\quad
C_\Lambda^{(1)} :=  m^{1/2} \|g\|_{L_\infty}^{1/2}\|g^{-1}\|_{L_\infty}^{1/2},
\\
\label{Lambda3}
&\| \Lambda\|_{H^p(\Omega)} \leqslant |\Omega|^{1/2}  C_\Lambda, \quad
 C_\Lambda := C_\Lambda^{(1)} \alpha_0^{-1/2} \biggl( \sum_{|\beta|\leqslant p}(2 r_0)^{-2 (p-|\beta|)}\biggr)^{1/2}.
\end{align}

\subsection{Estimates in the neighborhood of the boundary}
In this subsection, we formulate two simple auxiliary statements that are valid in a bounded domain $\O$
with Lipschitz boundary. Precisely, we impose the following condition.

\begin{condition}\label{cond_eps}
Let $\O \subset \R^d$ be a bounded domain. Let
\hbox{$(\partial \O)_\eps := \{\x \in \R^d: {\rm dist}\,\{\x;\partial {\mathcal O} \}< \eps \}$}.
Suppose that there exists a number $\eps_0 \in (0,1]$ such that the strip
$(\partial \O)_{\eps_0}$ can be covered by a finite number of open sets admitting diffeomorphisms
of class  $C^{0,1}$ rectifying the boundary $\partial \O$. Denote $\eps_1 = \eps_0 (1+r_1)^{-1}$, where $2r_1 = {\rm diam}\,\Omega$.
\end{condition}

Obviously, Condition \ref{cond_eps} is less restrictive than the above assumption $\partial \O \in C^{2p}$.

\begin{lemma}\label{lem_dO1}
Suppose that Condition~{\rm \ref{cond_eps}} is satisfied.
Then for any function $u \in H^1(\R^d)$ we have
$$
\int_{({\partial \O})_\eps} |u|^2\,d\x \leq
\beta_0 \eps \|u\|_{H^1(\R^d)} \|u\|_{L_2(\R^d)},\quad 0 <  \eps \leq \eps_0.
$$
The constant $\beta_0$ depends only on the domain $\O$.
\end{lemma}

\begin{lemma}\label{lem_dO2}
Suppose that Condition {\rm \ref{cond_eps}} is satisfied. Let $f(\x)$~be a \hbox{$\Gamma$-periodic} function in $\R^d$ such that
$f \in L_2(\Omega)$. Let $S_\eps$~be the operator \eqref{1.20}.
Denote $\beta_* := \beta_0 (1+r_1)$, where $2r_1 = {\rm diam}\,\Omega$.
Then for any function $\u \in H^1(\R^d;\C^m)$ we have
$$
\int_{({\partial \O})_\eps} |f^\eps(\x)|^2 |(S_\eps \u)(\x)|^2\,d\x
\leq \beta_*  \eps |\Omega|^{-1} \|f\|_{L_2(\Omega)}^2 \|\u\|_{H^1(\R^d)}
\|\u\|_{L_2(\R^d)}, \quad 0 <  \eps \leq \eps_1.
$$
\end{lemma}

Lemma \ref{lem_dO2} is an analogue of Lemma~2.6 from~\cite{ZhPas}.
Lemmas \ref{lem_dO1} and \ref{lem_dO2} were checked in~\cite[\S5]{PSu} under the condition $\partial \O \in C^1$,
but the proofs work also under Condition~\ref{cond_eps}.

\section{The results for the Neumann problem}

\subsection{Approximation of the resolvent $(A_{N,\eps} - \zeta I)^{-1}$ for $|\zeta| \ge 1$}
Let us formulate the main results.

\begin{theorem}\label{th3.1}
Suppose that $\O\subset \R^d$ is a bounded domain of class $C^{2p}$.
Let $\zeta = |\zeta| e^{i\varphi} \in \C \setminus \R_+$,
$|\zeta| \ge 1$, and let $c(\varphi)$ be given by \eqref{2.1}. Suppose that the operators $A_{N,\eps}$ and $A_N^0$
correspond to the quadratic forms \eqref{22.3} and~\eqref{22.10}, respectively. Let $\u_\eps = (A_{N,\eps} - \zeta I)^{-1} \FF$ and $\u_0 = (A_{N}^0 - \zeta I)^{-1} \FF$, where $\FF \in L_2(\O;\C^n)$.
Suppose that the number $\eps_1$ is subject to  Condition~\emph{\ref{cond_eps}}. Then for $0 < \eps \le \eps_1$
we have
\begin{equation*}
\|\u_\eps - \u_0\|_{L_2(\O)} \le {\mathcal C}_1 c(\varphi)^2  \eps |\zeta|^{-1+1/2p} \|\FF\|_{L_2(\O)}.
\end{equation*}
In operator terms,
\begin{equation}
\label{3.2}
\| (A_{N,\eps} - \zeta I)^{-1} - (A_{N}^0 - \zeta I)^{-1} \|_{L_2(\O) \to L_2(\O)}
\le {\mathcal C}_1 c(\varphi)^2  \eps |\zeta|^{-1+1/2p}.
\end{equation}
The constant ${\mathcal C}_1$ depends only on $d$, $p$, $m$, $\alpha_0$, $\alpha_1$, $\|g\|_{L_\infty}$, $\|g^{-1}\|_{L_\infty}$, $k_1$, $k_2$,
the parameters of the lattice $\Gamma$, and the domain $\O$.
\end{theorem}

To approximate $\u_\eps$ in $H^p(\O;\C^n)$, we need to introduce a corrector. We fix a linear continuous extension operator
\begin{equation*}
P_\O: H^s(\O;\C^n) \to H^s(\R^d;\C^n),\quad s=0,1,\dots,2p.
\end{equation*}
Such an operator exists (see, e.~g., \cite{St}). Denote
\begin{equation}
\label{3.4}
\| P_\O \|_{H^s(\O) \to H^s(\R^d)} =: C_\O^{(s)},\quad s=0,1,\dots,2p.
\end{equation}
The constants $C_\O^{(s)}$ depend only on the domain $\O$ and $s$.
By $R_\O$ we denote the operator of restriction of functions in $\R^d$ onto the domain $\O$.
We introduce a corrector
\begin{equation}
\label{3.5}
K_N(\zeta;\eps) := R_\O [\Lambda^\eps] S_\eps b(\D) P_\O (A_{N}^0 -\zeta I)^{-1}.
\end{equation}
The operator $K_N(\zeta;\eps)$ is a continuous mapping of $L_2(\O;\C^n)$ into $H^p(\O;\C^n)$.
Indeed, the operator $b(\D) P_\O (A_{N}^0 -\zeta I)^{-1}$ is continuous from $L_2(\O;\C^n)$ to $H^p(\R^d;\C^m)$,
and the operator $[\Lambda^\eps] S_\eps$ is continuous from $H^p(\R^d;\C^m)$ to $H^p(\R^d;\C^n)$ (this follows from
Proposition~\ref{prop1.5} and the relation $\Lambda \in \widetilde{H}^p(\Omega)$).

Let $\u_0 = (A_{N}^0 - \zeta I)^{-1} \FF$. Denote $\wt{\u}_0 := P_\O \u_0$ and put
\begin{align}
\label{3.6}
\wt{\mathbf v}_\eps(\x) &:= \wt{\u}_0(\x) + \eps^p \Lambda^\eps(\x) (S_\eps b(\D) \wt{\u}_0)(\x),
\quad \x \in \R^d,
\\
\label{3.7}
{\mathbf v}_\eps &:= \wt{\mathbf v}_\eps \vert_\O.
\end{align}
Then
\begin{equation}
\label{3.8}
{\mathbf v}_\eps = (A_{N}^0 - \zeta I)^{-1} \FF + \eps^p K_{N}( \zeta; \eps) \FF.
\end{equation}

\begin{theorem}\label{th3.2}
Suppose that the assumptions of Theorem~\emph{\ref{th3.1}} are satisfied. Let $K_N(\zeta;\eps)$ be the operator \eqref{3.5}.
Let ${\mathbf v}_\eps$ be given by~\eqref{3.8}. Then for \hbox{$0 < \eps \le \eps_1$} we have
\begin{equation}
\label{3.9}
\begin{split}
\|\u_\eps - {\mathbf v}_\eps\|_{H^p(\O)} \le {\mathcal C}_2 \left( c(\varphi) \eps^{1/2} |\zeta|^{-1/2+1/4p}
+ c(\varphi)^2 \eps |\zeta|^{-1/2+1/2p}  + c(\varphi) \eps^p \right) \|\FF\|_{L_2(\O)}.
\end{split}
\end{equation}
In operator terms,
\begin{equation}
\label{3.10}
\begin{split}
&\| (A_{N,\eps} - \zeta I)^{-1} - (A_{N}^0 - \zeta I)^{-1} - \eps^p K_N(\zeta;\eps)\|_{L_2(\O) \to H^p(\O)}
\\
&\le  {\mathcal C}_2 \left( c(\varphi) \eps^{1/2} |\zeta|^{-1/2+1/4p}
+ c(\varphi)^2 \eps |\zeta|^{-1/2+1/2p}  + c(\varphi) \eps^p \right).
\end{split}
\end{equation}
Let $\wt{g}(\x)$ be given by~\eqref{1.12}.
For the flux $\p_\eps = g^\eps b(\D) \u_\eps$  for $0 < \eps \le \eps_1$ we have
\begin{equation}
\label{3.11}
\begin{split}
\|\p_\eps - \wt{g}^\eps S_\eps b(\D) \wt{\u}_0 \|_{L_2(\O)}
\le {\mathcal C}_3  \left( c(\varphi) \eps^{1/2} |\zeta|^{-1/2+1/4p}
+ c(\varphi)^2 \eps |\zeta|^{-1/2+1/2p}  + c(\varphi) \eps^p \right) \|\FF\|_{L_2(\O)}.
\end{split}
\end{equation}
The constants ${\mathcal C}_2$ and ${\mathcal C}_3$ depend only on $d$, $p$, $m$, $\alpha_0$, $\alpha_1$, $\|g\|_{L_\infty}$, $\|g^{-1}\|_{L_\infty}$, $k_1$, $k_2$, the parameters of the lattice $\Gamma$, and the domain $\O$.
\end{theorem}

\begin{remark}\label{rem3.2a}\mbox{}

{\rm 1)}
For fixed $\zeta \in \C \setminus \R_+$, $|\zeta| \geqslant 1$,
estimates of Theorem~{\rm \ref{th3.1}} are of sharp order $O(\eps)$.

{\rm 2)}
For fixed $\zeta \in \C \setminus \R_+$, $|\zeta| \geqslant 1$, estimates of Theorem~{\rm \ref{th3.2}} are of order $O(\eps^{1/2})$.
This is explained by the boundary influence.

{\rm 3)}
The error in the estimates of Theorems~{\rm \ref{th3.1}} and~{\rm \ref{th3.2}} becomes smaller, as  $|\zeta|$ grows.

{\rm 4)}
Estimates of Theorems~{\rm \ref{th3.1}} and {\rm \ref{th3.2}} are uniform with respect to $\varphi$ in any domain
$\{ \zeta = |\zeta| e^{i\varphi}: |\zeta| \geqslant 1, \varphi_0 \leqslant \varphi \leqslant 2\pi - \varphi_0\}$
with arbitrarily small $\varphi_0 >0$.
\end{remark}

\subsection{The first step of the proof. The associated problem in $\R^d$}

The proof of Theorems \ref{th3.1} and \ref{th3.2} relies on application of the results
for the problem in $\R^d$ and introduction of the boundary layer correction term.

By Lemma \ref{lem_eff} and \eqref{3.4}, we have
\begin{align}
\label{3.12}
\| \wt{\u}_0 \|_{L_2(\R^d)} &\le C_\O^{(0)} c(\varphi) |\zeta|^{-1} \| \FF \|_{L_2(\O)},
\\
\label{3.13}
\| \wt{\u}_0 \|_{H^p(\R^d)} &\le C^{(p)} c(\varphi) |\zeta|^{-1/2}  \| \FF \|_{L_2(\O)},
\\
\label{3.14}
\| \wt{\u}_0 \|_{H^{2p}(\R^d)} &\le C^{(2p)}  c(\varphi) \| \FF \|_{L_2(\O)},
\end{align}
where $C^{(p)} := C_\O^{(p)}{\mathcal C}_0$ and $C^{(2p)} := 2 C_\O^{(2p)} \widehat{c}$.
Interpolating between \eqref{3.13} and \eqref{3.14}, we obtain
\begin{equation}
\label{3.15}
\| \wt{\u}_0 \|_{H^{p+k}(\R^d)} \le C^{(p+k)}  c(\varphi) |\zeta|^{-1/2 + k/2p} \| \FF \|_{L_2(\O)},\quad k=0,1,\dots,p.
\end{equation}
Here $C^{(p+k)} := (C^{(p)})^{1-k/p} (C^{(2p)})^{k/p}$.

We put
\begin{equation}
\label{3.16}
 \wt{\FF} := A^0  \wt{\u}_0  - \zeta \wt{\u}_0.
\end{equation}
Then $\wt{\FF} \in L_2(\R^d;\C^n)$ and $\wt{\FF}\vert_\O = \FF$.
From \eqref{1.13a}, \eqref{3.12}, and \eqref{3.14} it follows that
\begin{equation}
\label{3.17}
 \|\wt{\FF} \|_{L_2(\R^d)} \le  C_* \| \wt{\u}_0\|_{H^{2p}(\R^d)} + | \zeta| \| \wt{\u}_0\|_{L_2(\R^d)}
\le
{\mathcal C}_4 c(\varphi) \| \FF \|_{L_2(\O)},
\end{equation}
where ${\mathcal C}_4 := C_* C^{(2p)}+ C_\O^{(0)}$. Let $\wt{\u}_\eps \in H^p(\R^d;\C^n)$ be the generalized solution
of the following equation in $\R^d$:
\begin{equation}
\label{3.18}
A_\eps \wt{\u}_\eps - \zeta \wt{\u}_\eps =  \wt{\FF},
\end{equation}
i.~e., $\wt{\u}_\eps = (A_\eps - \zeta I)^{-1} \wt{\FF}$.
One can apply theorems of Section~\ref{sec1}. From Theorems \ref{th2.1} and \ref{th2.2} and relations~\eqref{3.16}--\eqref{3.18}
it follows that
\begin{align}
\label{3.19}
&\| \wt{\u}_\eps - \wt{\u}_0 \|_{L_2(\R^d)} \le C_1 {\mathcal C}_4 c(\varphi)^3 \eps |\zeta|^{-1+1/2p} \| \FF \|_{L_2(\O)},
\\
\label{3.20}
&\| \wt{\u}_\eps - \wt{\mathbf v}_\eps \|_{H^p(\R^d)} \le  2C_2 {\mathcal C}_4 \left( c(\varphi)^3 \eps |\zeta|^{-1/2+1/2p} +
 c(\varphi)^2 \eps^p \right) \| \FF \|_{L_2(\O)},
\end{align}
\begin{equation}
\label{3.20a}
\begin{split}
\| g^\eps b(\D)\wt{\u}_\eps - \wt{g}^\eps S_\eps b(\D) \wt{\u}_0 \|_{L_2(\R^d)}
\le C_3 {\mathcal C}_4 \left( c(\varphi)^3 \eps |\zeta|^{-1/2+1/2p} +
 c(\varphi)^2 \eps^p \right) \| \FF \|_{L_2(\O)},
\end{split}
\end{equation}
for $\eps>0$. We have taken into account that $|\zeta| \ge 1$.

\subsection{The second step of the proof. Introduction of the boundary layer correction term}
We introduce the correction term $\w_\eps \in H^p(\O;\C^n)$  as a function satisfying the integral identity
\begin{equation}
\label{3.24}
\begin{split}
(& g^\eps b(\D) \w_\eps, b(\D)  \eeta )_{L_2(\O)} - \zeta  ( \w_\eps, \eeta )_{L_2(\O)}
\\
&=
(\wt{g}^\eps S_\eps b(\D) \wt{\u}_0,  b(\D)  \eeta )_{L_2(\O)} -
(\zeta \u_0 + \FF, \eeta )_{L_2(\O)},\quad \eeta \in H^p(\O;\C^n).
\end{split}
\end{equation}
Since the right-hand side here is an antilinear continuous functional of $\eeta \in H^p(\O;\C^n)$,
one can check in the standard way that the solution $\w_\eps$ exists and is unique.

Taking the term $\w_\eps$ into account, we obtain approximation for $\u_\eps$
in the $H^p$-norm with sharp-order error estimate $O(\eps)$.

\begin{theorem}
Suppose that the assumptions of Theorem~\emph{\ref{th3.2}} are satisfied.
Suppose that $\w_\eps$ satisfies identity \eqref{3.24}.
Then for $\zeta \in \C \setminus \R_+$, $|\zeta| \ge 1$, and $\eps >0$ we have
\begin{equation}
\label{3.25}
\begin{split}
\| \u_\eps - {\mathbf v}_\eps + \w_\eps \|_{H^p(\O)}
\le {\mathcal C}_5 \left( c(\varphi)^4 \eps |\zeta|^{-1/2+1/2p} + c(\varphi)^3 \eps^p  \right) \|\FF\|_{L_2(\O)}.
\end{split}
\end{equation}
The constant ${\mathcal C}_5$ depends only on  $d$, $p$, $m$, $\alpha_0$, $\alpha_1$, $\|g\|_{L_\infty}$, $\|g^{-1}\|_{L_\infty}$, $k_1$, $k_2$,
the parameters of the lattice $\Gamma$, and the domain $\O$.
\end{theorem}

\begin{proof}
Let $\wt{\u}_\eps = (A_\eps - \zeta I)^{-1}\wt{\FF}$. Denote $\VV_\eps := \u_\eps - \wt{\u}_\eps + \w_\eps$.
By \eqref{22.9a} and \eqref{3.24}, the function $\VV_\eps \in H^p(\O;\C^n)$ satisfies the identity
\begin{equation}
\label{3.27}
\begin{split}
&(g^\eps b(\D) \VV_\eps, b(\D) \eeta)_{L_2(\O)} - \zeta  ( \VV_\eps,  \eeta)_{L_2(\O)}
\\
&= (\wt{g}^\eps S_\eps b(\D) \wt{\u}_0 - g^\eps b(\D) \wt{\u}_\eps,  b(\D) \eeta)_{L_2(\O)}
 + \zeta (\wt{\u}_\eps - \u_0, \eeta)_{L_2(\O)}
\end{split}
\end{equation}
for $\eeta \in H^p(\O;\C^n)$. Denote the right-hand side of \eqref{3.27} by ${\mathcal I}_\eps[\eeta]$.
From \eqref{3.19} and \eqref{3.20a} it follows that
\begin{equation}
\label{3.28}
\begin{split}
| {\mathcal I}_\eps[\eeta] | &\le C_3 {\mathcal C}_4 \left( c(\varphi)^3 \eps |\zeta|^{-1/2+1/2p} +
 c(\varphi)^2 \eps^p \right) \| \FF \|_{L_2(\O)} \| b(\D) \eeta \|_{L_2(\O)}
\\
&+ C_1 {\mathcal C}_4  c(\varphi)^3 \eps |\zeta|^{1/2p} \| \FF \|_{L_2(\O)} \| \eeta \|_{L_2(\O)}.
\end{split}
\end{equation}
We substitute $\eeta = \VV_\eps$ in \eqref{3.27}:
\begin{equation}
\label{3.29}
(g^\eps b(\D) \VV_\eps, b(\D) \VV_\eps)_{L_2(\O)} - \zeta \| \VV_\eps \|^2_{L_2(\O)} =  {\mathcal I}_\eps[\VV_\eps ].
\end{equation}
Taking the imaginary part of \eqref{3.29} and using estimate \eqref{3.28}, we have
\begin{equation}
\label{3.30}
\begin{split}
| {\rm Im} \, \zeta | \| \VV_\eps \|^2_{L_2(\O)}
&\le  C_3 {\mathcal C}_4 \left( c(\varphi)^3 \eps |\zeta|^{-1/2+1/2p} + c(\varphi)^2 \eps^p \right) \| \FF \|_{L_2(\O)} \| b(\D) \VV_\eps \|_{L_2(\O)}
\\
&+ C_1 {\mathcal C}_4  c(\varphi)^3 \eps |\zeta|^{1/2p} \| \FF \|_{L_2(\O)} \| \VV_\eps \|_{L_2(\O)}.
\end{split}
\end{equation}
If ${\rm Re} \, \zeta \ge 0$ (and then ${\rm Im} \, \zeta \ne 0$), it follows that
\begin{equation}
\label{3.31}
\begin{split}
| \zeta | \| \VV_\eps \|^2_{L_2(\O)}
&\le  2 C_3 {\mathcal C}_4 \left( c(\varphi)^4 \eps |\zeta|^{-1/2+1/2p} + c(\varphi)^3 \eps^p \right) \| \FF \|_{L_2(\O)} \| b(\D) \VV_\eps \|_{L_2(\O)}
\\
&+ C_1^2 {\mathcal C}_4^2  c(\varphi)^8 \eps^2 |\zeta|^{-1 + 1/p} \| \FF \|^2_{L_2(\O)}, \quad {\rm Re} \, \zeta \ge 0.
\end{split}
\end{equation}
If ${\rm Re} \, \zeta < 0$, we take the real part of \eqref{3.29} and obtain
\begin{equation}
\label{3.32}
\begin{split}
| {\rm Re} \, \zeta | \| \VV_\eps \|^2_{L_2(\O)}
&\le  C_3 {\mathcal C}_4 \left( \eps |\zeta|^{-1/2+1/2p} +  \eps^p \right) \| \FF \|_{L_2(\O)} \| b(\D) \VV_\eps \|_{L_2(\O)}
\\
&+ C_1 {\mathcal C}_4  \eps |\zeta|^{1/2p} \| \FF \|_{L_2(\O)} \|  \VV_\eps \|_{L_2(\O)} , \quad {\rm Re} \, \zeta < 0.
\end{split}
\end{equation}
We have taken into account that $c(\varphi)=1$ if ${\rm Re} \, \zeta < 0$.
Summing \eqref{3.30} and \eqref{3.32} up, we deduce
\begin{equation}
\label{3.33}
\begin{split}
| \zeta | \| \VV_\eps \|^2_{L_2(\O)}
&\le 4 C_3 {\mathcal C}_4 \left( \eps |\zeta|^{-1/2+1/2p} +  \eps^p \right) \| \FF \|_{L_2(\O)} \| b(\D) \VV_\eps \|_{L_2(\O)}
\\
&+ 4 C_1^2 {\mathcal C}_4^2  \eps^2 |\zeta|^{-1+1/p} \| \FF \|^2_{L_2(\O)} , \quad {\rm Re} \, \zeta < 0.
\end{split}
\end{equation}
By \eqref{3.31} and \eqref{3.33}, we see that
\begin{equation}
\label{3.34}
\begin{split}
| \zeta | \| \VV_\eps \|^2_{L_2(\O)}
&\le 4 C_3 {\mathcal C}_4 \left( c(\varphi)^4 \eps |\zeta|^{-1/2+1/2p} + c(\varphi)^3 \eps^p \right) \| \FF \|_{L_2(\O)} \| b(\D) \VV_\eps \|_{L_2(\O)}
\\
&+ 4 C_1^2 {\mathcal C}_4^2  c(\varphi)^8 \eps^2 |\zeta|^{-1 + 1/p} \| \FF \|^2_{L_2(\O)}
\end{split}
\end{equation}
for all values of $\zeta$ under consideration.

Now, considering the real part of \eqref{3.29} and using \eqref{3.28} and \eqref{3.34}, we arrive at
\begin{equation*}
\begin{split}
&(g^\eps b(\D) \VV_\eps, b(\D) \VV_\eps)_{L_2(\O)} \le
| {\mathcal I}_\eps [\VV_\eps] | + |\zeta|  \| \VV_\eps \|^2_{L_2(\O)}
\\
&\le 9 C_3 {\mathcal C}_4 \left( c(\varphi)^4 \eps |\zeta|^{-1/2+1/2p} + c(\varphi)^3 \eps^p \right) \| \FF \|_{L_2(\O)} \| b(\D) \VV_\eps \|_{L_2(\O)}
\\
&+ 9 C_1^2 {\mathcal C}_4^2  c(\varphi)^8 \eps^2 |\zeta|^{-1 + 1/p} \| \FF \|^2_{L_2(\O)}.
\end{split}
\end{equation*}
This implies
\begin{equation}
\label{3.36}
\begin{split}
 \|b(\D) \VV_\eps \|_{L_2(\O)} \le
 {\mathcal C}_6 \left( c(\varphi)^4 \eps |\zeta|^{-1/2+1/2p} + c(\varphi)^3 \eps^p \right) \| \FF \|_{L_2(\O)},
\end{split}
\end{equation}
where ${\mathcal C}_6^2 := 81 \| g^{-1}\|_{L_\infty}^2 C_3^2 {\mathcal C}_4^2 + 18 \| g^{-1}\|_{L_\infty} C_1^2 {\mathcal C}_4^2$.
Combining this with \eqref{3.34}, we obtain
\begin{equation}
\label{3.37}
\begin{split}
 |\zeta|^{1/2} \| \VV_\eps \|_{L_2(\O)} \le
 {\mathcal C}_7 \left( c(\varphi)^4 \eps |\zeta|^{-1/2+1/2p} + c(\varphi)^3 \eps^p \right) \| \FF \|_{L_2(\O)},
\end{split}
\end{equation}
where ${\mathcal C}_7^2 := 4  C_3  {\mathcal C}_4 {\mathcal C}_6 + 4 C_1^2 {\mathcal C}_4^2$.

From \eqref{22.2}, \eqref{3.36}, \eqref{3.37}, and the inequality $|\zeta| \ge 1$ it follows that
\begin{equation}
\label{3.38}
\begin{split}
  \| \VV_\eps \|_{H^p(\O)} \le
 {\mathcal C}_8 \left( c(\varphi)^4 \eps |\zeta|^{-1/2+1/2p} + c(\varphi)^3 \eps^p \right) \| \FF \|_{L_2(\O)},
\end{split}
\end{equation}
where $ {\mathcal C}_8^2 := k_1  {\mathcal C}_6^2 + k_2  {\mathcal C}_7^2$.

Finally, estimates \eqref{3.20} and \eqref{3.38} yield the required inequality \eqref{3.25} with the constant
${\mathcal C}_5 := 2 C_2 {\mathcal C}_4 + {\mathcal C}_8$.
\end{proof}

Further plan of the proof of Theorems \ref{th3.1} and \ref{th3.2} is as follows. First, we shall prove estimate
\eqref{3.10} for $\textrm{Re}\,\zeta \le 0$. Next, we shall check estimate \eqref{3.2} also for $\textrm{Re}\,\zeta \le 0$,
using the already proved estimate  \eqref{3.10} and the duality arguments. After that we shall complete the proofs of theorems
by employing suitable identities for the resolvents, that allow us to transfer the already proven estimates from a point
$\zeta$ in the left half-plane to the symmetric point of the right one. (The last trick is borrowed from \cite[Section~10]{MSu2}.)

\textbf{Conclusions.} 1) By \eqref{3.25},
\begin{equation}
\label{3.39}
\begin{split}
\| \u_\eps - {\mathbf v}_\eps \|_{H^p(\O)}
\le {\mathcal C}_5 \left( \eps |\zeta|^{-1/2+1/2p} + \eps^p  \right) \|\FF\|_{L_2(\O)} + \| \w_\eps \|_{H^p(\O)},
\quad
\rm{Re}\,\zeta \le 0, \quad |\zeta| \ge 1.
\end{split}
\end{equation}
Hence, in order to prove estimate \eqref{3.10} (for $\textrm{Re}\,\zeta \le 0$), it suffices to prove an appropriate
estimate for the norm of $\w_\eps$ in $H^p(\O;\C^n)$.

2) From \eqref{3.19} and \eqref{3.37} it follows that
\begin{equation}
\label{3.40}
\begin{split}
\| \u_\eps - {\u}_0 \|_{L_2(\O)}
\le \wt{\mathcal C}_7 \left( \eps |\zeta|^{-1+1/2p} + \eps^p |\zeta|^{-1/2}  \right) \|\FF\|_{L_2(\O)}
+ \| \w_\eps \|_{L_2(\O)},
\quad
\textrm{Re}\,\zeta \le 0, \quad |\zeta| \ge 1,
\end{split}
\end{equation}
where $\wt{\mathcal C}_7 := {\mathcal C}_7 + C_1 {\mathcal C}_4$. Therefore, in order to prove Theorem \ref{th3.1} (for \hbox{$\textrm{Re}\,\zeta \le 0$}), it suffices to obtain an appropriate estimate for the norm of $\w_\eps$ in $L_2(\O;\C^n)$.

\section{Proof of Theorems \ref{th3.1} and \ref{th3.2}}

\subsection{The case where $\rm{Re}\,\zeta \le 0$. Estimate for the correction term $\w_\eps$ in $H^p(\O;\C^n)$}
Denote
\begin{equation}
\label{4.1}
{\mathcal J}_\eps[\eeta] :=
(\wt{g}^\eps S_\eps b(\D) \wt{\u}_0 - g^0 b(\D) {\u}_0,  b(\D) \eeta)_{L_2(\O)},
\quad \eeta \in H^p(\O;\C^n).
\end{equation}
The function $\u_0$ satisfies the identity
\begin{equation}
\label{4.2}
 ({g}^0 b(\D) {\u}_0,  b(\D) \eeta)_{L_2(\O)} -
(\zeta \u_0 + \FF, \eeta)_{L_2(\O)} =0,
\quad \eeta \in H^p(\O;\C^n).
\end{equation}
From \eqref{3.24} and \eqref{4.2} it follows that
\begin{equation}
\label{4.3}
 ({g}^\eps b(\D) {\w}_\eps,  b(\D) \eeta)_{L_2(\O)} -
\zeta (\w_\eps , \eeta)_{L_2(\O)} = {\mathcal J}_\eps[\eeta], \quad \eeta \in H^p(\O;\C^n).
\end{equation}

\begin{lemma}
\label{lem4.1}
Let $\rm{Re}\,\zeta \le 0$ and $|\zeta| \ge 1$. Suppose that the number $\eps_1$ is subject to Condition~\emph{\ref{cond_eps}}.
Then for $0< \eps \le \eps_1$ the functional \eqref{4.1} satisfies
\begin{equation}
\label{4.4}
|{\mathcal J}_\eps[\eeta]| \le
{\mathcal C}_{9} \left( \eps^{1/2} |\zeta|^{-1/2+1/4p} + \eps^p  \right) \|\FF\|_{L_2(\O)} \| \D^p \eeta \|_{L_2(\O)},
\quad \eeta \in H^p(\O;\C^n).
\end{equation}
The constant ${\mathcal C}_{9}$ depends only on $d$, $p$, $m$, $\alpha_0$, $\alpha_1$, $\|g\|_{L_\infty}$, $\|g^{-1}\|_{L_\infty}$, $k_1$, $k_2$,
the parameters of the lattice $\Gamma$, and the domain $\O$.
\end{lemma}

\begin{proof}
We represent the functional \eqref{4.1} as
\begin{align}
\label{4.5}
{\mathcal J}_\eps[\eeta] =& {\mathcal J}_\eps^{(1)}[\eeta] + {\mathcal J}_\eps^{(2)}[\eeta],
\\
\label{4.6}
{\mathcal J}_\eps^{(1)}[\eeta] :=&  ({g}^0 S_\eps b(\D) \wt{\u}_0 - g^0 b(\D) {\u}_0,  b(\D) \eeta)_{L_2(\O)},
\\
\label{4.7}
{\mathcal J}_\eps^{(2)}[\eeta] :=&  ((\wt{g}^\eps - g^0 ) S_\eps b(\D) \wt{\u}_0,  b(\D) \eeta)_{L_2(\O)}.
\end{align}

It is easy to estimate the term \eqref{4.6} with the help of Proposition \ref{prop1.4} and relations \eqref{1.3}, \eqref{1.4}, \eqref{1.5},
\eqref{1.17a}, and \eqref{3.15}:
\begin{equation}
\label{4.8}
\begin{split}
|{\mathcal J}_\eps^{(1)}[\eeta]| &\le |g^0| \| (S_\eps -I) b(\D) \wt{\u}_0\|_{L_2(\R^d)}  \|b(\D) \eeta\|_{L_2(\O)}
\\
&\le {\mathcal C}_{10}    \eps  |\zeta|^{-1/2+ 1/2p}
\|\FF\|_{L_2(\O)} \| \D^p \eeta \|_{L_2(\O)}, \quad \eeta \in H^p(\O;\C^n),
\end{split}
\end{equation}
where ${\mathcal C}_{10} := \|g\|_{L_\infty} r_1 \alpha_1 C^{(p+1)} \wt{\mathfrak c}^{1/2}_p$.

Using \eqref{1.3}, we transform the term \eqref{4.7}:
\begin{equation}
\label{4.9}
\begin{split}
{\mathcal J}_\eps^{(2)}[\eeta]
= \sum_{|\alpha|=p} (f_\alpha^\eps S_\eps b(\D) \wt{\u}_0, \D^\alpha \eeta )_{L_2(\O)},
\end{split}
\end{equation}
where  $f_\alpha(\x):= b_\alpha^* (\wt{g}(\x) - g^0)$, $|\alpha|=p$.
From \eqref{1.5}, \eqref{1.11}, \eqref{1.12}, and \eqref{Lambda2} it follows that
\begin{equation}
\label{4.10}
\begin{split}
\| f_\alpha \|_{L_2(\Omega)} \le {\mathfrak C}, \quad |\alpha|=p,\quad
{\mathfrak C} := \alpha_1^{1/2} \|g\|_{L_\infty} |\Omega|^{1/2} (C_\Lambda^{(1)}+1).
\end{split}
\end{equation}
Combining \eqref{1.3} and \eqref{1.10}--\eqref{1.12}, we see that the functions $f_\alpha(\x)$, $|\alpha|=p$,
satisfy the assumptions of Lemma~\ref{traditional}. Hence, there exist $\Gamma$-periodic matrix-valued functions
$M_{\alpha \beta}(\x)$, $|\alpha|=|\beta|=p$, satisfying relations
\eqref{trad2}--\eqref{trad4}.  From \eqref{trad4} and \eqref{4.10} it follows that
\begin{equation}
\label{4.11}
\begin{split}
\| M_{\alpha \beta} \|_{H^p(\Omega)} \le \check{\mathfrak C}, \quad \check{\mathfrak C}:= 2 \check{\mathfrak c}_p {\mathfrak C},
 \quad |\alpha|=|\beta|= p.
\end{split}
\end{equation}

By \eqref{trad3}, we have
$$
f^\eps_\alpha(\x) =\eps^p \sum_{|\beta|=p} \partial^\beta M^\eps_{\alpha \beta}(\x),\quad |\alpha|=p,
$$
whence the term  \eqref{4.9} can be represented as
\begin{align}
\label{4.12}
{\mathcal J}_\eps^{(2)}[\eeta] =
 \eps^p \sum_{|\alpha|=|\beta|=p} \left( (\partial^\beta M^\eps_{\alpha \beta}) S_\eps b(\D) \wt{\u}_0, \D^\alpha \eeta\right )_{L_2(\O)}.
\end{align}
We have
\begin{equation*}
\begin{split}
(\partial^\beta M^\eps_{\alpha \beta}) S_\eps b(\D) \wt{\u}_0 =
\partial^\beta \left(M^\eps_{\alpha \beta} S_\eps b(\D) \wt{\u}_0\right)
- \sum_{\gamma \le \beta: |\gamma|\ge 1} C_\beta^\gamma (\partial^{\beta-\gamma} M^\eps_{\alpha \beta}) S_\eps b(\D) \partial^\gamma \wt{\u}_0.
\end{split}
\end{equation*}
Consequently, the functional \eqref{4.12} can be written as the sum of two terms:
\begin{align}
\label{4.13}
{\mathcal J}_\eps^{(2)}[\eeta] =& \wt{{\mathcal J}}_\eps^{(2)}[\eeta] + \wh{\mathcal J}_\eps^{(2)}[\eeta],
\\
\label{4.14}
 \wt{{\mathcal J}}_\eps^{(2)}[\eeta] :=& \, \eps^p \sum_{|\alpha|=|\beta|=p}
\left( \partial^\beta \left( M^\eps_{\alpha \beta} S_\eps b(\D) \wt{\u}_0\right), \D^\alpha \eeta\right )_{L_2(\O)},
\\
\label{4.15}
 \wh{{\mathcal J}}_\eps^{(2)}[\eeta] :=& - \sum_{|\alpha|=|\beta|=p}
\sum_{\gamma \le \beta: |\gamma|\ge 1} \eps^{|\gamma|} C_\beta^\gamma \left(
(\partial^{\beta-\gamma} M_{\alpha \beta})^\eps
S_\eps b(\D)  \partial^\gamma \wt{\u}_0, \D^\alpha \eeta\right)_{L_2(\O)}.
\end{align}

The term  \eqref{4.15} is estimated by using Proposition \ref{prop1.5} and relations \eqref{1.4}, \eqref{4.11}:
\begin{equation*}
\begin{split}
 |\wh{{\mathcal J}}_\eps^{(2)}[\eeta]|
&\le c_1(d,p) \sum_{|\alpha|=|\beta|=p} \sum_{\gamma \le \beta: |\gamma|\ge 1} \eps^{|\gamma|}
|\Omega|^{-1/2} \| \partial^{\beta-\gamma} M_{\alpha \beta} \|_{L_2(\Omega)}
\|b(\D)  \partial^\gamma \wt{\u}_0 \|_{L_2(\R^d)} \| \D^\alpha \eeta \|_{L_2(\O)}
\\
&\le c_2(d,p)|\Omega|^{-1/2} \check{\mathfrak C} \alpha_1^{1/2} \sum_{l=1}^p \eps^l \| \wt{\u}_0\|_{H^{p+l}(\R^d)} \| \D^p \eeta\|_{L_2(\O)}.
\end{split}
\end{equation*}
Together with \eqref{3.15}, this implies
\begin{equation}
\label{4.16}
\begin{split}
 |\wh{{\mathcal J}}_\eps^{(2)}[\eeta]|
&\le {\mathcal C}_{11} \sum_{l=1}^p \eps^l |\zeta|^{-1/2 + l/2p} \| \FF \|_{L_2(\O)} \| \D^p \eeta\|_{L_2(\O)}
\\
&\le p \, {\mathcal C}_{11} \left(\eps |\zeta|^{-1/2 + 1/2p} + \eps^p\right) \| \FF \|_{L_2(\O)} \| \D^p \eeta\|_{L_2(\O)}.
\end{split}
\end{equation}
Here ${\mathcal C}_{11} := c_2(d,p)|\Omega|^{-1/2} \check{\mathfrak C} \alpha_1^{1/2} \max_{1\le l \le p} \{C^{(p+l)}\}$.

Now, we consider the term \eqref{4.14}.
We fix a smooth cut-off function $\theta_\eps$ in $\R^d$ such that
\begin{equation}
\label{4.100}
\begin{aligned}
&\theta_\eps \in C_0^\infty(\R^d),\quad {\rm supp}\,\theta_\eps \subset (\partial \O)_\eps, \quad
0 \leqslant \theta_\eps(\x) \leqslant 1,
\cr
&\theta_\eps(\x) = 1 \ \text{for}\ \x \in (\partial \O)_{\eps/2}, \quad \eps^l |\D^l \theta_\eps(\x)| \leqslant
\varkappa,\ l=1,\dots,p.
\end{aligned}
\end{equation}
The constant $\varkappa$ depends only on the domain $\O$. We have
 $$
\sum_{|\alpha|=|\beta|=p}
\left( \partial^\beta \left( (1-\theta_\eps) M^\eps_{\alpha \beta} S_\eps b(\D) \wt{\u}_0\right), \D^\alpha \eeta \right )_{L_2(\O)}=0,
\quad \eeta \in H^p(\O;\C^n),
$$
which can be easily checked by integration by parts
and taking the relations $M_{\alpha \beta} = - M_{\beta \alpha}$ into account (see \eqref{trad2}).
Hence, the term \eqref{4.14} can be written as
\begin{equation}
\label{4.18}
 \wt{{\mathcal J}}_\eps^{(2)}[\eeta] = \sum_{|\alpha|=p} \left( {\boldsymbol \psi}_\alpha(\eps), \D^\alpha \eeta\right )_{L_2(\O)},
\end{equation}
where
\begin{equation}
\label{4.19}
  {\boldsymbol \psi}_\alpha(\eps) := \eps^p \sum_{|\beta|=p}  \partial^\beta \left( \theta_\eps M^\eps_{\alpha \beta} S_\eps b(\D) \wt{\u}_0\right),
\quad |\alpha|=p.
\end{equation}
Next, we have
\begin{equation}
\label{4.20}
  {\boldsymbol \psi}_\alpha(\eps) = \eps^p \sum_{|\beta|=p}  \sum_{\gamma \le \beta} \sum_{\nu \le \beta - \gamma}
C_\beta^\gamma C_{\beta - \gamma}^\nu (\partial^\gamma \theta_\eps) (\partial^{\beta - \gamma -\nu} M^\eps_{\alpha \beta})
 S_\eps b(\D) \partial^\nu \wt{\u}_0.
\end{equation}
For $k=|\nu|\ge 1$, we use Proposition \ref{prop1.5} and relations \eqref{1.4}, \eqref{3.15}, \eqref{4.11}, and \eqref{4.100}:
\begin{equation}
\label{4.21}
\begin{split}
&\eps^p \| (\partial^\gamma \theta_\eps) (\partial^{\beta - \gamma -\nu} M^\eps_{\alpha \beta}) S_\eps b(\D) \partial^\nu \wt{\u}_0 \|_{L_2(\R^d)}
\\
&\le \varkappa \eps^k \| (\partial^{\beta - \gamma -\nu} M_{\alpha \beta})^\eps S_\eps b(\D) \partial^\nu \wt{\u}_0 \|_{L_2(\R^d)}
\\
&\le \varkappa \eps^k |\Omega|^{-1/2} \check{\mathfrak C} \alpha_1^{1/2} \| \wt{\u}_0 \|_{H^{p+k}(\R^d)}
\\
&\le {\mathcal C}^{(k)} \eps^k |\zeta|^{-1/2 + k/2p} \| \FF \|_{L_2(\O)},
\end{split}
\end{equation}
where ${\mathcal C}^{(k)} := \varkappa |\Omega|^{-1/2} \check{\mathfrak C} \alpha_1^{1/2} C^{(p+k)}$.

For $\nu=0$, we apply Lemma \ref{lem_dO2}. Let $0< \eps \le \eps_1$. By \eqref{4.100}, we have
\begin{equation*}
\begin{split}
&\eps^p \| (\partial^\gamma \theta_\eps) (\partial^{\beta - \gamma} M^\eps_{\alpha \beta}) S_\eps b(\D) \wt{\u}_0 \|_{L_2(\R^d)}
\\
&\le \varkappa \| (\partial^{\beta - \gamma} M_{\alpha \beta})^\eps S_\eps b(\D) \wt{\u}_0 \|_{L_2((\partial \O)_\eps)}
\\
&\le \eps^{1/2} \varkappa \beta_*^{1/2} |\Omega|^{-1/2} \|\partial^{\beta - \gamma} M_{\alpha \beta} \|_{L_2(\Omega)}
 \|b(\D) \wt{\u}_0 \|^{1/2}_{H^1(\R^d)} \|b(\D) \wt{\u}_0 \|^{1/2}_{L_2(\R^d)}.
\end{split}
\end{equation*}
Together with \eqref{1.4}, \eqref{3.15}, and \eqref{4.11}, this implies
\begin{equation}
\label{4.22}
\begin{split}
\eps^p \| (\partial^\gamma \theta_\eps) (\partial^{\beta - \gamma} M^\eps_{\alpha \beta}) S_\eps b(\D) \wt{\u}_0 \|_{L_2(\R^d)}
\le {\mathcal C}_{12} \eps^{1/2} |\zeta|^{-1/2 + 1/4p} \| \FF \|_{L_2(\O)},
\end{split}
\end{equation}
where ${\mathcal C}_{12} := \varkappa \beta_*^{1/2} |\Omega|^{-1/2} \check{\mathfrak C} \alpha_1^{1/2} (C^{(p+1)} C^{(p)})^{1/2}$.

Estimating the summands in \eqref{4.20} with $k=|\nu| \ge 1$ by \eqref{4.21}, and the summands with $\nu=0$ with the help of \eqref{4.22},
we arrive at the inequality
\begin{equation*}
\| {\boldsymbol \psi}_\alpha(\eps) \|_{L_2(\R^d)} \le {\mathcal C}_{13} \left( \eps^{1/2} |\zeta|^{-1/2 + 1/4p} +
\sum_{k=1}^p \eps^{k} |\zeta|^{-1/2 + k/2p} \right) \| \FF \|_{L_2(\O)},
\end{equation*}
where ${\mathcal C}_{13} := c_3(d,p) \max\{ {\mathcal C}_{12}, {\mathcal C}^{(1)},\dots, {\mathcal C}^{(p)}\}$.
It is easily seen that expression in the brackets does not exceed
$(p+1) (\eps^{1/2} |\zeta|^{-1/2 + 1/4p} + \eps^p)$. Hence,
\begin{equation}
\label{4.22a}
\| {\boldsymbol \psi}_\alpha(\eps) \|_{L_2(\R^d)} \le (p+1) \,{\mathcal C}_{13} \left( \eps^{1/2} |\zeta|^{-1/2 + 1/4p} +
 \eps^{p} \right) \| \FF \|_{L_2(\O)}, \quad |\alpha|=p.
\end{equation}
Together with \eqref{4.18} this yields
\begin{equation}
\label{4.23}
| \wt{{\mathcal J}}_\eps^{(2)}[\eeta]| \le {\mathcal C}_{14} \left(\eps^{1/2} |\zeta|^{-1/2 + 1/4p} + \eps^p\right) \| \FF \|_{L_2(\O)} \| \D^p \eeta\|_{L_2(\O)}
\end{equation}
with the constant ${\mathcal C}_{14} := c_4(d,p) {\mathcal C}_{13}$.

As a result, relations \eqref{4.5}, \eqref{4.8}, \eqref{4.13}, \eqref{4.16}, and \eqref{4.23} imply the required estimate \eqref{4.4} with the
constant ${\mathcal C}_{9} := {\mathcal C}_{10} + 2 p \,{\mathcal C}_{11} + {\mathcal C}_{14}$.
\end{proof}

\begin{lemma}
Let $\rm{Re}\,\zeta \le 0$ and $|\zeta| \ge 1$. Suppose that the number $\eps_1$ is subject to Condition \emph{\ref{cond_eps}}.
Suppose that the function $\w_\eps \in H^p(\O;\C^n)$ satisfies \eqref{3.24}. Then for $0< \eps \le \eps_1$ we have
\begin{equation}
\label{4.24}
\| \w_\eps \|_{H^p(\O)} \le
{\mathcal C}_{15} \left( \eps^{1/2} |\zeta|^{-1/2+1/4p} + \eps^p  \right) \|\FF\|_{L_2(\O)}.
\end{equation}
The constant ${\mathcal C}_{15}$ depends only on $d$, $p$, $m$, $\alpha_0$, $\alpha_1$, $\|g\|_{L_\infty}$, $\|g^{-1}\|_{L_\infty}$, $k_1$, $k_2$,
the parameters of the lattice $\Gamma$, and the domain $\O$.
\end{lemma}

\begin{proof}
 We substitute $\eeta = \w_\eps$ in \eqref{4.3} and consider the imaginary part of the corresponding
identity. Taking \eqref{4.4} into account, we obtain
\begin{equation}
\label{4.25}
\begin{split}
|{\textrm{Im}}\,\zeta|\| \w_\eps \|^2_{L_2(\O)} \le | {\mathcal J}_\eps [\w_\eps]|
\le {\mathcal C}_{9} \left( \eps^{1/2} |\zeta|^{-1/2+1/4p} + \eps^p  \right) \|\FF\|_{L_2(\O)} \| \D^p \w_\eps\|_{L_2(\O)}.
\end{split}
\end{equation}
Now we consider the real part of the corresponding identity, using that ${\rm{Re}}\,\zeta \le 0$. It follows that
\begin{equation}
\label{4.26}
\begin{split}
| {\textrm{Re}}\, \zeta|\| \w_\eps \|^2_{L_2(\O)}
\le {\mathcal C}_{9} \left( \eps^{1/2} |\zeta|^{-1/2+1/4p} + \eps^p  \right) \|\FF\|_{L_2(\O)} \| \D^p \w_\eps\|_{L_2(\O)}.
\end{split}
\end{equation}
Summing \eqref{4.25} and \eqref{4.26} up, we arrive at
\begin{equation}
\label{4.27}
\begin{split}
| \zeta|\| \w_\eps \|^2_{L_2(\O)}
\le 2 {\mathcal C}_{9} \left( \eps^{1/2} |\zeta|^{-1/2+1/4p} + \eps^p  \right) \|\FF\|_{L_2(\O)} \| \D^p \w_\eps\|_{L_2(\O)}.
\end{split}
\end{equation}

On the other hand, relations \eqref{4.3} (with $\eeta = \w_\eps$) and \eqref{4.4} yield
\begin{equation}
\label{4.28}
\begin{split}
 a_{N,\eps}[ \w_\eps, \w_\eps] \le
 {\mathcal C}_{9} \left( \eps^{1/2} |\zeta|^{-1/2+1/4p} + \eps^p  \right) \|\FF\|_{L_2(\O)} \| \D^p \w_\eps\|_{L_2(\O)}.
\end{split}
\end{equation}
Comparing \eqref{22.5}, \eqref{4.27}, and \eqref{4.28}, and taking into account that $|\zeta| \ge 1$, we obtain
\begin{equation}
\label{4.28a}
 \| \w_\eps \|^2_{H^p(\O)} \le
{\mathcal C}_{15} \left( \eps^{1/2} |\zeta|^{-1/2+1/4p} + \eps^p  \right) \|\FF\|_{L_2(\O)} \| \D^p \w_\eps\|_{L_2(\O)},
\end{equation}
where ${\mathcal C}_{15} : = {\mathcal C}_{9}(k_1 \|g^{-1}\|_{L_\infty} + 2 k_2)$.
This implies the required estimate \eqref{4.24}.
\end{proof}

\subsection{Completion of the proof of estimate \eqref{3.10} for $\rm{Re}\,\zeta \le 0$}

By \eqref{3.39} and \eqref{4.24}, for $\rm{Re}\,\zeta \le 0$, $|\zeta| \ge 1$, and $0< \eps \le \eps_1$ we have
\begin{equation*}
\begin{split}
&\| \u_\eps - {\mathbf v}_\eps \|_{H^p(\O)}
\le {\mathcal C}_5 \left( \eps |\zeta|^{-1/2+1/2p} + \eps^p  \right) \|\FF\|_{L_2(\O)}
+ {\mathcal C}_{15} \left( \eps^{1/2} |\zeta|^{-1/2+1/4p} + \eps^p  \right) \|\FF\|_{L_2(\O)}
\\
&\le {\mathcal C}_{2}' \left( \eps^{1/2} |\zeta|^{-1/2+1/4p} + \eps^p  \right) \|\FF\|_{L_2(\O)},
\quad \rm{Re}\,\zeta \le 0, \quad |\zeta| \ge 1,
\end{split}
\end{equation*}
where ${\mathcal C}_2' := 2 {\mathcal C}_{5} + {\mathcal C}_{15}$. This implies estimate \eqref{3.10} in the case under consideration:
\begin{equation}
\label{4.30}
\begin{split}
&\| (A_{N,\eps} - \zeta I)^{-1} - (A_{N}^0 - \zeta I)^{-1} - \eps^p K_{N}(\zeta;\eps) \|_{L_2(\O) \to H^p(\O)}
\\
&\le {\mathcal C}_{2}' \left( \eps^{1/2} |\zeta|^{-1/2+1/4p} + \eps^p  \right),
\quad \rm{Re}\,\zeta \le 0, \quad |\zeta| \ge 1,\quad 0< \eps \le \eps_1.
\end{split}
\end{equation}

\subsection{Estimate of the correction term $\w_\eps$ in $L_2$}

\begin{lemma}
Let $\rm{Re}\,\zeta \le 0$ and $|\zeta| \ge 1$. Suppose that the number $\eps_1$ is subject to Condition~\emph{\ref{cond_eps}}.
Suppose that the function $\w_\eps \in H^p(\O;\C^n)$ satisfies~\eqref{3.24}. Then for $0< \eps \le \eps_1$ we have
\begin{equation}
\label{4.31}
\| \w_\eps \|_{L_2(\O)} \le
{\mathcal C}_{16} \left( \eps |\zeta|^{-1+1/2p} + \eps^{2p}  \right) \|\FF\|_{L_2(\O)}.
\end{equation}
The constant ${\mathcal C}_{16}$ depends only on $d$, $p$, $m$, $\alpha_0$, $\alpha_1$, $\|g\|_{L_\infty}$, $\|g^{-1}\|_{L_\infty}$, $k_1$, $k_2$,
the parameters of the lattice $\Gamma$, and the domain $\O$.
\end{lemma}

\begin{proof}
In the identity \eqref{4.3}, we substitute the function
$$
\eeta = \eeta_\eps = (A_{N,\eps} - \overline{\zeta} I)^{-1} \bPhi, \quad \bPhi \in L_2(\O;\C^n).
$$
Then the left-hand side of \eqref{4.3} can be written as $(\w_\eps, \bPhi)_{L_2(\O)}$. Hence,
\begin{equation}
\label{4.32}
(\w_\eps, \bPhi)_{L_2(\O)} = {\mathcal J}_\eps[\eeta_\eps].
\end{equation}

Assume that $0< \eps \le \eps_1$. By \eqref{4.5}, \eqref{4.8}, \eqref{4.13}, and \eqref{4.16},
\begin{equation*}
|{\mathcal J}_\eps[\eeta_\eps]| \le ({\mathcal C}_{10}+ p\, {\mathcal C}_{11}) \left( \eps |\zeta|^{-1/2+1/2p} + \eps^p  \right) \|\FF\|_{L_2(\O)} \| \D^p \eeta_\eps\|_{L_2(\O)}
+ |\wt{\mathcal J}_\eps^{(2)}[\eeta_\eps]|.
\end{equation*}
Applying Lemma \ref{lem1} to estimate the term $\| \D^p \eeta_\eps\|_{L_2(\O)}$, we arrive at
\begin{equation}
\label{4.34}
|{\mathcal J}_\eps[\eeta_\eps]| \le {\mathcal C}_{17} \left( \eps |\zeta|^{-1 +1/2p} + \eps^{2p} \right) \|\FF\|_{L_2(\O)} \| \bPhi \|_{L_2(\O)}
+ |\wt{\mathcal J}_\eps^{(2)}[\eeta_\eps]|,
\end{equation}
where ${\mathcal C}_{17} := 2 {\mathcal C}_0 ({\mathcal C}_{10}+ p \, {\mathcal C}_{11})$.

 Now, we consider the term $|\wt{\mathcal J}_\eps^{(2)}[\eeta_\eps]|$, using  representation \eqref{4.18}.
We apply the already proved estimate \eqref{4.30} (at the point $\overline{\zeta}$) in order to approximate the function
$\eeta_\eps$. Let ${\eeta}_0 := (A_{N}^0 - \overline{\zeta} I)^{-1} \bPhi$ and $\wt{\eeta}_0 := P_\O \eeta_0$. We have
\begin{equation}
\label{4.35}
\| \eeta_\eps - \eeta_0 - \eps^p \Lambda^\eps S_\eps b(\D) \wt{\eeta}_0 \|_{H^p(\O)} \le
{\mathcal C}_{2}' \left( \eps^{1/2} |\zeta|^{-1/2+1/4p} + \eps^p  \right) \| \bPhi\|_{L_2(\O)}.
\end{equation}
According to \eqref{4.18},
\begin{align}
\label{4.36}
\wt{\mathcal J}_\eps^{(2)}[\eeta_\eps] & = {\mathcal L}_1(\eps)+ {\mathcal L}_2(\eps) + {\mathcal L}_3(\eps),
\\
\nonumber
{\mathcal L}_1(\eps) & := \sum_{|\alpha|=p} \left( {\boldsymbol \psi}_\alpha(\eps), \D^\alpha (\eeta_\eps - \eeta_0 - \eps^p \Lambda^\eps S_\eps b(\D) \wt{\eeta}_0 ) \right )_{L_2(\O)},
\\
\label{4.38}
{\mathcal L}_2(\eps) & := \sum_{|\alpha|=p} \left( {\boldsymbol \psi}_\alpha(\eps), \D^\alpha \eeta_0 \right )_{L_2(\O)},
\\
\label{4.39}
{\mathcal L}_3(\eps) & := \sum_{|\alpha|=p} \left( {\boldsymbol \psi}_\alpha(\eps), \D^\alpha (\eps^p \Lambda^\eps S_\eps b(\D) \wt{\eeta}_0 ) \right )_{L_2(\O)}.
\end{align}

From \eqref{4.22a} and \eqref{4.35} it follows that
\begin{equation}
\label{4.40}
|{\mathcal L}_1(\eps)| \le {\mathcal C}_{18} \left( \eps |\zeta|^{-1+1/2p} + \eps^{2p}  \right) \| \FF \|_{L_2(\O)} \| \bPhi\|_{L_2(\O)},
\end{equation}
where ${\mathcal C}_{18} := 2 c_4(d,p) {\mathcal C}_{2}' {\mathcal C}_{13}$.

By \eqref{4.100}, \eqref{4.19}, and \eqref{4.22a}, the term \eqref{4.38} satisfies
\begin{equation*}
|{\mathcal L}_2(\eps)| \le c_4(d,p) {\mathcal C}_{13}
 \left( \eps^{1/2} |\zeta|^{-1/2+1/4p} + \eps^{p}  \right) \| \FF \|_{L_2(\O)} \left(\int_{(\partial \O)_\eps} |\D^p \wt{\eeta}_0|^2 \,d\x \right)^{1/2}.
\end{equation*}
Combining this with Lemma \ref{lem_dO1} and estimates \eqref{3.13}, \eqref{3.15} (for $\wt{\eeta}_0$), we deduce
\begin{equation}
\label{4.42}
|{\mathcal L}_2(\eps)| \le {\mathcal C}_{19}
\left( \eps |\zeta|^{-1+1/2p} + \eps^{2p}  \right) \| \FF \|_{L_2(\O)} \| \bPhi\|_{L_2(\O)},
\end{equation}
where ${\mathcal C}_{19} := 2 c_4(d,p) {\mathcal C}_{13} \beta_0^{1/2} (C^{(p)} C^{(p+1)})^{1/2}$.

It remains to estimate the term \eqref{4.39}, which can be written as
\begin{align}
\label{4.43}
{\mathcal L}_3(\eps) &= {\mathcal L}_3^{(1)}(\eps) + {\mathcal L}_3^{(2)}(\eps),
\\
\nonumber
{\mathcal L}_3^{(1)}(\eps) & := \sum_{|\alpha|=p} \left( {\boldsymbol \psi}_\alpha(\eps), (\D^\alpha \Lambda)^\eps S_\eps b(\D) \wt{\eeta}_0  \right )_{L_2(\O)},
\\
\nonumber
{\mathcal L}_3^{(2)}(\eps) & := \sum_{|\alpha|=p} \sum_{\beta \le \alpha: |\beta|\ge 1}
C_\alpha^\beta \eps^{|\beta|} \left( {\boldsymbol \psi}_\alpha(\eps), (\D^{\alpha - \beta} \Lambda)^\eps S_\eps b(\D) \D^\beta \wt{\eeta}_0 \right )_{L_2(\O)}.
\end{align}
By \eqref{4.100}, \eqref{4.19}, and \eqref{4.22a},
\begin{equation*}
| {\mathcal L}_3^{(1)}(\eps) | \le c_4(d,p) {\mathcal C}_{13}\left( \eps^{1/2} |\zeta|^{-1/2+1/4p} + \eps^{p}  \right) \| \FF \|_{L_2(\O)}
\left(\int_{(\partial \O)_\eps} |(\D^p \Lambda )^\eps S_\eps b(\D) \wt{\eeta}_0|^2 \,d\x \right)^{1/2}.
\end{equation*}
Applying Lemma \ref{lem_dO2}, relations \eqref{1.4}, \eqref{Lambda3}, and the analogs of estimates \eqref{3.13},
\eqref{3.15} for $\wt{\eeta}_0$, we arrive at
\begin{equation}
\label{4.47}
|{\mathcal L}_3^{(1)}(\eps)| \le {\mathcal C}_{20}
 \left( \eps |\zeta|^{-1+1/2p} + \eps^{2p}  \right) \| \FF \|_{L_2(\O)} \| \bPhi \|_{L_2(\O)},
\end{equation}
where  ${\mathcal C}_{20} := 2 c_4(d,p) C_\Lambda {\mathcal C}_{13} (\beta_* \alpha_1 C^{(p)} C^{(p+1)})^{1/2}$.

Next, by Proposition \ref{prop1.5}, \eqref{1.4}, \eqref{Lambda3}, and \eqref{4.22a}, we have
\begin{equation*}
|{\mathcal L}_3^{(2)}(\eps)| \le c_5(d,p) {\mathcal C}_{13}
 \left( \eps^{1/2} |\zeta|^{-1/2+1/4p} + \eps^{p}  \right) \| \FF \|_{L_2(\O)}
 \alpha_1^{1/2} C_\Lambda \sum_{k=1}^p \eps^k \| \wt{\eeta}_0 \|_{H^{p+k}(\R^d)}.
\end{equation*}
 Together with the analog of \eqref{3.15} for  $\wt{\eeta}_0$, this yields
\begin{equation*}
|{\mathcal L}_3^{(2)}(\eps)| \le {\mathcal C}_{21}
 \left( \eps^{1/2} |\zeta|^{-1/2+1/4p} + \eps^{p}  \right)
 \Bigl( \sum_{k=1}^p \eps^k |\zeta|^{-1/2 + k/2p} \Bigr) \| \FF \|_{L_2(\O)} \| \bPhi \|_{L_2(\O)},
\end{equation*}
where  ${\mathcal C}_{21} := c_5(d,p) {\mathcal C}_{13} \alpha_1^{1/2} C_\Lambda \max \{ C^{(p+1)},\dots, C^{(2p)} \}$.
It follows that
\begin{equation}
\label{4.49}
|{\mathcal L}_3^{(2)}(\eps)| \le 2p\, {\mathcal C}_{21}
 \left( \eps |\zeta|^{-1+1/2p} + \eps^{2p}  \right) \| \FF \|_{L_2(\O)} \| \bPhi \|_{L_2(\O)}.
\end{equation}

As a result, relations \eqref{4.36} and \eqref{4.40}--\eqref{4.49} imply that
\begin{equation}
\label{4.50}
|\wt{\mathcal J}_\eps^{(2)}[\eeta_\eps]| \le {\mathcal C}_{22}
 \left( \eps |\zeta|^{-1+1/2p} + \eps^{2p}  \right) \| \FF \|_{L_2(\O)} \| \bPhi \|_{L_2(\O)},
\end{equation}
where ${\mathcal C}_{22} := {\mathcal C}_{18} + {\mathcal C}_{19} + {\mathcal C}_{20} + 2p\,{\mathcal C}_{21}$.
Now, inequalities \eqref{4.34} and \eqref{4.50} lead to the estimate
$$
|{\mathcal J}_\eps[\eeta_\eps]| \le {\mathcal C}_{16}
 \left( \eps |\zeta|^{-1+1/2p} + \eps^{2p}  \right) \| \FF \|_{L_2(\O)} \| \bPhi \|_{L_2(\O)},
$$
where ${\mathcal C}_{16} := {\mathcal C}_{17} + {\mathcal C}_{22}$. Combining this with \eqref{4.32}, we obtain the required
estimate \eqref{4.31}.
\end{proof}

\subsection{Completion of the proof of estimate \eqref{3.2} for $\rm{Re}\,\zeta \le 0$}
Let $\rm{Re}\,\zeta \le 0$ and $|\zeta|\ge 1$. By \eqref{3.40} and \eqref{4.31},
\begin{equation*}
\begin{split}
\| \u_\eps - \u_0 \|_{L_2(\O)}  &\le \wt{\mathcal C}_{7}
 \left( \eps |\zeta|^{-1+1/2p} + \eps^{p} |\zeta|^{-1/2}  \right) \| \FF \|_{L_2(\O)} +
{\mathcal C}_{16}  \left( \eps |\zeta|^{-1+1/2p} + \eps^{2p}  \right) \| \FF \|_{L_2(\O)}
\\
&\le {\mathcal C}_{23}  \left( \eps |\zeta|^{-1+1/2p} + \eps^{2p}  \right) \| \FF \|_{L_2(\O)},
\end{split}
\end{equation*}
where ${\mathcal C}_{23} := 2 \wt{\mathcal C}_{7} + {\mathcal C}_{16}$. In operator terms, this means
\begin{equation}
\label{4.51}
\begin{split}
\| (A_{N,\eps} - \zeta I)^{-1} - (A_{N}^0 - \zeta I)^{-1}  \|_{L_2(\O)\to L_2(\O)} \le
{\mathcal C}_{23}  \left( \eps |\zeta|^{-1+1/2p} + \eps^{2p}  \right).
\end{split}
\end{equation}

If $\eps \le |\zeta|^{-1/2p}$, then $ \eps^{2p} \le  \eps |\zeta|^{-1+1/2p}$, whence the right-hand side of \eqref{4.51} does not exceed
$2 {\mathcal C}_{23}  \eps |\zeta|^{-1+1/2p}$. In the case where $\eps > |\zeta|^{-1/2p}$, we use \eqref{22.8} and \eqref{22.15}.
Then the left-hand side of \eqref{4.51} does not exceed $2|\zeta|^{-1} \le  2\eps |\zeta|^{-1+1/2p}$.
As a result, we obtain the required estimate
\begin{equation}
\label{4.52}
\begin{split}
\| (A_{N,\eps} - \zeta I)^{-1} - (A_{N}^0 - \zeta I)^{-1}  \|_{L_2(\O)\to L_2(\O)} \le
{\mathcal C}_{1}' \eps |\zeta|^{-1+1/2p},
\quad
\rm{Re}\,\zeta \le 0, \ |\zeta|\ge 1,
\end{split}
\end{equation}
where ${\mathcal C}_{1}' := 2 \max\{ {\mathcal C}_{23}, 1\}$.

\subsection{The case where $\rm{Re}\,\zeta > 0$. Completion of the proof of Theorem~\ref{th3.1}}
Now, assume that $\zeta \in \C \setminus \R_+$, $\rm{Re}\,\zeta > 0$, $|\zeta| \ge 1$. Denote
$\wh{\zeta} = - {\rm{Re}}\,\zeta + i {\rm{Im}}\,\zeta$. Then $|\wh{\zeta}|=|\zeta|$. We have
\begin{equation}
\label{4.53}
\begin{aligned}
&(A_{N,\eps} - \zeta I)^{-1} - (A_{N}^0 - \zeta I)^{-1} =
(A_{N,\eps} - \widehat{\zeta} I) (A_{N,\eps} - \zeta I)^{-1}
\\
&\quad\times \left((A_{N,\eps} - \widehat{\zeta} I)^{-1} - (A_{N}^0 - \widehat{\zeta} I)^{-1}\right)
(A_{N}^0 - \widehat{\zeta} I) (A_{N}^0 - \zeta I)^{-1}.
\end{aligned}
\end{equation}

By \eqref{4.52} (at the point $\wh{\zeta}$) and \eqref{4.53},
\begin{equation}
\label{4.54}
\| (A_{N,\eps} - \zeta I)^{-1} - (A_{N}^0 - \zeta I)^{-1} \|_{L_2(\O)\to L_2(\O)} \le
{\mathcal C}_{1}' \eps |\zeta|^{-1+1/2p} \sup_{x \ge 0} \frac{|x - \wh{\zeta}|^2}{|x - \zeta|^2}.
\end{equation}
A calculation shows that
\begin{equation}
\label{4.55}
\sup_{x\ge 0} \frac{|x- \widehat{\zeta}|^2}{|x - \zeta|^2} \le 4 c(\varphi)^2.
\end{equation}
As a result, \eqref{4.54} and \eqref{4.55} imply estimate \eqref{3.2}  with the constant
${\mathcal C}_1 := 4 {\mathcal C}_1'$. This completes the proof of Theorem \ref{th3.1}. $\square$

\subsection{Completion of the proof of Theorem \ref{th3.2}}
Suppose that $\zeta \in \C \setminus \R_+$, $\rm{Re}\,\zeta > 0$, $|\zeta| \ge 1$. Let
$\wh{\zeta} = - {\rm{Re}}\,\zeta + i {\rm{Im}}\,\zeta$. We have
\begin{equation}
\label{4.56}
\begin{aligned}
&(A_{N,\eps} - \zeta I)^{-1} - (A_{N}^0 - \zeta I)^{-1} - \eps^p K_N(\zeta;\eps)
\\
&= \left( (A_{N,\eps} - \wh{\zeta} I)^{-1} - (A_{N}^0 - \wh{\zeta} I)^{-1} - \eps^p K_N(\wh{\zeta};\eps)\right)
(A_{N}^0 - \widehat{\zeta} I) (A_{N}^0 - \zeta I)^{-1}
\\
&+ (\zeta - \wh{\zeta}) (A_{N,\eps} - \zeta I)^{-1}
\left( (A_{N,\eps} - \wh{\zeta} I)^{-1} - (A_{N}^0 - \wh{\zeta} I)^{-1} \right) (A_{N}^0 - \widehat{\zeta} I) (A_{N}^0 - \zeta I)^{-1}.
\end{aligned}
\end{equation}

Hence,
\begin{equation*}
\begin{aligned}
&\| (A_{N,\eps} - \zeta I)^{-1} - (A_{N}^0 - \zeta I)^{-1} - \eps^p K_N(\zeta;\eps) \|_{L_2(\O) \to H^p(\O)}
\\
&\le
\|  (A_{N,\eps} - \wh{\zeta} I)^{-1} - (A_{N}^0 - \wh{\zeta} I)^{-1} - \eps^p K_N(\wh{\zeta};\eps)\|_{L_2(\O) \to H^p(\O)}
\\
&\times \| (A_{N}^0 - \widehat{\zeta} I) (A_{N}^0 - \zeta I)^{-1}\|_{L_2(\O) \to L_2(\O)} +
2 ({\rm{Re}}\,\zeta)  \|(A_{N,\eps} - \zeta I)^{-1}\|_{L_2(\O) \to H^p(\O)}
\\
&\times \|(A_{N,\eps} - \wh{\zeta} I)^{-1} - (A_{N}^0 - \wh{\zeta} I)^{-1}\|_{L_2(\O) \to L_2(\O)}
\| (A_{N}^0 - \widehat{\zeta} I) (A_{N}^0 - \zeta I)^{-1}\|_{L_2(\O) \to L_2(\O)}.
\end{aligned}
\end{equation*}
Combining this with \eqref{22.9}, estimates \eqref{4.30} and \eqref{4.52} (at the point $\wh{\zeta}$), and \eqref{4.55}, we obtain
\begin{equation*}
\begin{aligned}
&\| (A_{N,\eps} - \zeta I)^{-1} - (A_{N}^0 - \zeta I)^{-1} - \eps^p K_N(\zeta;\eps) \|_{L_2(\O) \to H^p(\O)}
\\
&\le 2 c(\varphi) {\mathcal C}_{2}' \left(\eps^{1/2} |\zeta|^{-1/2+1/4p} + \eps^p\right)
 + 4 ({\rm{Re}}\,\zeta) {\mathcal C}_{0}  {\mathcal C}_{1}' c(\varphi)^2 |\zeta|^{-1/2} \eps |\zeta|^{-1+1/2p}
  \\
& \le {\mathcal C}_{2}\left( c(\varphi) \eps^{1/2} |\zeta|^{-1/2+1/4p}
+  c(\varphi)^2 \eps |\zeta|^{-1/2+1/2p} + c(\varphi) \eps^p\right),
\end{aligned}
\end{equation*}
where ${\mathcal C}_{2} := 2 {\mathcal C}_{2}' + 4 {\mathcal C}_{0} {\mathcal C}_{1}'$.
This proves estimate \eqref{3.10}.

It remains to check \eqref{3.11}. From  \eqref{1.3}, \eqref{1.5}, and \eqref{3.9} it follows that
\begin{equation}
\label{4.59}
\begin{aligned}
\| \p_\eps - g^\eps b(\D) {\mathbf v}_\eps \|_{L_2(\O)} \le
{\mathcal C}_{24} \left( c(\varphi)  \eps^{1/2} |\zeta|^{-1/2+1/4p} + c(\varphi)^2 \eps |\zeta|^{-1/2 +1/2p} + c(\varphi) \eps^p \right)
\| \FF \|_{L_2(\O)}
\end{aligned}
\end{equation}
for $0< \eps \le \eps_1$, where
${\mathcal C}_{24} : = c_6(d,p) \| g\|_{L_\infty} \alpha_1^{1/2} {\mathcal C}_{2}$.
Relations \eqref{1.3}, \eqref{3.6}, and \eqref{3.7} imply that
\begin{equation}
\label{4.60}
\begin{aligned}
g^\eps b(\D) {\mathbf v}_\eps
&= g^\eps b(\D) \u_0 + g^\eps (b(\D) \Lambda)^\eps S_\eps b(\D) \wt{\u}_0
\\
&+ \sum_{|\alpha|=p} \sum_{\beta \leqslant \alpha: |\beta|\geqslant 1} g^\eps b_\alpha C_\alpha^\beta \eps^{|\beta|} (\D^{\alpha - \beta} \Lambda)^\eps S_\eps b(\D) \D^\beta \wt{\u}_0.
\end{aligned}
\end{equation}
By Proposition \ref{prop1.4} and \eqref{1.4},
\begin{equation}
\label{4.61}
\begin{aligned}
\| g^\eps b(\D) \u_0 - g^\eps S_\eps b(\D) \wt{\u}_0 \|_{L_2(\O)}
&\leqslant \| g^\eps (I-S_\eps)  b(\D) \wt{\u}_0 \|_{L_2(\R^d)}
\\
&\leqslant \eps \| g\|_{L_\infty} r_1 \alpha_1^{1/2} \| \wt{\u}_0 \|_{H^{p+1}(\R^d)}.
\end{aligned}
\end{equation}
The third term in the right-hand side of \eqref{4.60} is estimated by employing Proposition \ref{prop1.5} and relations
\eqref{1.4}, \eqref{1.5}, and \eqref{Lambda3}:
\begin{equation}
\label{4.62}
\begin{aligned}
\bigl\| \sum_{|\alpha|=p} \sum_{\beta \leqslant \alpha: |\beta|\geqslant 1} g^\eps b_\alpha C_\alpha^\beta \eps^{|\beta|} (\D^{\alpha - \beta} \Lambda)^\eps S_\eps b(\D) \D^\beta \wt{\u}_0 \bigr\|_{L_2(\R^d)}
\leqslant {\mathcal C}_{25} \sum_{l=1}^p \eps^l \| \wt{\u}_0 \|_{H^{p+l}(\R^d)},
\end{aligned}
\end{equation}
where ${\mathcal C}_{25} : = c_{7}(d,p) \|g\|_{L_\infty} \alpha_1 C_\Lambda$.
Combining \eqref{1.12}, \eqref{3.15}, and \eqref{4.60}--\eqref{4.62}, we arrive at
\begin{equation}
\label{4.63}
\| g^\eps b(\D) {\mathbf v}_\eps - \wt{g}^\eps S_\eps b(\D) \wt{\u}_0 \|_{L_2(\O)}
\leqslant {\mathcal C}_{26} c(\varphi) \left(  \eps |\zeta|^{-1/2+1/2p} + \eps^p \right) \| \FF \|_{L_2(\O)},
\end{equation}
where ${\mathcal C}_{26} : = \| g\|_{L_\infty} r_1 \alpha_1^{1/2} C^{(p+1)} + p \,{\mathcal C}_{25} \max\{ C^{(p+1)},\dots, C^{(2p)}\}$.

As a result, relations \eqref{4.59} and \eqref{4.63} imply the required inequality \eqref{3.11}
with the constant ${\mathcal C}_{3} : = {\mathcal C}_{24}+{\mathcal C}_{26}$. This completes the proof of Theorem~\ref{th3.2}. $\square$

\section{Removal of the smoothing operator.\\ Special cases}

\subsection{Removal of the smoothing operator }

\begin{condition}\label{cond_Lambda}
Suppose that the $\Gamma$-periodic solution $\Lambda$ of problem \eqref{1.10}
is bounded and is a multiplier from $H^p(\R^d;\C^m)$ to $H^p(\R^d;\C^n)${\rm:}
$$
\Lambda \in L_\infty(\R^d) \cap M(H^p(\R^d;\C^m)\to H^p(\R^d;\C^n)).
$$
\end{condition}

Since the matrix-valued function $\Lambda$ is periodic, Condition \ref{cond_Lambda} is equivalent to the relation
$$
\Lambda \in L_\infty(\Omega) \cap M(H^p(\Omega;\C^m)\to H^p(\Omega;\C^n)).
$$
The norm of the operator $[\Lambda]$ of multiplication by the matrix-valued function $\Lambda(\x)$ is denoted by
\begin{equation}
\label{MLambda}
M_\Lambda := \|[\Lambda]\|_{H^p(\R^d) \to H^p(\R^d)}.
\end{equation}

Description of the spaces of multipliers in the Sobolev classes can be found in the book~\cite{MaSh}.
Some sufficient conditions ensuring that Condition~\ref{cond_Lambda} is satisfied are known
(see~\cite[Proposition~7.10]{KuSu}).

\begin{proposition}\label{prop_MLambda}
Suppose that at least one of the following two assumptions is satisfied{\rm :}

$1^\circ$. $2p>d${\rm ;}

$2^\circ$. $g^0=\underline{g}$, i.~e., representations \eqref{1.19} are valid.

\noindent Then Condition {\rm \ref{cond_Lambda}} holds, and $\|\Lambda\|_{L_\infty}$ and the multiplier norm  \eqref{MLambda}
are controlled in terms of $m$, $n$, $d$, $p$, $\alpha_0$, $\alpha_1$, $\|g\|_{L_\infty}$, $\|g^{-1}\|_{L_\infty}$,
and the parameters of the lattice $\Gamma$.
\end{proposition}

Under Condition \ref{cond_Lambda}, instead of the corrector  \eqref{3.5}, it is possible to use the
standard corrector
\begin{equation}
\label{6.1}
K^0_N(\zeta;\eps) := [\Lambda^\eps] b(\D) (A_N^0 -\zeta I)^{-1},
\end{equation}
which in this case is a continuous mapping of $L_2(\O;\C^n)$ into $H^p(\O;\C^n)$.
Correspondingly, instead of the function~\eqref{3.8}, one can use the following approximation of $\u_\eps$:
\begin{equation}
\label{6.2}
{\mathbf v}^0_\eps := (A_N^0 - \zeta I)^{-1} \FF + \eps^p K^0_N(\zeta;\eps) \FF = \u_0 + \eps^p \Lambda^\eps b(\D) \u_0.
\end{equation}

\begin{theorem}
\label{no_smooth_O}
Suppose that the assumptions of Theorem {\rm \ref{th3.1}} and Condition {\rm \ref{cond_Lambda}} are satisfied.
Let $K^0_N(\zeta;\eps)$~be the operator \eqref{6.1}, and let  ${\mathbf v}_\eps^0$~be given by~\eqref{6.2}.
Then for $0< \eps \leqslant \eps_1$ we have
\begin{equation}
\label{6.3}
\| \u_\eps - {\mathbf v}_\eps^0 \|_{H^p(\O)} \leqslant
 \wt{\mathcal C}_2 \left( c(\varphi) \eps^{1/2} |\zeta|^{-1/2+1/4p} + c(\varphi)^2 \eps |\zeta|^{-1/2+1/2p}
+ c(\varphi) \eps^{p} \right) \|\FF\|_{L_2(\O)},
\end{equation}
or, in operator terms,
\begin{equation}
\label{6.4}
\begin{aligned}
&\|( A_{N,\eps} - \zeta I)^{-1} - ( A_{N}^0 - \zeta I)^{-1} - \eps^p K_N^0(\zeta;\eps) \|_{L_2(\O) \to H^p(\O)}
\\
&\quad\leqslant  \wt{\mathcal C}_2
\left( c(\varphi) \eps^{1/2} |\zeta|^{-1/2+1/4p} + c(\varphi)^2 \eps |\zeta|^{-1/2+1/2p} + c(\varphi) \eps^{p} \right).
\end{aligned}
\end{equation}
The flux $\p_\eps = g^\eps b(\D) \u_\eps$ satisfies
\begin{equation}
\label{6.5}
\| \p_\eps - \wt{g}^\eps b(\D) \u_0 \|_{L_2(\O)} \leqslant
  \wt{\mathcal C}_3
\left( c(\varphi) \eps^{1/2} |\zeta|^{-1/2+1/4p} + c(\varphi)^2 \eps |\zeta|^{-1/2+1/2p}
+ c(\varphi) \eps^{p} \right) \|\FF\|_{L_2(\O)}
\end{equation}
for  $0< \eps \leqslant \eps_1$.
The constants $\wt{\mathcal C}_2$ and $\wt{\mathcal C}_3$ depend only on $m$, $d$, $p$,
$\alpha_0$, $\alpha_1$, $\|g\|_{L_\infty}$, $\|g^{-1}\|_{L_\infty}$, $k_1$, $k_2$, the parameters of the lattice $\Gamma$, the domain $\O$,
and also on $\| \Lambda \|_{L_\infty}$ and $M_\Lambda$.
\end{theorem}

Note that
\begin{equation}
\label{6.8}
\| {\mathbf v} \|^2_{H^p(\R^d)} \leqslant \wh{\mathfrak c}_p \left( \| {\mathbf v} \|^2_{L_2(\R^d)} + \| \D^p {\mathbf v} \|^2_{L_2(\R^d)} \right),
\quad {\mathbf v} \in H^p(\R^d;\C^n),
\end{equation}
where $\wh{\mathfrak c}_p$ depends only on $d$ and $p$.

In order to prove Theorem \ref{no_smooth_O}, we need the following lemma proved in \cite[Lemma~7.2]{Su2017}.

\begin{lemma}\label{lem6.2}\mbox{}
$1^\circ$.
Suppose that $\Lambda$ is a multiplier from $H^p(\R^d;\C^m)$ to $H^p(\R^d;\C^n)$, and~$M_\Lambda$~is the norm of this
muliplier. Then for any $\u{\in} H^p(\R^d;\C^m)$ and $\eps >0$ we have
\begin{equation*}
\eps^{2p} \int_{\R^d} |\D^p (\Lambda^\eps(\x) \u(\x) )|^2 \,d\x
\leqslant \wh{\mathfrak c}_p M_\Lambda^2 \int_{\R^d} \left( |\u(\x)|^2 + \eps^{2p} |\D^p \u(\x)|^2 \right) \, d\x.
\end{equation*}

$2^\circ$.
Suppose that Condition {\rm \ref{cond_Lambda}} is satisfied.
Then the matrix-valued function \eqref{1.12} is a multiplier from $H^p(\R^d;\C^m)$ to $L_2(\R^d;\C^m)$,
and the norm of this multiplier is estimated by a constant $M_{\wt{g}}$ depending only on $d$, $p$, $\|g\|_{L_\infty}$,
$\alpha_1$, $\|\Lambda\|_{L_\infty}$, and $M_\Lambda$. Moreover, for any $\u \in H^p(\R^d;\C^m)$ and $\eps >0$ we have
\begin{equation}
\label{6.7}
 \int_{\R^d} | \wt{g}^\eps(\x) \u(\x) |^2 \,d\x
\leqslant \wh{\mathfrak c}_p M_{\wt{g}}^2 \int_{\R^d} \left( |\u(\x)|^2 + \eps^{2p} |\D^p \u(\x)|^2 \right) \, d\x.
\end{equation}
\end{lemma}

\begin{proof}[Proof of Theorem \emph{\ref{no_smooth_O}}]
Let ${\mathbf v}_\eps$ and ${\mathbf v}_\eps^0$ be defined by \eqref{3.8} and \eqref{6.2}, respectively. We estimate their difference
in the $H^p(\O;\C^n)$-norm. By \eqref{6.8},
\begin{equation}
\label{6.9}
\begin{aligned}
&\| {\mathbf v}_\eps - {\mathbf v}_\eps^0 \|_{H^p(\O)}^2 \leqslant
\eps^{2p} \| \Lambda^\eps (I-S_\eps) b(\D) \wt{\u}_0 \|^2_{H^p(\R^d)}
\\
&\quad\leqslant \wh{\mathfrak c}_p \eps^{2p}
\left( \| \Lambda^\eps (I-S_\eps) b(\D) \wt{\u}_0 \|^2_{L_2(\R^d)}
+ \| \D^p(\Lambda^\eps (I-S_\eps) b(\D) \wt{\u}_0) \|^2_{L_2(\R^d)} \right).
\end{aligned}
\end{equation}
Combining Condition \ref{cond_Lambda}, inequality  $\|S_\eps\|_{L_2 \to L_2} \leqslant 1$, and \eqref{1.4}, we have
\begin{equation}
\label{6.10}
\| \Lambda^\eps (I-S_\eps) b(\D) \wt{\u}_0 \|_{L_2(\R^d)}
\leqslant 2 \|\Lambda\|_{L_\infty} \alpha_1^{1/2} \| \wt{\u}_0 \|_{H^p(\R^d)}.
\end{equation}
Next, Lemma \ref{lem6.2} implies that
\begin{equation}
\label{6.11}
\begin{aligned}
&\eps^{2p} \| \D^p(\Lambda^\eps (I-S_\eps) b(\D) \wt{\u}_0) \|^2_{L_2(\R^d)}
\\
&\quad\leqslant
\wh{\mathfrak c}_p M_\Lambda^2
\bigl( \|(I-S_\eps) b(\D) \wt{\u}_0 \|^2_{L_2(\R^d)} +
\eps^{2p} \|(I-S_\eps) \D^p b(\D) \wt{\u}_0 \|^2_{L_2(\R^d)} \bigr).
\end{aligned}
\end{equation}
Applying Proposition \ref{prop1.4} and \eqref{1.4}, we obtain
\begin{equation}
\label{6.12}
\|(I-S_\eps) b(\D) \wt{\u}_0 \|_{L_2(\R^d)} \leqslant
\eps r_1 \alpha_1^{1/2} \| \wt{\u}_0 \|_{H^{p+1}(\R^d)}.
\end{equation}
Next, by inequality $\|S_\eps\|_{L_2 \to L_2} \leqslant 1$ and \eqref{1.4},
\begin{equation}
\label{6.13}
\|(I-S_\eps) \D^p b(\D) \wt{\u}_0 \|_{L_2(\R^d)} \leqslant 2 \alpha_1^{1/2} \| \wt{\u}_0 \|_{H^{2p}(\R^d)}.
\end{equation}

Combining  \eqref{3.14}, \eqref{3.15}, and \eqref{6.9}--\eqref{6.13}, and taking into account that $|\zeta|\ge 1$, we arrive at
\begin{equation}
\label{6.14}
\| {\mathbf v}_\eps - {\mathbf v}_\eps^0 \|_{H^p(\O)} \leqslant
{\mathcal C}_{27} c(\varphi) \bigl( \eps |\zeta|^{-1/2+1/2p} + \eps^p \bigr)\| \FF \|_{L_2(\O)},
\end{equation}
where
${\mathcal C}_{27} := \alpha_1^{1/2} \max \{ \wh{\mathfrak c}_p M_\Lambda r_1 C^{(p+1)},
2  \wh{\mathfrak c}_p^{1/2} \| \Lambda \|_{L_\infty} C^{(p)}+ 2 \wh{\mathfrak c}_p M_\Lambda C^{(2p)} \}$.

Now, relations \eqref{3.9} and \eqref{6.14} imply the required estimate  \eqref{6.3} with the constant
$\wt{\mathcal C}_2 := {\mathcal C}_2 + {\mathcal C}_{27}$.

We proceed to the proof of inequality  \eqref{6.5}. By \eqref{6.7},
\begin{equation}
\label{6.14a}
\begin{aligned}
&\| \wt{g}^\eps b(\D) \u_0 - \wt{g}^\eps S_\eps b(\D) \wt{\u}_0 \|^2_{L_2(\O)}
\leqslant
\| \wt{g}^\eps (I-S_\eps) b(\D) \wt{\u}_0 \|^2_{L_2(\R^d)}
\\
&\quad\leqslant \wh{\mathfrak c}_p M_{\wt{g}}^2 \bigl( \|(I-S_\eps) b(\D) \wt{\u}_0 \|^2_{L_2(\R^d)} +
\eps^{2p} \|(I-S_\eps) \D^p b(\D) \wt{\u}_0 \|^2_{L_2(\R^d)} \bigr).
\end{aligned}
\end{equation}
Together with \eqref{3.14}, \eqref{3.15}, \eqref{6.12}, and \eqref{6.13}, this yields
\begin{equation}
\label{6.15}
\| \wt{g}^\eps b(\D) \u_0 - \wt{g}^\eps S_\eps b(\D) \wt{\u}_0 \|_{L_2(\O)}
\leqslant
{\mathcal C}_{28} c(\varphi) \bigl( \eps |\zeta|^{-1/2+1/2p} + \eps^p \bigr)\| \FF \|_{L_2(\O)},
\end{equation}
where ${\mathcal C}_{28} := \wh{\mathfrak c}_p^{1/2} M_{\wt{g}} \alpha_1^{1/2} \max\{ r_1 C^{(p+1)}, 2 C^{(2p)}\}$.

Finally, relations \eqref{3.11} and \eqref{6.15} give the desired inequality \eqref{6.5} with the constant
$\wt{\mathcal C}_3 :={\mathcal C}_3 +  {\mathcal C}_{28}$.
\end{proof}

Comparing Theorem \ref{no_smooth_O} and Proposition \ref{prop_MLambda}, we arrive at the following statement.

\begin{corollary}
\label{cor6.2}
Suppose that  $2p >d$.
Then for $0< \eps \leqslant \eps_1$ estimates \eqref{6.3}--\eqref{6.5} hold, and the constants
$\wt{\mathcal C}_2$ and $\wt{\mathcal C}_3$ depend only on $m$, $n$, $d$, $p$,
$\alpha_0$, $\alpha_1$, $\|g\|_{L_\infty}$, $\|g^{-1}\|_{L_\infty}$, $k_1$, $k_2$, the parameters of the lattice $\Gamma$, and the domain $\O$.
\end{corollary}

\begin{remark}\label{rem7.2a}\mbox{}

{\rm 1)}
For fixed $\zeta \in \C \setminus \R_+$, $|\zeta| \geqslant 1$, estimates of Theorem~{\rm \ref{no_smooth_O}}
are of order $O(\eps^{1/2})$. The error becomes smaller, as $|\zeta|$ grows.

{\rm 2)}
Estimates of Theorem {\rm \ref{no_smooth_O}} are uniform with respect to $\varphi$ in any domain of the form
$\{ \zeta = |\zeta| e^{i\varphi}: |\zeta| \geqslant 1, \varphi_0 \leqslant \varphi \leqslant 2\pi - \varphi_0\}$
with arbitrarily small $\varphi_0 >0$.

{\rm 3)}
The assumptions of Corollary {\rm \ref{cor6.2}} are satisfied in the following cases that are interesting for applications{\rm:}
~if  $p=2$ and $d=2$ or $d=3$.

{\rm 4)}~If $m=n$, then $g^0 = \underline{g}$.
For instance, this assumption is satisfied for the operator $A_\eps = \Delta g^\eps(\x)\Delta$ in $L_2(\R^d)$
in arbitrary dimension. In this case Proposition \emph{\ref{prop6.8}} {\rm (}see below{\rm )} applies.
\end{remark}

\subsection{Special cases}

Let $g^0 = \overline{g}$, i.~e., relations \eqref{1.18} are satisfied.
Then the $\Gamma$-periodic solution $\Lambda(\mathbf{x})$ of problem \eqref{1.10} is equal to zero.
By \eqref{3.5}--\eqref{3.7}, we have \hbox{$K_N(\zeta;\varepsilon)=0$} and ${\mathbf v}_\varepsilon = {\mathbf u}_0$.
Applying Theorem~\ref{th3.2}, we arrive at the following statement.

\begin{proposition}
\label{prop6.7}
Suppose that the assumptions of Theorem~\emph{\ref{th3.1}} are satisfied. Let $g^0 = \overline{g}$, i.~e., relations \eqref{1.18} hold.
Then for $\zeta \in \C \setminus \R_+$, $|\zeta|\ge 1$, and $0< \eps \le \eps_1$ we have
\begin{equation*}
\begin{split}
\|\u_\eps - \u_0 \|_{H^p(\O)} \le {\mathcal C}_2 \left( c(\varphi) \eps^{1/2} |\zeta|^{-1/2+1/4p}
+ c(\varphi)^2 \eps |\zeta|^{-1/2+1/2p}  + c(\varphi) \eps^p \right) \|\FF\|_{L_2(\O)}.
\end{split}
\end{equation*}
\end{proposition}

Now, we consider the case where $g^0 = \underline{g}$, i.~e., representations \eqref{1.19} are satisfied.
By Remark \ref{rem1.4},  we have $\widetilde{g}(\mathbf{x}) = g^0 = \underline{g}$.
In this case, the results can be improved.

\begin{proposition}
\label{prop6.8}
Suppose that the assumptions of Theorem~\emph{\ref{th3.1}} are satisfied. Let $g^0 = \underline{g}$, i.~e., representations \eqref{1.19} hold.
Let ${\mathbf v}_\eps^0$ be given by \eqref{6.2}. Let $\p_\eps = g^\eps b(\D) \u_\eps$.
Then for $\zeta \in \C \setminus \R_+$, $|\zeta|\ge 1$, and $0< \eps \le \eps_1$ we have
\begin{align}
\label{6.18}
&\|\u_\eps - {\mathbf v}_\eps^0 \|_{H^p(\O)} \le \widehat{\mathcal C}_2 \left( c(\varphi)^2 \eps |\zeta|^{-1/2+1/2p}
+ c(\varphi) \eps^p \right) \|\FF\|_{L_2(\O)},
\\
\label{6.19}
&\|\p_\eps - g^0 b(\D)\u_0 \|_{L_2(\O)} \le \widehat{\mathcal C}_3 \left( c(\varphi)^2 \eps |\zeta|^{-1/2+1/2p}
+ c(\varphi) \eps^p \right) \|\FF\|_{L_2(\O)}.
\end{align}
The constants $\widehat{\mathcal C}_2$ and $\widehat{\mathcal C}_3$ depend only on
$m$, $n$, $d$, $p$, $\alpha_0$, $\alpha_1$, $\|g\|_{L_\infty}$, $\|g^{-1}\|_{L_\infty}$, $k_1$, $k_2$, the parameters of the lattice $\Gamma$,
and the domain $\O$.
\end{proposition}

\begin{proof}
First, assume that ${\rm{Re}}\,\zeta \le 0$ and $|\zeta| \ge 1$. If $g^0 = \underline{g}$, then $\wt{g}=g^0$
and ${\mathcal J}_\eps^{(2)}[\eeta]=0$ (see \eqref{4.7}).
Hence, according to \eqref{4.5} and \eqref{4.8}, the functional \eqref{4.1} satisfies
\begin{equation}
\label{6.20}
| {\mathcal J}_\eps[\eeta] | \le {\mathcal C}_{10} \eps |\zeta|^{-1/2 + 1/2p} \| \FF \|_{L_2(\O)} \| \D^p \eeta\|_{L_2(\O)},
\quad \eeta \in H^p(\O;\C^n).
\end{equation}
Substituting $\eeta = \w_\eps$ in \eqref{4.3} and using \eqref{6.20} (similarly to \eqref{4.25}--\eqref{4.28a}), we obtain
\begin{equation}
\label{6.21}
\| \w_\eps \|_{H^p(\O)} \le {\mathcal C}_{15}' \eps |\zeta|^{-1/2 + 1/2p} \| \FF \|_{L_2(\O)}, \quad 0 < \eps \le \eps_1,
\end{equation}
where ${\mathcal C}_{15}' : = {\mathcal C}_{10}(k_1 \|g^{-1}\|_{L_\infty} + 2 k_2)$.
From \eqref{3.39} and \eqref{6.21} it follows that
\begin{equation}
\label{6.22}
\| \u_\eps - {\mathbf v}_\eps \|_{H^p(\O)} \le ({\mathcal C}_{5} + {\mathcal C}_{15}') \bigl( \eps |\zeta|^{-1/2 + 1/2p} + \eps^p \bigr) \| \FF \|_{L_2(\O)}, \quad
{\rm{Re}}\,\zeta \le 0,\ |\zeta| \ge 1,\ 0 < \eps \le \eps_1.
\end{equation}

Now, let $\zeta \in \C \setminus \R_+$, $|\zeta| \ge 1$, and ${\rm{Re}}\,\zeta > 0$. Let
$\widehat{\zeta} = - {\rm{Re}}\,\zeta + i {\rm{Im}}\,\zeta$. Estimate \eqref{6.22} at the point $\widehat{\zeta}$ means that
\begin{equation}
\label{6.23}
\| (A_{N,\eps} - \widehat{\zeta} I)^{-1} - (A_{N}^0 - \widehat{\zeta} I)^{-1} - \eps^p K_N(\widehat{\zeta};\eps)  \|_{L_2(\O) \to H^p(\O)}
\le  ({\mathcal C}_{5} + {\mathcal C}_{15}') \bigl( \eps |\zeta|^{-1/2 + 1/2p} + \eps^p \bigr)
\end{equation}
for $0 < \eps \le \eps_1$. Combining identity \eqref{4.56} and estimates \eqref{22.9}, \eqref{4.52} (at the point $\widehat{\zeta}$),
\eqref{4.55}, and \eqref{6.23}, we obtain
\begin{equation}
\label{6.24}
\| (A_{N,\eps} - {\zeta} I)^{-1} - (A_{N}^0 - {\zeta} I)^{-1} - \eps^p K_N( {\zeta};\eps)  \|_{L_2(\O) \to H^p(\O)}
\le  \widehat{\mathcal C}_{2}' \bigl( c(\varphi)^2 \eps |\zeta|^{-1/2 + 1/2p} + c(\varphi) \eps^p \bigr)
\end{equation}
for  $0 < \eps \le \eps_1$, where $\widehat{\mathcal C}_{2}' : = 2 ({\mathcal C}_{5} + {\mathcal C}_{15}') + 4 {\mathcal C}_{0} {\mathcal C}_{1}'$.

By Proposition~\ref{prop_MLambda}($2^\circ$), Condition~\ref{cond_Lambda} is satisfied; moreover,
$\|\Lambda\|_{L_\infty}$ and $M_\Lambda$ are controlled in terms of $m$, $n$, $d$, $p$,
$\alpha_0$, $\alpha_1$, $\|g\|_{L_\infty}$, $\|g^{-1}\|_{L_\infty}$, and the parameters of the lattice $\Gamma$.
Then estimate \eqref{6.14} holds. Relations \eqref{6.14} and \eqref{6.24} imply the required inequality \eqref{6.18} with the constant
$\widehat{\mathcal C}_{2} := \widehat{\mathcal C}_{2}' + {\mathcal C}_{27}$.

Estimate \eqref{6.19} is deduced from \eqref{4.63}, \eqref{6.15}, \eqref{6.24}, and the identity $\widetilde{g}(\x) = g^0$.
Herewith,
$\widehat{\mathcal C}_{3} := c_6(d,p) \|g\|_{L_\infty} \alpha_1^{1/2} \widehat{\mathcal C}'_{2} + {\mathcal C}_{26}+ {\mathcal C}_{28}$.
\end{proof}

\section{Approximation of the resolvent $(A_{N,\eps} - \zeta I)^{-1}$ for $|\zeta| \le 1$}

\subsection{The case $|\zeta|\le 1$}
In this section, we extend the results of Theorems \ref{th3.1} and  \ref{th3.2}
to the set $\zeta \in \C \setminus \R_+$, $|\zeta|\le 1$, using suitable identities for the resolvents.

\begin{lemma}
Let $\zeta \in \C \setminus \R_+$, $|\zeta|\le 1$. Then for $\eps>0$ we have
\begin{align}
\label{7.1}
\| (A_{N,\eps} - \zeta I)^{-1} \|_{L_2(\O) \to H^p(\O)} \le {\mathcal C}_0 c(\varphi) |\zeta|^{-1},
\\
\label{7.2a}
\| (A_{N}^0 - \zeta I)^{-1} \|_{L_2(\O) \to H^{2p}(\O)} \le 2 \wh{c} c(\varphi) |\zeta|^{-1}.
\end{align}
\end{lemma}

\begin{proof}
Recall that inequality \eqref{2.10a} holds for any $\zeta \in \C \setminus \R_+$.
If $|\zeta| \le 1$, it implies \eqref{7.1}.

Next, by \eqref{22.11},
\begin{equation}
\label{7.3a}
\|( A_{N}^0 -\zeta I)^{-1} \|_{L_2(\O) \to H^{2p}(\O)}\le
\wh{c}  \|( A_{N}^0 +I) ( A_{N}^0 -\zeta I)^{-1} \|_{L_2(\O) \to L_2(\O)}
\le \wh{c} \sup_{x\ge 0} \frac{x+1}{|x-\zeta|}.
\end{equation}
A calculation shows that
\begin{equation}
\label{7.11}
\sup_{x\ge 0} \frac{x+1}{|x-\zeta|} \le 2 c(\varphi)  |\zeta|^{-1}, \quad \zeta \in \C \setminus \R_+,\ |\zeta|\le 1.
\end{equation}
Relations \eqref{7.3a} and \eqref{7.11} imply estimate \eqref{7.2a}.
\end{proof}

\begin{theorem}\label{th7.2}
Suppose that $\O\subset \R^d$ is a bounded domain of class $C^{2p}$.
Let $\zeta \in \C \setminus \R_+$ and $|\zeta|\le 1$.
Let $\u_\eps = (A_{N,\eps} - \zeta I)^{-1}\FF$ and $\u_0 = (A_{N}^0 - \zeta I)^{-1}\FF$, where $\FF \in L_2(\O;\C^n)$.
Let $K_N(\zeta;\eps)$~be the operator \eqref{3.5}, and let  ${\mathbf v}_\eps$~be given by \eqref{3.6}, \eqref{3.7}.
Suppose that the number $\eps_1$ is subject to Condition~\emph{\ref{cond_eps}}.
Then for $0< \eps \leqslant \eps_1$ we have
\begin{align}
\nonumber
&\| \u_\eps - \u_0 \|_{L_2(\O)} \leqslant
 {\mathfrak C}_1  c(\varphi)^2 \eps |\zeta|^{-2}  \|\FF\|_{L_2(\O)},
\\
\label{7.5}
&\| \u_\eps - {\mathbf v}_\eps \|_{H^p(\O)} \leqslant
 {\mathfrak C}_2 \left( c(\varphi) \eps^{1/2} |\zeta|^{-1} + c(\varphi)^2 \eps |\zeta|^{-2} \right) \|\FF\|_{L_2(\O)},
\end{align}
or, in operator terms,
\begin{align}
\label{7.6}
&\|( A_{N,\eps} - \zeta I)^{-1} - ( A_{N}^0 - \zeta I)^{-1} \|_{L_2(\O) \to L_2(\O)}
\le {\mathfrak C}_1  c(\varphi)^2 \eps |\zeta|^{-2},
\\
\label{7.7}
&\|( A_{N,\eps} - \zeta I)^{-1} - ( A_{N}^0 - \zeta I)^{-1} - \eps^p K_N(\zeta;\eps) \|_{L_2(\O) \to H^p(\O)}
\leqslant
 {\mathfrak C}_2 \left( c(\varphi) \eps^{1/2} |\zeta|^{-1} + c(\varphi)^2 \eps |\zeta|^{-2} \right).
\end{align}
The flux  $\p_\eps = g^\eps b(\D) \u_\eps$  satisfies
\begin{equation}
\label{7.8}
\| \p_\eps - \wt{g}^\eps S_\eps b(\D)\wt{\u}_0 \|_{L_2(\O)} \leqslant
  {\mathfrak C}_3 \left( c(\varphi) \eps^{1/2} |\zeta|^{-1} + c(\varphi)^2 \eps |\zeta|^{-2} \right) \|\FF\|_{L_2(\O)}
\end{equation}
for $0< \eps \leqslant \eps_1$.
The constants ${\mathfrak C}_1$, ${\mathfrak C}_2$, and ${\mathfrak C}_3$ depend only on $m$, $d$, $p$,
$\alpha_0$, $\alpha_1$, $\|g\|_{L_\infty}$, $\|g^{-1}\|_{L_\infty}$, $k_1$, $k_2$, the parameters of the lattice $\Gamma$, and the domain $\O$.
\end{theorem}

\begin{proof}
We apply \eqref{4.52} at the point $\zeta=-1$:
\begin{equation}
\label{7.9}
\|( A_{N,\eps} + I)^{-1} - ( A_{N}^0 + I)^{-1} \|_{L_2(\O) \to L_2(\O)}
\le {\mathcal C}_1'  \eps, \quad 0< \eps \le \eps_1.
\end{equation}
Together with the analog of identity \eqref{4.53}, this yields
\begin{equation}
\label{7.10}
\|( A_{N,\eps} -\zeta I)^{-1} - ( A_{N}^0 -\zeta I)^{-1} \|_{L_2(\O) \to L_2(\O)}
\le {\mathcal C}_1'  \eps \sup_{x\ge 0} \frac{(x+1)^2}{|x-\zeta|^2}, \quad 0< \eps \le \eps_1.
\end{equation}
Relations \eqref{7.11} and \eqref{7.10} imply the required estimate \eqref{7.6} with the constant ${\mathfrak C}_1= 4 {\mathcal C}_1'$.

Now, we apply \eqref{4.30} at the point $\zeta=-1$:
\begin{equation}
\label{7.12}
\|( A_{N,\eps} + I)^{-1} - ( A_{N}^0 + I)^{-1} - \eps^p K_N(-1;\eps) \|_{L_2(\O) \to H^p(\O)}
\le 2 {\mathcal C}'_2  \eps^{1/2}
\end{equation}
for $0< \eps \le \eps_1$.
Similarly to \eqref{4.56}, we have
\begin{equation}
\label{7.13}
\begin{aligned}
&(A_{N,\eps} - \zeta I)^{-1} - (A_{N}^0 - \zeta I)^{-1} - \eps^p K_N(\zeta;\eps)
\\
&=
\left( (A_{N,\eps} + I)^{-1} - (A_{N}^0 + I)^{-1} - \eps^p K_N(-1;\eps)\right)
(A_{N}^0 +I) (A_{N}^0 - \zeta I)^{-1}
\\
&+  (\zeta +1) (A_{N,\eps} - \zeta I)^{-1} \left( (A_{N,\eps} + I)^{-1} - (A_{N}^0 + I)^{-1} \right) (A_{N}^0 +I) (A_{N}^0 - \zeta I)^{-1}.
\end{aligned}
\end{equation}
Together with \eqref{7.11}, this implies
\begin{equation}
\label{7.12a}
\begin{aligned}
&\| (A_{N,\eps} - \zeta I)^{-1} - (A_{N}^0 - \zeta I)^{-1} - \eps^p K_N(\zeta;\eps)\|_{L_2(\O) \to H^p(\O)}
\\
& \le 2 c(\varphi) |\zeta|^{-1}
\|  (A_{N,\eps} + I)^{-1} - (A_{N}^0 + I)^{-1} - \eps^p K_N(-1;\eps)\|_{L_2(\O) \to H^p(\O)}
\\
&+ 2 c(\varphi) |\zeta|^{-1} |\zeta +1|  \|(A_{N,\eps} - \zeta I)^{-1}\|_{L_2(\O) \to H^p(\O)}
  \|(A_{N,\eps} + I)^{-1} - (A_{N}^0 + I)^{-1}\|_{L_2(\O) \to L_2(\O)}.
 \end{aligned}
\end{equation}
Combining this with \eqref{7.1}, \eqref{7.9}, \eqref{7.12}, and taking into account that $|\zeta|\le 1$, we obtain the required estimate
\eqref{7.7} with the constant ${\mathfrak C}_2 := 4 \max \{ {\mathcal C}'_{2}, {\mathcal C}_{0} {\mathcal C}'_{1} \}$.

It remains to check \eqref{7.8}. Relations \eqref{1.3}, \eqref{1.5}, and \eqref{7.5} show that
\begin{equation}
\label{7.16}
\| \p_\eps - g^\eps b(\D) {\mathbf v}_\eps \|_{L_2(\O)} \le
c_6(d,p) \|g\|_{L_\infty} \alpha_1^{1/2} {\mathfrak C}_2 \left( c(\varphi) \eps^{1/2} |\zeta|^{-1} + c(\varphi)^2 \eps |\zeta|^{-2} \right) \|\FF\|_{L_2(\O)}
\end{equation}
for $0< \eps \le \eps_1$. Similarly to \eqref{4.60}--\eqref{4.62}, we have
\begin{equation}
\label{7.17}
\begin{split}
&\| g^\eps b(\D) {\mathbf v}_\eps - \wt{g}^\eps S_\eps b(\D) \wt{\u}_0 \|_{L_2(\O)}
\\
&\le
\eps \|g\|_{L_\infty} r_1 \alpha_1^{1/2} \| \wt{\u}_0 \|_{H^{p+1}(\R^d)}+
{\mathcal C}_{25} \sum_{l=1}^p \eps^l \| \wt{\u}_0 \|_{H^{p+l}(\R^d)}
\le {\mathfrak C}_4 \eps \| \wt{\u}_0 \|_{H^{2p}(\R^d)},
\end{split}
\end{equation}
where ${\mathfrak C}_4 := \|g\|_{L_\infty} r_1 \alpha_1^{1/2} + p \,{\mathcal C}_{25}$.
From \eqref{3.4} and \eqref{7.2a} it follows that
\begin{equation}
\label{7.18}
\| \wt{\u}_0 \|_{H^{2p}(\O)} \le 2 C_\O^{(2p)} \wh{c} c(\varphi) |\zeta|^{-1} \|\FF\|_{L_2(\O)}.
\end{equation}

As a result, relations \eqref{7.16}--\eqref{7.18} imply the required estimate \eqref{7.8} with the constant
${\mathfrak C}_3 := c_6(d,p) \|g\|_{L_\infty} \alpha_1^{1/2} {\mathfrak C}_2 + 2 {\mathfrak C}_4 C_\O^{(2p)} \wh{c}$.
\end{proof}

\subsection{Removal of the smoothing operator}

\begin{theorem}
Suppose that the assumptions of Theorem \emph{\ref{th7.2}} and Condition \emph{\ref{cond_Lambda}} are satisfied.
Let $K_N^0(\zeta;\eps)$~be the operator \eqref{6.1}, and let  ${\mathbf v}_\eps^0$ be given by~\eqref{6.2}.
Then for \hbox{$0< \eps \leqslant \eps_1$} we have
\begin{align}
\label{7.19}
\| \u_\eps - {\mathbf v}_\eps^0 \|_{H^p(\O)} \leqslant
 \wt{\mathfrak C}_2 \left( c(\varphi) \eps^{1/2} |\zeta|^{-1} + c(\varphi)^2 \eps |\zeta|^{-2} \right) \|\FF\|_{L_2(\O)}.
\end{align}
In operator terms,
\begin{equation*}
\begin{split}
&\|( A_{N,\eps} - \zeta I)^{-1} - ( A_{N}^0 - \zeta I)^{-1} - \eps^p K_N^0(\zeta;\eps) \|_{L_2(\O) \to H^p(\O)}
\\
&\leqslant
 \wt{\mathfrak C}_2 \left( c(\varphi) \eps^{1/2} |\zeta|^{-1} + c(\varphi)^2 \eps |\zeta|^{-2} \right).
\end{split}
\end{equation*}
The flux  $\p_\eps = g^\eps b(\D) \u_\eps$ satisfies
\begin{equation}
\label{7.21}
\| \p_\eps - \wt{g}^\eps b(\D) {\u}_0 \|_{L_2(\O)} \leqslant
  \wt{\mathfrak C}_3 \left( c(\varphi) \eps^{1/2} |\zeta|^{-1} + c(\varphi)^2 \eps |\zeta|^{-2} \right) \|\FF\|_{L_2(\O)}
\end{equation}
for $0< \eps \leqslant \eps_1$.
The constants $\wt{\mathfrak C}_2$ and $\wt{\mathfrak C}_3$ depend only on $m$, $d$, $p$,
$\alpha_0$, $\alpha_1$, $\|g\|_{L_\infty}$, $\|g^{-1}\|_{L_\infty}$, $k_1$, $k_2$, the parameters of the lattice $\Gamma$, the domain $\O$,
and also on $\|\Lambda\|_{L_\infty}$ and $M_\Lambda$.
\end{theorem}

\begin{proof}
Similarly to \eqref{6.9}--\eqref{6.13}, we have
\begin{equation}
\label{7.22}
\| {\mathbf v}_\eps - {\mathbf v}_\eps^0 \|_{H^p(\O)} \leqslant
 {\mathfrak C}_5 \eps \| \wt{\u}_0 \|_{H^{2p}(\R^d)},
\end{equation}
where ${\mathfrak C}_5 := 2 \wh{\mathfrak c}_p^{1/2} \| \Lambda \|_{L_\infty} \alpha_1^{1/2} +\wh{\mathfrak c}_p M_\Lambda \alpha_1^{1/2} (r_1 +2)$.
From \eqref{7.18} and \eqref{7.22} it follows that
\begin{equation}
\label{7.23}
\| {\mathbf v}_\eps - {\mathbf v}_\eps^0 \|_{H^p(\O)} \leqslant
2{\mathfrak C}_5 C_\O^{(2p)} \wh{c} \, c(\varphi) \eps |\zeta|^{-1} \| \FF \|_{L_2(\O)}.
\end{equation}
Inequalities \eqref{7.5} and \eqref{7.23} imply the required estimate \eqref{7.19} with the constant
$\wt{\mathfrak C}_2 : = {\mathfrak C}_2 + 2 {\mathfrak C}_5 C_\O^{(2p)} \wh{c}$.

It remains to check \eqref{7.21}. Similarly to \eqref{6.12}, \eqref{6.13}, and \eqref{6.14a}, we have
\begin{equation}
\label{7.24}
\| \wt{g}^\eps b(\D)\u_0 - \wt{g}^\eps S_\eps b(\D)\wt{\u}_0 \|_{L_2(\O)} \leqslant
{\mathfrak C}_6  \eps \| \wt{\u}_0 \|_{H^{2p}(\R^d)},
\end{equation}
where ${\mathfrak C}_6 := \wh{\mathfrak c}_p^{1/2} M_{\wt{g}} \alpha_1^{1/2} (r_1 +2)$.
Together with \eqref{7.18}, this yields
\begin{equation}
\label{7.25}
\| \wt{g}^\eps b(\D)\u_0 - \wt{g}^\eps S_\eps b(\D)\wt{\u}_0 \|_{L_2(\O)} \leqslant
2 {\mathfrak C}_6 C_\O^{(2p)} \wh{c} \, c(\varphi) \eps |\zeta|^{-1} \| \FF \|_{L_2(\O)}.
\end{equation}
Combining \eqref{7.8} and \eqref{7.25}, we arrive at the desired estimate \eqref{7.21} with the constant
$\wt{\mathfrak C}_3 := {\mathfrak C}_3 + 2 {\mathfrak C}_6 C_\O^{(2p)} \wh{c}$.
\end{proof}

\subsection{Special cases}
The following statement can be checked with the help of Theorem~\ref{th7.2} (similarly to Proposition~\ref{prop6.7}).

\begin{proposition}
\label{prop7.4}
Suppose that the assumptions of Theorem~\emph{\ref{th7.2}} are satisfied. Let $g^0 = \overline{g}$, i.~e., relations \eqref{1.18} hold.
Then for $\zeta \in \C \setminus \R_+$, $|\zeta|\le 1$, and $0< \eps \le \eps_1$ we have
\begin{equation*}
\begin{split}
\|\u_\eps - \u_0 \|_{H^p(\O)} \le {\mathfrak C}_2 \left( c(\varphi) \eps^{1/2} |\zeta|^{-1} + c(\varphi)^2 \eps |\zeta|^{-2} \right) \|\FF\|_{L_2(\O)}.
\end{split}
\end{equation*}
\end{proposition}

Now, we consider the case where $g^0 = \underline{g}$.

\begin{proposition}
\label{prop7.5}
Suppose that the assumptions of Theorem~\emph{\ref{th7.2}} are satisfied. Suppose that $g^0 = \underline{g}$, i.~e., representations \eqref{1.19} hold.
Let ${\mathbf v}_\eps^0$ be given by \eqref{6.2}. Let $\p_\eps = g^\eps b(\D) \u_\eps$.
Then for $\zeta \in \C \setminus \R_+$, $|\zeta|\le 1$, and $0< \eps \le \eps_1$ we have
\begin{align}
\label{7.26}
&\|\u_\eps - {\mathbf v}_\eps^0 \|_{H^p(\O)} \le \widehat{\mathfrak C}_2 c(\varphi)^2 \eps |\zeta|^{-2} \|\FF\|_{L_2(\O)},
\\
\label{7.27}
&\|\p_\eps - g^0 b(\D)\u_0 \|_{L_2(\O)} \le \widehat{\mathfrak C}_3 c(\varphi)^2 \eps |\zeta|^{-2} \|\FF\|_{L_2(\O)}.
\end{align}
The constants $\widehat{\mathfrak C}_2$ and $\widehat{\mathfrak C}_3$ depend only on
$m$, $n$, $d$, $p$, $\alpha_0$, $\alpha_1$, $\|g\|_{L_\infty}$, $\|g^{-1}\|_{L_\infty}$, $k_1$, $k_2$, the parameters of the lattice $\Gamma$,
and the domain $\O$.
\end{proposition}

\begin{proof}
Under condition $g^0 = \underline{g}$, inequality \eqref{6.24}
at the point $\zeta = -1$ hods. Combining it with identity \eqref{7.13}, similarly to \eqref{7.12a}, we obtain
\begin{equation}
\label{77.24}
\begin{split}
\| (A_{N,\eps} -\zeta  I)^{-1} - ( A_{N}^0 - \zeta I)^{-1} - \eps^p K_N(\zeta; \eps) \|_{L_2(\O) \to H^p(\O)}
\le \widehat{\mathfrak C}_2' c(\varphi)^2 \eps |\zeta|^{-2},
\quad 0< \eps \le \eps_1,
\end{split}
\end{equation}
where $\widehat{\mathfrak C}'_2 := 4 \widehat{\mathcal C}'_2 + 4 {\mathcal C}_0 {\mathcal C}_1'$.

By Proposition~\ref{prop_MLambda}($2^\circ$), Condition \ref{cond_Lambda} is now satisfied; moreover,
$\|\Lambda\|_{L_\infty}$ and $M_\Lambda$ are controlled in terms of $m$, $n$, $d$, $p$,
$\alpha_0$, $\alpha_1$, $\|g\|_{L_\infty}$, $\|g^{-1}\|_{L_\infty}$, and the parameters of the lattice $\Gamma$.
Then inequality \eqref{7.23} holds. Together with \eqref{77.24}, it implies the
required estimate \eqref{7.26} with the constant
$\widehat{\mathfrak C}_2 := \widehat{\mathfrak C}'_2 + 2{\mathfrak C}_5 C_\O^{(2p)} \wh{c}$.

Estimate \eqref{7.27} is deduced from \eqref{7.17}, \eqref{7.18}, \eqref{7.25},  \eqref{77.24}, and the identity
 $\widetilde{g}(\x)=g^0$.
\end{proof}

\section{Approximation of the resolvent $( B_{N,\eps} - \zeta I)^{-1}$}

\subsection{The operator $B_{N,\eps}$}

We denote
$$
Z := \operatorname{Ker} b(\D) = \{ \z \in H^p(\O;\C^n):\ b(\D)\z =0 \}.
$$
From \eqref{22.2} it follows that
\begin{equation*}
\| \z \|^2_{H^p(\O)} \le k_2 \|\z\|^2_{L_2(\O)},\quad \z \in Z.
\end{equation*}
Together with the compactness of the embedding of $H^p(\O;\C^n)$ in $L_2(\O;\C^n)$, this shows that $Z$ is finite-dimensional.
We denote $\operatorname{dim} Z =:q$.  Obviously, the space $Z$ contains a subspace of $\C^n$-valued polynomials
of degree not exceeding $p-1$.
We put ${\mathcal H}(\O):= L_2(\O;\C^n) \ominus Z$ and $H^p_\perp (\O;\C^n):= H^p(\O;\C^n) \cap {\mathcal H}(\O)$.
It is easily seen (cf. \cite[Proposition~9.1]{Su03}) that the form $\|b(\D) \u\|_{L_2(\O)}$
defines a norm in $H^p_\perp(\O;\C^n)$ equivalent to the standard norm: there exists a constant $\wt{k}_1>0$ such that
\begin{equation}
\label{8.1}
\| \u \|_{H^p(\O)}^2 \leqslant \wt{k}_1  \| b(\D) \u \|^2_{L_2(\O)}, \quad \u \in H^p_\perp(\O;\C^n).
\end{equation}

Recall that the operator $A_{N,\eps}$ is generated in $L_2(\O;\C^n)$ by the form \eqref{22.3}, and $A_N^0$ is generated by the form \eqref{22.10}.
Obviously, $\operatorname{Ker} A_{N,\eps} = \operatorname{Ker} A_{N}^0 = Z$. The orthogonal decomposition
$L_2(\O;\C^n) = {\mathcal H}(\O) \oplus Z$ reduces both operators $A_{N,\eps}$ and $A_N^0$.
Denote by $B_{N,\eps}$ (respectively, $B_N^0$)
the part of the operator $A_{N,\eps}$ (respectively, $A_N^0$) in the subspace ${\mathcal H}(\O)$.
In other words, $B_{N,\eps}$ is the selfadjoint operator in ${\mathcal H}(\O)$ generated by the quadratic form
$$
b_{N,\eps}[\u,\u] = (g^\eps b(\D) \u, b(\D) \u)_{L_2(\O)},\quad \u \in H^p_\perp (\O;\C^n).
$$
Similarly, $B_{N}^0$ is the operator in ${\mathcal H}(\O)$ corresponding to the form
$$
b_{N}^0[\u,\u] = (g^0 b(\D) \u, b(\D) \u)_{L_2(\O)},\quad \u \in H^p_\perp (\O;\C^n).
$$
By \eqref{22.4} and \eqref{8.1}, we have
\begin{equation}
\label{8.2}
\|g^{-1}\|^{-1}_{L_\infty} (\wt{k}_1)^{-1} \| \u \|^2_{H^p(\O)} \le b_{N,\eps}[\u,\u] \leqslant \wt{\mathfrak c}_p \alpha_1 \|g\|_{L_\infty}  \| \D^p \u \|^2_{L_2(\O)}, \quad \u \in H^p_\perp(\O;\C^n).
\end{equation}
Similarly,
\begin{equation}
\label{8.3}
\|g^{-1}\|^{-1}_{L_\infty} (\wt{k}_1)^{-1} \| \u \|^2_{H^p(\O)} \le b_{N}^0[\u,\u] \leqslant \wt{\mathfrak c}_p \alpha_1 \|g\|_{L_\infty}  \| \D^p \u \|^2_{L_2(\O)}, \quad \u \in H^p_\perp(\O;\C^n).
\end{equation}
By $\mathcal P$ (respectively, ${\mathcal P}_Z$) we denote the orthogonal projection of $L_2(\O;\C^n)$ onto ${\mathcal H}(\O)$
(respectively, onto $Z$). Then ${\mathcal P} =I - {\mathcal P}_Z$.

Suppose that $\zeta \in \C \setminus [c_\flat, \infty)$, where $c_\flat>0$ is a common lower bound of the operators
$B_{N,\eps}$ and $B_N^0$. In other words, $0< c_\flat \le \min \{ \lambda_{2,\eps}(N), \lambda_2^0(N)\}$, where
$\lambda_{2,\eps}(N)$ (respectively, $\lambda_2^0(N)$) is the first nonzero eigenvalue of the operator $A_{N,\eps}$
(respectively, $A_N^0$). If the eigenvalues are enumerated in the nondecreasing order counting multiplicities,
these eigenvalues have number $q+1$.

\begin{remark}
\emph{1)}
By \eqref{8.2} and \eqref{8.3}, one can take $c_\flat$ equal to $\|g^{-1}\|^{-1}_{L_\infty} (\wt{k}_1)^{-1}$.

\emph{2)} Let $\delta >0$ be arbitrarily small. If $\eps$ is sufficiently small, we can choose $c_\flat = \lambda_2^0(N) - \delta$.

\emph{3)} It is easy to give the upper bound for $c_\flat${\rm :} from \eqref{8.2}, \eqref{8.3}, and the variational pronciple it follows that
$c_\flat \le \wt{\mathfrak c}_p \alpha_1 \|g\|_{L_\infty} \mu^0_{q+1}(N)$, where $\mu^0_{q+1}(N)$ is the $(q+1)$th eigenvalue of the operator
$(-1)^p \wt{\Delta}_p = (-1)^{p} \sum_{|\alpha|=p} \partial^{2\alpha}$  in $L_2(\O;\C^n)$ with the Neumann conditions.
Thus, $c_\flat$ does not exceed a number depending only on $d$, $p$, $n$, $q$, $\alpha_1$, $\|g\|_{L_\infty}$, and the domain~$\O$.
\end{remark}

Let ${\boldsymbol \varphi}_\eps := (B_{N,\eps} - \zeta I)^{-1} \FF$, where $\FF \in {\mathcal H}(\O)$,
and let ${\boldsymbol \varphi}_0 := (B_{N}^0 - \zeta I)^{-1} \FF$. Denote
\begin{equation}
\label{8.4}
{\mathcal K}_{N} (\zeta;\eps) := R_\O [\Lambda^\eps] S_\eps b(\D) P_\O  (B_{N}^0 - \zeta I)^{-1},
\end{equation}
and put $\wt{\boldsymbol \varphi}_0 := P_\O {\boldsymbol \varphi}_0$,
\begin{equation}
\label{8.5}
{\boldsymbol \psi}_\eps : = {\boldsymbol \varphi}_0 + \eps^p \Lambda^\eps S_\eps b(\D) \wt{\boldsymbol \varphi}_0 =
(B_{N}^0 - \zeta I)^{-1} \FF + \eps^p {\mathcal K}_{N} (\zeta;\eps) \FF.
\end{equation}

\begin{lemma}
Let $\zeta \in \C \setminus [c_\flat,\infty)$, where $c_\flat>0$ is a common lower bound of the operators $B_{N,\eps}$ and $B_N^0$.
We put $\zeta - c_\flat = |\zeta - c_\flat| e^{i \vartheta}$ and denote
\begin{equation}
\label{8.6}
\rho_\flat(\zeta)  = \begin{cases}
c(\vartheta)^2 |\zeta - c_\flat|^{-2}, & |\zeta - c_\flat|<1,
\\ c(\vartheta)^2, & |\zeta - c_\flat|\ge 1.
\end{cases}
\end{equation}
Here $c(\vartheta)$ is defined according to \eqref{2.1}.
Then for $\eps>0$ we have
\begin{align}
\label{88.7}
\| (B_{N,\eps} - \zeta I)^{-1} \|_{{\mathcal H}(\O)\to H^p(\O)} \le {\mathfrak C}_7 \rho_\flat(\zeta)^{1/2},
\\
\label{88.8}
\| (B_{N}^0 - \zeta I)^{-1} \|_{{\mathcal H}(\O)\to H^{2p}(\O)} \le {\mathfrak C}_8 \rho_\flat(\zeta)^{1/2}.
\end{align}
The constants ${\mathfrak C}_7$ and ${\mathfrak C}_8$ depend only on $m$, $n$, $d$, $p$, $q$,
$\alpha_0$, $\alpha_1$, $\|g\|_{L_\infty}$, $\|g^{-1}\|_{L_\infty}$, $k_1$, $k_2$, the parameters of the lattice $\Gamma$, and the domain $\O$.
\end{lemma}

\begin{proof}
From \eqref{22.9} (with $\zeta = -1$) it follows that
\begin{equation}
\label{8.6a}
\begin{split}
&\| (B_{N,\eps} - \zeta I)^{-1} \|_{{\mathcal H}(\O)\to H^p(\O)}
\\
&\le
\| (B_{N,\eps} + I)^{-1} \|_{{\mathcal H}(\O)\to H^p(\O)}
\| (B_{N,\eps} + I) (B_{N,\eps} - \zeta I)^{-1} \|_{{\mathcal H}(\O)\to {\mathcal H}(\O)}
\le  {\mathcal C}_0 \sup_{x \ge c_\flat} \frac{x+1}{|x -\zeta|}.
\end{split}
\end{equation}
A calculation shows that
\begin{equation}
\label{8.15}
\sup_{x \ge c_\flat} \frac{x+1}{|x -\zeta|} \le \check{c}_\flat \rho_\flat(\zeta)^{1/2},
\end{equation}
where $\check{c}_\flat := c_\flat +2$. Relations \eqref{8.6a} and \eqref{8.15} imply estimate \eqref{88.7} with the constant
 ${\mathfrak C}_7 := {\mathcal C}_0 \check{c}_\flat$.

Estimate \eqref{88.8} with the constant ${\mathfrak C}_8 := 2 \wh{c} \check{c}_\flat$ follows from \eqref{22.17} (with $\zeta=-1$) and \eqref{8.15}.
\end{proof}

\begin{theorem}\label{th8.3}
Suppose that $\O \subset \R^d$ is a bounded domain of class $C^{2p}$.
Suppose that the number $\eps_1>0$ is subject to Condition \emph{\ref{cond_eps}}. Let $0< \eps \le \eps_1$.
Let $\zeta \in \C \setminus [c_\flat,\infty)$, where $c_\flat>0$ is a common lower bound of the operators $B_{N,\eps}$ and $B_N^0$.
Let ${\boldsymbol \varphi}_\eps := (B_{N,\eps} - \zeta I)^{-1} \FF$ and
${\boldsymbol \varphi}_0 := (B_{N}^0 - \zeta I)^{-1} \FF$, where $\FF \in {\mathcal H}(\O)$.
Let ${\mathcal K}_N(\zeta;\eps)$ be the operator  \eqref{8.4}, and let
${\boldsymbol \psi}_\eps$ be given by \eqref{8.5}. Then we have
\begin{align}
\nonumber
&\| {\boldsymbol \varphi}_\eps - {\boldsymbol \varphi}_0 \|_{L_2(\O)} \le {\mathfrak C}_9 \eps \rho_\flat(\zeta) \| \FF \|_{L_2(\O)},
\\
\label{8.8}
&\| {\boldsymbol \varphi}_\eps - {\boldsymbol \psi}_\eps \|_{H^p(\O)} \le
{\mathfrak C}_{10} \left(\eps^{1/2} \rho_\flat(\zeta)^{1/2} +\eps |\zeta +1| \rho_\flat(\zeta)\right) \| \FF \|_{L_2(\O)}.
\end{align}
In operator terms,
\begin{align}
\label{8.9}
&\| (B_{N,\eps} - \zeta I)^{-1} - (B_{N}^0 - \zeta I)^{-1} \|_{{\mathcal H}(\O) \to {\mathcal H}(\O)} \le {\mathfrak C}_9 \eps \rho_\flat(\zeta),
\\
\label{8.10}
\begin{split}
&\| (B_{N,\eps} - \zeta I)^{-1} - (B_{N}^0 - \zeta I)^{-1} - \eps^p {\mathcal K}_N(\zeta;\eps) \|_{{\mathcal H}(\O) \to H^p(\O)}
\\
&\le
{\mathfrak C}_{10} \left(\eps^{1/2} \rho_\flat(\zeta)^{1/2} +\eps |\zeta +1| \rho_\flat(\zeta)\right).
\end{split}
\end{align}
The flux $g^\eps b(\D) {\boldsymbol \varphi}_\eps$ satisfies
\begin{equation}
\label{8.11}
\| g^\eps b(\D) {\boldsymbol \varphi}_\eps - \wt{g}^\eps S_\eps b(\D) \wt{\boldsymbol \varphi}_0 \|_{L_2(\O)} \le
{\mathfrak C}_{11} \left(\eps^{1/2} \rho_\flat(\zeta)^{1/2} +\eps |\zeta +1| \rho_\flat(\zeta)\right) \| \FF \|_{L_2(\O)}
\end{equation}
for $0< \eps \le \eps_1$.
The constants ${\mathfrak C}_9$, ${\mathfrak C}_{10}$, and ${\mathfrak C}_{11}$ depend only on $m$, $n$, $d$, $p$, $q$,
$\alpha_0$, $\alpha_1$, $\|g\|_{L_\infty}$, $\|g^{-1}\|_{L_\infty}$, $k_1$, $k_2$, the parameters of the lattice $\Gamma$, and the domain $\O$.
\end{theorem}

\begin{proof}
By \eqref{7.9},
\begin{equation}
\label{8.12}
\begin{split}
\| (B_{N,\eps} + I)^{-1} - (B_{N}^0 + I)^{-1} \|_{{\mathcal H}(\O) \to {\mathcal H}(\O)}
= \| \left( (A_{N,\eps} + I)^{-1} - (A_{N}^0 + I)^{-1} \right) {\mathcal P} \|_{L_2(\O) \to L_2(\O)} \le
{\mathcal C}_1' \eps
\end{split}
\end{equation}
for $0< \eps \le \eps_1$.
Similarly to \eqref{4.53}, we have
\begin{equation*}
\begin{split}
& (B_{N,\eps} - \zeta I)^{-1} - (B_{N}^0 - \zeta I)^{-1}
\\
&= (B_{N,\eps} + I) (B_{N,\eps} - \zeta I)^{-1} \left( (B_{N,\eps} + I)^{-1} - (B_{N}^0 + I)^{-1} \right) (B_{N}^0 + I) (B_{N}^0 - \zeta I)^{-1}.
\end{split}
\end{equation*}
Combining this with \eqref{8.12}, we obtain
\begin{equation}
\label{8.14}
\begin{split}
\| (B_{N,\eps} - \zeta I)^{-1} - (B_{N}^0 - \zeta I)^{-1} \|_{{\mathcal H}(\O) \to {\mathcal H}(\O)}
\le {\mathcal C}_1' \eps \sup_{x \ge c_\flat} \frac{(x+1)^2}{|x -\zeta|^2}
\end{split}
\end{equation}
for $0< \eps \le \eps_1$.
Relations \eqref{8.15} and \eqref{8.14} imply the required estimate \eqref{8.9} with the constant ${\mathfrak C}_9 := {\mathcal C}_1' \check{c}_\flat^2$.

Multiplying the operators under the norm sign in \eqref{7.12} by $\mathcal P$ from the right, we obtain
\begin{equation}
\label{8.17}
\|( B_{N,\eps} + I)^{-1} - ( B_{N}^0 + I)^{-1} - \eps^p  {\mathcal K}_N(-1;\eps) \|_{{\mathcal H}(\O) \to H^p(\O)}
\le  2 {\mathcal C}_2' \eps^{1/2}.
\end{equation}

Similarly to \eqref{7.13}, we have
\begin{equation}
\label{8.18}
\begin{aligned}
&(B_{N,\eps} - \zeta I)^{-1} - (B_{N}^0 - \zeta I)^{-1} - \eps^p {\mathcal K}_N(\zeta;\eps)
\\
&=
\left( (B_{N,\eps} + I)^{-1} - (B_{N}^0 + I)^{-1} - \eps^p {\mathcal K}_N(-1;\eps)\right) (B_{N}^0 +I) (B_{N}^0 - \zeta I)^{-1}
\\
&+(\zeta +1) (B_{N,\eps} - \zeta I)^{-1} \left( (B_{N,\eps} + I)^{-1} - (B_{N}^0 + I)^{-1}\right)
(B_{N}^0 +I) (B_{N}^0 - \zeta I)^{-1}.
\end{aligned}
\end{equation}
Together with \eqref{8.15}, this yields
\begin{equation}
\label{8.19a}
\begin{aligned}
&\|(B_{N,\eps} - \zeta I)^{-1} - (B_{N}^0 - \zeta I)^{-1} - \eps^p {\mathcal K}_N(\zeta;\eps)\|_{{\mathcal H}(\O) \to H^p(\O)}
\\
&\le \check{c}_\flat \rho_\flat(\zeta)^{1/2} \| (B_{N,\eps} + I)^{-1} - (B_{N}^0 + I)^{-1} - \eps^p {\mathcal K}_N(-1;\eps)\|_{{\mathcal H}(\O) \to H^p(\O)}
\\
&+ |\zeta +1|  \check{c}_\flat \rho_\flat(\zeta)^{1/2} \|(B_{N,\eps} - \zeta I)^{-1}\|_{{\mathcal H}(\O) \to H^p(\O)}
\| (B_{N,\eps} + I)^{-1} - (B_{N}^0 + I)^{-1} \|_{{\mathcal H}(\O) \to {\mathcal H}(\O)}.
\end{aligned}
\end{equation}
Combining this with \eqref{88.7}, \eqref{8.12}, and  \eqref{8.17}, we arrive at
\begin{equation}
\label{8.20}
\begin{aligned}
\|(B_{N,\eps} - \zeta I)^{-1} - (B_{N}^0 - \zeta I)^{-1} - \eps^p {\mathcal K}_N(\zeta;\eps)\|_{{\mathcal H}(\O) \to H^p(\O)}
\le {\mathfrak C}_{12} \eps^{1/2} \rho_\flat(\zeta)^{1/2}
+ {\mathfrak C}_{13} \eps |\zeta +1| \rho_\flat(\zeta),
\end{aligned}
\end{equation}
where ${\mathfrak C}_{12} : = 2 \check{c}_\flat {\mathcal C}_{2}'$ and
${\mathfrak C}_{13} : = {\mathfrak C}_7  {\mathcal C}_1' \check{c}_\flat$.
As a result, inequality \eqref{8.20} implies the desired estimate \eqref{8.10} with the constant ${\mathfrak C}_{10} : =
\max\{ {\mathfrak C}_{12} ; {\mathfrak C}_{13}\}$.

It remains to check \eqref{8.11}. By \eqref{1.3}, \eqref{1.5}, and \eqref{8.8},
\begin{equation}
\label{8.23}
\| g^\eps b(\D) {\boldsymbol \varphi}_\eps - {g}^\eps b(\D) {\boldsymbol \psi}_\eps \|_{L_2(\O)} \le c_6(d,p)
\|g\|_{L_\infty} \alpha_1^{1/2}{\mathfrak C}_{10}  \left(\eps^{1/2} \rho_\flat(\zeta)^{1/2} + \eps |\zeta +1| \rho_\flat(\zeta)\right)  \| \FF \|_{L_2(\O)}.
\end{equation}
Similarly to \eqref{4.60}--\eqref{4.62} (cf. \eqref{7.17}), we have
\begin{equation}
\label{8.24}
\begin{split}
\| g^\eps b(\D) {\boldsymbol \psi}_\eps - \wt{g}^\eps S_\eps b(\D) \wt{\boldsymbol \varphi}_0 \|_{L_2(\O)}
\le {\mathfrak C}_4 \eps \| \wt{\boldsymbol \varphi}_0 \|_{H^{2p}(\R^d)}.
\end{split}
\end{equation}
From \eqref{3.4} and \eqref{88.8} it follows that
\begin{equation}
\label{8.25}
\| \wt{\boldsymbol \varphi}_0 \|_{H^{2p}(\O)} \le C_\O^{(2p)} {\mathfrak C}_8  \rho_\flat(\zeta)^{1/2} \|\FF\|_{L_2(\O)}.
\end{equation}

As a result, relations \eqref{8.23}--\eqref{8.25} imply the required estimate \eqref{8.11} with the constant
${\mathfrak C}_{11} := c_6(d,p) \|g\|_{L_\infty} \alpha_1^{1/2} {\mathfrak C}_{10} + {\mathfrak C}_4 {\mathfrak C}_8 C_\O^{(2p)}$.
\end{proof}

\subsection{Removal of the smoothing operator}

Consider the case where Condition \ref{cond_Lambda} is satisfied.
We introduce a corrector
\begin{equation}
\label{8.31}
{\mathcal K}^0_{N} (\zeta;\eps) := [\Lambda^\eps] b(\D)  (B_{N}^0 - \zeta I)^{-1},
\end{equation}
and denote
\begin{equation}
\label{8.32}
{\boldsymbol \psi}^0_\eps := {\boldsymbol \varphi}_0 + \eps^p \Lambda^\eps b(\D) {\boldsymbol \varphi}_0 =
(B_{N}^0 - \zeta I)^{-1} \FF + \eps^p {\mathcal K}_{N}^0 (\zeta;\eps) \FF.
\end{equation}

\begin{theorem}\label{th8.4}
Suppose that the assumptions of Theorem~\emph{\ref{th8.3}} and Condition \emph{\ref{cond_Lambda}} are satisfied.
Let ${\mathcal K}_N^0(\zeta;\eps)$~be the operator~\eqref{8.31}, and let ${\boldsymbol \psi}^0_\eps$ be given by~\eqref{8.32}.
Then for $\zeta \in \C \setminus [c_\flat,\infty)$ and $0< \eps \leqslant \eps_1$ we have
\begin{align}
\label{8.33}
\| {\boldsymbol \varphi}_\eps  -   {\boldsymbol \psi}^0_\eps \|_{H^p(\O)} \leqslant
 \wt{\mathfrak C}_{10} \left( \eps^{1/2} \rho_\flat(\zeta)^{1/2} + \eps |\zeta +1| \rho_\flat(\zeta) \right) \|\FF\|_{L_2(\O)},
\end{align}
or, in operator terms,
\begin{equation}
\begin{split}
\label{8.34}
\|( B_{N,\eps} - \zeta I)^{-1} - ( B_{N}^0 - \zeta I)^{-1} - \eps^p {\mathcal K}_N^0(\zeta;\eps) \|_{{\mathcal H}(\O) \to H^p(\O)}
\leqslant
 \wt{\mathfrak C}_{10} \left( \eps^{1/2} \rho_\flat(\zeta)^{1/2} + \eps |\zeta +1| \rho_\flat(\zeta) \right).
\end{split}
\end{equation}
The flux  $g^\eps b(\D) {\boldsymbol \varphi}_\eps$ satisfies
\begin{equation}
\label{8.35}
\|  g^\eps b(\D) {\boldsymbol \varphi}_\eps - \wt{g}^\eps b(\D) {\boldsymbol \varphi}_0 \|_{L_2(\O)} \leqslant
  \wt{\mathfrak C}_{11} \left( \eps^{1/2} \rho_\flat(\zeta)^{1/2} + \eps |\zeta +1| \rho_\flat(\zeta) \right) \|\FF\|_{L_2(\O)}
\end{equation}
for $0< \eps \leqslant \eps_1$.
The constants $\wt{\mathfrak C}_{10}$ and $\wt{\mathfrak C}_{11}$ depend only on $m$, $n$, $d$, $p$, $q$,
$\alpha_0$, $\alpha_1$, $\|g\|_{L_\infty}$, $\|g^{-1}\|_{L_\infty}$, $k_1$, $k_2$, the parameters of the lattice $\Gamma$, the domain $\O$,
and also on $\|\Lambda\|_{L_\infty}$ and $M_\Lambda$.
\end{theorem}

\begin{proof}
Similarly to \eqref{6.9}--\eqref{6.13}, we have
\begin{equation}
\label{8.36}
\|  {\boldsymbol \psi}_\eps - {\boldsymbol \psi}^0_\eps  \|_{H^p(\O)} \leqslant
 {\mathfrak C}_5 \eps \|   \wt{\boldsymbol \varphi}_0  \|_{H^{2p}(\R^d)},
\end{equation}
cf. \eqref{7.22}.
From \eqref{8.25} and \eqref{8.36} it follows that
\begin{equation}
\label{8.37}
\|  {\boldsymbol \psi}_\eps - {\boldsymbol \psi}^0_\eps  \|_{H^p(\O)} \leqslant
C_\O^{(2p)}  {\mathfrak C}_5 {\mathfrak C}_8 \eps \rho_\flat(\zeta)^{1/2} \| \FF \|_{L_2(\O)}.
\end{equation}
Estimates \eqref{8.8} and \eqref{8.37} imply \eqref{8.33} with the constant
$\wt{\mathfrak C}_{10} := {\mathfrak C}_{10} + C_\O^{(2p)} {\mathfrak C}_5 {\mathfrak C}_8$.

Now, we check \eqref{8.35}. Similarly to \eqref{6.12}, \eqref{6.13}, and \eqref{6.14a}, we have
\begin{equation*}
\| \wt{g}^\eps b(\D) {\boldsymbol \varphi}_0 - \wt{g}^\eps S_\eps b(\D) \wt{\boldsymbol \varphi}_0  \|_{L_2(\O)} \leqslant
{\mathfrak C}_6  \eps \| \wt{\boldsymbol \varphi}_0 \|_{H^{2p}(\R^d)},
\end{equation*}
cf. \eqref{7.24}.
Together with \eqref{8.25}, this yields
\begin{equation}
\label{8.39}
\| \wt{g}^\eps b(\D) {\boldsymbol \varphi}_0 - \wt{g}^\eps S_\eps b(\D) \wt{\boldsymbol \varphi}_0  \|_{L_2(\O)} \leqslant
C_\O^{(2p)} {\mathfrak C}_6 {\mathfrak C}_8  \eps  \rho_\flat(\zeta)^{1/2} \| \FF \|_{L_2(\O)}.
\end{equation}
As a result, relations \eqref{8.11} and \eqref{8.39} imply \eqref{8.35} with the constant
$\wt{\mathfrak C}_{11} := {\mathfrak C}_{11} +  C_\O^{(2p)}  {\mathfrak C}_6 {\mathfrak C}_8$.
\end{proof}

\subsection{Special cases}
The case where the corrector is equal to zero is distinguished by the following statement, which directly follows from Theorem~\ref{th8.3}.

\begin{proposition}
\label{prop8.5}
Suppose that the assumptions of Theorem~\emph{\ref{th8.3}} are satisfied. Let $g^0 = \overline{g}$, i.~e., relations \eqref{1.18} hold.
Then for $\zeta \in \C \setminus [c_\flat, \infty)$ and $0< \eps \le \eps_1$ we have
\begin{equation*}
\begin{split}
\| {\boldsymbol \varphi}_\eps  -   {\boldsymbol \varphi}_0 \|_{H^p(\O)} \leqslant
 {\mathfrak C}_{10} \left( \eps^{1/2} \rho_\flat(\zeta)^{1/2} + \eps |\zeta +1| \rho_\flat(\zeta) \right) \|\FF\|_{L_2(\O)}.
\end{split}
\end{equation*}
\end{proposition}

Now, consider the case where $g^0 = \underline{g}$.

\begin{proposition}
\label{prop8.6}
Suppose that the assumptions of Theorem~\emph{\ref{th8.3}} are satisfied. Let $g^0 = \underline{g}$, i.~e., representations \eqref{1.19} hold.
Let ${\boldsymbol \psi}^0_\eps$ be given by \eqref{8.32}.
Then for $\zeta \in \C \setminus [c_\flat, \infty)$ and $0< \eps \le \eps_1$ we have
\begin{align}
\label{8.40}
&\| {\boldsymbol \varphi}_\eps  -   {\boldsymbol \psi}_\eps^0 \|_{H^p(\O)} \leqslant
\widehat{\mathfrak C}_{10} \eps \left( \rho_\flat(\zeta)^{1/2} + |\zeta +1| \rho_\flat(\zeta) \right) \|\FF\|_{L_2(\O)},
\\
\label{8.41}
&\| g^\eps b(\D) {\boldsymbol \varphi}_\eps - g^0 b(\D) {\boldsymbol \varphi}_0 \|_{L_2(\O)}
 \le \widehat{\mathfrak C}_{11} \eps \left( \rho_\flat(\zeta)^{1/2} + |\zeta +1| \rho_\flat(\zeta) \right) \|\FF\|_{L_2(\O)}.
\end{align}
The constants $\widehat{\mathfrak C}_{10}$ and $\widehat{\mathfrak C}_{11}$ depend only on
$m$, $n$, $d$, $p$, $\alpha_0$, $\alpha_1$, $\|g\|_{L_\infty}$, $\|g^{-1}\|_{L_\infty}$, $k_1$, $k_2$, the parameters of the lattice $\Gamma$, and the domain $\O$.
\end{proposition}

\begin{proof}
Under condition $g^0 = \underline{g}$, inequality \eqref{6.23} at the point $\zeta=-1$ holds, whence
\begin{equation}
\label{8.42}
\|( B_{N,\eps} + I)^{-1} - ( B_{N}^0 + I)^{-1} - \eps^p {\mathcal K}_N(-1;\eps) \|_{{\mathcal H}(\O) \to H^p(\O)}
\leqslant 2 ({\mathcal C}_5 + {\mathcal C}_{15}') \eps, \quad 0 < \eps \le \eps_1.
\end{equation}
Applying identity \eqref{8.18} and \eqref{8.42}, similarly to \eqref{8.19a}, we obtain
\begin{equation}
\label{8.34a}
\begin{split}
\| (B_{N,\eps} -\zeta  I)^{-1} - ( B_{N}^0 - \zeta I)^{-1} - \eps^p {\mathcal K}_N(\zeta; \eps) \|_{{\mathcal H}(\O) \to H^p(\O)}
\le \widehat{\mathfrak C}_{10}'  \eps \left( \rho_\flat(\zeta)^{1/2} + |\zeta +1| \rho_\flat(\zeta) \right),
\end{split}
\end{equation}
where $\widehat{\mathfrak C}_{10}' : = \max \{ 2 ({\mathcal C}_5 + {\mathcal C}_{15}') \check{c}_\flat, {\mathfrak C}_{13} \}$.

By Proposition~\ref{prop_MLambda}($2^\circ$), Condition \ref{cond_Lambda} is now satisfied; moreover,
$\|\Lambda\|_{L_\infty}$ and $M_\Lambda$ are controlled in terms of $m$, $n$, $d$, $p$,
$\alpha_0$, $\alpha_1$, $\|g\|_{L_\infty}$, $\|g^{-1}\|_{L_\infty}$, and the parameters of the lattice $\Gamma$.
Then inequality \eqref{8.37} holds.
Relations \eqref{8.37} and \eqref{8.34a} imply \eqref{8.40} with the constant $\widehat{\mathfrak C}_{10} : = \widehat{\mathfrak C}_{10}' +
C_\O^{(2p)}  {\mathfrak C}_5 {\mathfrak C}_8$.

Inequality \eqref{8.41} is deduced from \eqref{8.24}, \eqref{8.25},  \eqref{8.39}, \eqref{8.34a}, and the identity
$\widetilde{g}(\x) = g^0$.
\end{proof}

\subsection{Application of the results for $B_{N,\eps}$ to the operator $A_{N,\eps}$}

Theorem~\ref{th8.3} allows us to obtain approximation for the resolvent $(A_{N,\eps} - \zeta I)^{-1}$ at a regular
point $\zeta \in \C \setminus [c_\flat, \infty)$, $\zeta \ne 0$.

\begin{theorem}\label{th8.7}
Suppose that $\O \subset \R^d$ is a bounded domain of class $C^{2p}$.
Suppose that the number $\eps_1>0$ is subject to Condition \emph{\ref{cond_eps}}.
Let $0< c_\flat \le \min \{ \lambda_{2,\eps}(N), \lambda_2^0(N)\}$, where
$\lambda_{2,\eps}(N)$ {\rm (}respectively, $\lambda_2^0(N)${\rm )} is the first nonzero eigenvalue of the operator
$A_{N,\eps}$ {\rm (}respectively, $A_N^0${\rm )}.
Suppose that $\zeta \in \C \setminus [c_\flat,\infty)$, $\zeta \ne 0$.
Let $\u_\eps = (A_{N,\eps} - \zeta I)^{-1} \FF$ and
$\u_0 = (A_{N}^0 - \zeta I)^{-1} \FF$, where $\FF \in L_2(\O; \C^n)$. Then for $0< \eps \le \eps_1$ we have
\begin{equation*}
\| \u_\eps - \u_0 \|_{L_2(\O)} \le {\mathfrak C}_9 \eps \rho_\flat(\zeta) \| \FF \|_{L_2(\O)},
\end{equation*}
where $\rho_\flat(\zeta)$ is defined by \eqref{8.6}. In operator terms,
\begin{equation}
\label{88.40}
\| (A_{N,\eps} - \zeta I)^{-1} - (A_{N}^0 - \zeta I)^{-1} \|_{L_2(\O) \to L_2(\O)} \le {\mathfrak C}_9 \eps \rho_\flat(\zeta).
\end{equation}
Denote $\widehat{\mathbf v}_\eps := \u_0 + \eps^p \Lambda^\eps S_\eps b(\D) \widehat{\u}_0$, where $\widehat{\u}_0 := P_\O ( A_N^0 - \zeta I)^{-1} {\mathcal P} \FF$.
Then for $0< \eps \le \eps_1$ we have
\begin{equation*}
\| \u_\eps - \widehat{\mathbf v}_\eps  \|_{H^p(\O)} \le
{\mathfrak C}_{10} \left(\eps^{1/2} \rho_\flat(\zeta)^{1/2} +\eps |\zeta +1| \rho_\flat(\zeta)\right) \| \FF \|_{L_2(\O)}.
\end{equation*}
In operator terms,
\begin{align}
\label{88.42}
\begin{split}
&\| (A_{N,\eps} - \zeta I)^{-1} - (A_{N}^0 - \zeta I)^{-1} - \eps^p \widehat{K}_N(\zeta;\eps) \|_{L_2(\O) \to H^p(\O)}
\\
&\le
{\mathfrak C}_{10} \left(\eps^{1/2} \rho_\flat(\zeta)^{1/2} +\eps |\zeta +1| \rho_\flat(\zeta)\right),
\end{split}
\end{align}
where $\widehat{K}_N(\zeta;\eps) = R_\O [\Lambda^\eps] S_\eps b(\D) P_\O (A_{N}^0 - \zeta I)^{-1} {\mathcal P}$.
The flux $\p_\eps = g^\eps b(\D) \u_\eps$ satisfies
\begin{equation}
\label{88.43}
\| \p_\eps - \wt{g}^\eps S_\eps b(\D) \widehat{\u}_0 \|_{L_2(\O)} \le
{\mathfrak C}_{11} \left(\eps^{1/2} \rho_\flat(\zeta)^{1/2} +\eps |\zeta +1| \rho_\flat(\zeta)\right) \| \FF \|_{L_2(\O)}
\end{equation}
for $0< \eps \le \eps_1$.
The constants ${\mathfrak C}_9$, ${\mathfrak C}_{10}$, and ${\mathfrak C}_{11}$ are the same as in Theorem~\emph{\ref{th8.3}}.
\end{theorem}

\begin{proof}
Note that for $\zeta \in \C \setminus [c_\flat,\infty)$, $\zeta \ne 0$, we have
$$
(A_{N,\eps} - \zeta I)^{-1} {\mathcal P} = (B_{N,\eps} - \zeta I)^{-1} {\mathcal P},
\quad
(A_{N,\eps} - \zeta I)^{-1} {\mathcal P}_Z = - \zeta^{-1} {\mathcal P}_Z.
$$
Similarly,
$$
(A_{N}^0 - \zeta I)^{-1} {\mathcal P} = (B_{N}^0 - \zeta I)^{-1} {\mathcal P},
\quad
(A_{N}^0 - \zeta I)^{-1} {\mathcal P}_Z = - \zeta^{-1} {\mathcal P}_Z.
$$
Since ${\mathcal P} + {\mathcal P}_Z =I$, this implies
\begin{equation}
\label{88.44}
(A_{N,\eps} - \zeta I)^{-1} - (A_{N}^0 - \zeta I)^{-1} =
\left((B_{N,\eps} - \zeta I)^{-1}  - (B_{N}^0 - \zeta I)^{-1}\right) {\mathcal P}.
\end{equation}
Estimate \eqref{88.40} follows directly from \eqref{8.9} and \eqref{88.44}.

Obviously, we have $\widehat{K}_N(\zeta;\eps) = \mathcal{K}_N(\zeta;\eps) {\mathcal P}$.
Combining this with \eqref{88.44}, we obtain
\begin{equation}
\label{88.45}
(A_{N,\eps} - \zeta I)^{-1} - (A_{N}^0 - \zeta I)^{-1} - \eps^p \widehat{K}_N(\zeta;\eps) =
\left((B_{N,\eps} - \zeta I)^{-1}  - (B_{N}^0 - \zeta I)^{-1} - \eps^p \mathcal{K}_N(\zeta;\eps) \right) {\mathcal P}.
\end{equation}
Relations \eqref{8.10} and \eqref{88.45} imply \eqref{88.42}.

Next, since $b(\D) (A_{N,\eps} - \zeta I)^{-1} {\mathcal P}_Z =0$, then
\begin{equation}
\label{88.46}
\begin{split}
&g^\eps b(\D)(A_{N,\eps} - \zeta I)^{-1} - \wt{g}^\eps S_\eps b(\D) P_\O (A_{N}^0 - \zeta I)^{-1}{\mathcal P}
\\
& =
\left(g^\eps b(\D)(B_{N,\eps} - \zeta I)^{-1} - \wt{g}^\eps S_\eps b(\D) P_\O (B_{N}^0 - \zeta I)^{-1}\right) {\mathcal P}.
\end{split}
\end{equation}
In operator terms, inequality \eqref{8.11} means that
\begin{equation*}
\begin{split}
\|& g^\eps b(\D)(B_{N,\eps} - \zeta I)^{-1} - \wt{g}^\eps S_\eps b(\D) P_\O (B_{N}^0 - \zeta I)^{-1}\|_{{\mathcal H}(\O)\to L_2(\O)}
\\
&\leqslant {\mathfrak C}_{11} \left(\eps^{1/2} \rho_\flat(\zeta)^{1/2} +\eps |\zeta +1| \rho_\flat(\zeta)\right)
\end{split}
\end{equation*}
for $0< \eps \le \eps_1$.
 Together with \eqref{88.46}, this implies \eqref{88.43}.
\end{proof}

\begin{remark}
Estimates \eqref{88.40}--\eqref{88.43} are useful for bounded values of $|\zeta|$ and small $\eps \rho_\flat(\zeta)$.
In this case, the value $\eps^{1/2} \rho_\flat(\zeta)^{1/2} +\eps |\zeta +1| \rho_\flat(\zeta)$ is majorated by
$C \eps^{1/2} \rho_\flat(\zeta)^{1/2}$. For large $|\zeta|$ estimates of Theorems~\emph{\ref{th3.1}}
and~\emph{\ref{th3.2}} may be preferable.
\end{remark}

The following result is deduced from Theorem~\ref{th8.4}.

\begin{theorem}\label{th8.8}
Suppose that the assumptions of Theorem~\emph{\ref{th8.7}} and Condition~\emph{\ref{cond_Lambda}} are satisfied.
Let $K_N^0(\zeta;\eps)$ be the operator~\eqref{6.1}, and let ${\mathbf v}_\eps^0$ be given by~\eqref{6.2}.
Then for $\zeta \in \C \setminus [c_\flat,\infty)$, $\zeta \ne 0$, and $0< \eps \le \eps_1$ we have
\begin{equation*}
\| \u_\eps - {\mathbf v}^0_\eps \|_{H^p(\O)} \le \wt{\mathfrak C}_{10}
\left(\eps^{1/2} \rho_\flat(\zeta)^{1/2} +\eps |\zeta +1| \rho_\flat(\zeta)\right) \| \FF \|_{L_2(\O)},
\end{equation*}
or, in operator terms,
\begin{align}
\label{88.48}
\begin{split}
&\| (A_{N,\eps} - \zeta I)^{-1} - (A_{N}^0 - \zeta I)^{-1} - \eps^p {K}^0_N(\zeta;\eps) \|_{L_2(\O) \to H^p(\O)}
\\
&\le
\wt{\mathfrak C}_{10} \left(\eps^{1/2} \rho_\flat(\zeta)^{1/2} +\eps |\zeta +1| \rho_\flat(\zeta)\right).
\end{split}
\end{align}
The flux $\p_\eps = g^\eps b(\D) \u_\eps$ satisfies
\begin{equation}
\label{88.49}
\| \p_\eps - \wt{g}^\eps b(\D) {\u}_0 \|_{L_2(\O)} \le
\wt{\mathfrak C}_{11} \left(\eps^{1/2} \rho_\flat(\zeta)^{1/2} +\eps |\zeta +1| \rho_\flat(\zeta)\right) \| \FF \|_{L_2(\O)}
\end{equation}
for $0< \eps \le \eps_1$.
The constants $\wt{\mathfrak C}_{10}$ and $\wt{\mathfrak C}_{11}$ are the same as in Theorem~\emph{\ref{th8.4}}.
\end{theorem}

\begin{proof}
By \eqref{6.1}, \eqref{8.31}, and the identity $b(\D) {\mathcal P}_Z =0$, we have
 $K_N^0(\zeta;\eps) = {\mathcal K}_N^0(\zeta;\eps) {\mathcal P}$. Together with \eqref{88.44}, this yields
\begin{equation}
\label{88.50}
(A_{N,\eps} - \zeta I)^{-1} - (A_{N}^0 - \zeta I)^{-1} - \eps^p {K}_N^0(\zeta;\eps) =
\left((B_{N,\eps} - \zeta I)^{-1}  - (B_{N}^0 - \zeta I)^{-1} - \eps^p \mathcal{K}_N^0(\zeta;\eps) \right) {\mathcal P}.
\end{equation}
Relations \eqref{8.34} and \eqref{88.50} directly imply \eqref{88.48}.

Next, since $b(\D) {\mathcal P}_Z =0$, then
\begin{equation}
\label{88.51}
\begin{split}
&g^\eps b(\D)(A_{N,\eps} - \zeta I)^{-1} - \wt{g}^\eps b(\D) (A_{N}^0 - \zeta I)^{-1}
\\
& =
\left(g^\eps b(\D)(B_{N,\eps} - \zeta I)^{-1} - \wt{g}^\eps b(\D) (B_{N}^0 - \zeta I)^{-1}\right) {\mathcal P}.
\end{split}
\end{equation}
Relations \eqref{8.35} and \eqref{88.51} imply \eqref{88.49}.
\end{proof}

The case where the corrector is equal to zero is distinguished by the next statement
which follows from Theorem~\ref{th8.7}.

\begin{proposition}
\label{prop8.9}
Suppose that the assumptions of Theorem~\emph{\ref{th8.7}} are satisfied. Let $g^0 = \overline{g}$, i.~e., relations \eqref{1.18} hold.
Then for $\zeta \in \C \setminus [c_\flat, \infty)$ and $0< \eps \le \eps_1$ we have
\begin{equation*}
\begin{split}
\| {\u}_\eps  -   {\u}_0 \|_{H^p(\O)} \leqslant
 {\mathfrak C}_{10} \left( \eps^{1/2} \rho_\flat(\zeta)^{1/2} + \eps |\zeta +1| \rho_\flat(\zeta) \right) \|\FF\|_{L_2(\O)}.
\end{split}
\end{equation*}
\end{proposition}

The following statement is deduced from Proposition \ref{prop8.6} with the help of identities \eqref{88.50} and \eqref{88.51}.

\begin{proposition}
\label{prop8.10}
Suppose that the assumptions of Theorem~\emph{\ref{th8.7}} are satisfied. Let $g^0 = \underline{g}$, i.~e., representations \eqref{1.19} hold.
Let ${\mathbf v}^0_\eps$ be given by \eqref{6.2}.
Then for $\zeta \in \C \setminus [c_\flat, \infty)$, $\zeta \ne 0$, and $0< \eps \le \eps_1$ we have
\begin{align*}
&\| {\u}_\eps  -   {\mathbf v}_\eps^0 \|_{H^p(\O)} \leqslant
\widehat{\mathfrak C}_{10} \eps \left( \rho_\flat(\zeta)^{1/2} + |\zeta +1| \rho_\flat(\zeta) \right) \|\FF\|_{L_2(\O)},
\\
&\| \p_\eps - g^0 b(\D) {\u}_0 \|_{L_2(\O)}
 \le \widehat{\mathfrak C}_{11} \eps \left( \rho_\flat(\zeta)^{1/2} + |\zeta +1| \rho_\flat(\zeta) \right) \|\FF\|_{L_2(\O)}.
\end{align*}
The constants $\widehat{\mathfrak C}_{10}$ and $\widehat{\mathfrak C}_{11}$ are the same as in Proposition~\emph{\ref{prop8.6}}.
\end{proposition}

\end{document}